\documentclass[a4paper,12pt,reqno,natbib]{amsart}

\usepackage{amsmath}
\usepackage{amssymb}
\usepackage{amsthm}
\usepackage{dsfont}
\usepackage{stmaryrd}
\usepackage{float}
\usepackage{color}
\usepackage{tabu}
\usepackage{graphics}
\usepackage{graphicx}
 \usepackage{subfigure}
\usepackage{booktabs}
\usepackage{multirow}
\usepackage{array}
\usepackage[ruled]{algorithm2e}
\usepackage{amsmath}
\usepackage{tikz}
\usetikzlibrary{shapes}
\usetikzlibrary{spy}

\newtheorem{theorem}{Theorem}
\newtheorem{lemma}[theorem]{Lemma}

\newtheorem{definition}[theorem]{Definition}
\newtheorem{exm}{Example}

\newtheorem{remark}[theorem]{Remark}

\newcommand{\dx} {\displaystyle\mathrm{d}x}
\newcommand{\dt} {\displaystyle\mathrm{d}t}
\newcommand{\dss} {\displaystyle\mathrm{d}s}

\def\ds{\displaystyle}

\def\K{\mathcal K}

\begin{document}

\title{Reconstruction of a potential parameter in subdiffusion via a Kohn--Vogelius type functional: Theory and computation}

\author{Hamza Kahlaoui}
\address[H. Kahlaoui]{Université de Tunis el Manar, école nationale d'ingénieurs de Tunis, 1002 Tunis, Tunisie}
\email{hamzakahlaoui9@gmail.com}

\author{Mourad Hrizi}
\address[M. Hrizi]{ Monastir University, Department of Mathematics, Faculty of Sciences, Avenue de l'Environnement 5000, Monastir, Tunisia
}
\email{mourad-hrizi@hotmail.fr}

\author{Abdessamad Oulmelk}
\address[A. Oulmelk]{Laboratory of Mathematics and Applications, FST Tanger, Abdelmalek Essaadi University, Morocco}
\email{a.oulmelk@uae.ac.ma}

\author{Xiangcheng Zheng}
\address[Xiangcheng Zheng]{School of Mathematics, Shandong University, Jinan 250100, China}
\email{xzheng@sdu.edu.cn}

\author{Mahmoud A. Zaky}
\address[M.
A. Zaky]{Department of Mathematics and Statistics, College of Science, Imam Mohammad Ibn Saud Islamic University (IMSIU), Riyadh, Saudi Arabia}
\email{ma.zaky@yahoo.com}

\author{Ahmed Hendy}
\address[A.S. Hendy]{Department of Computational Mathematics and Computer science, Institute of Natural Sciences and Mathematics, Ural Federal University, 19 Mira St., Yekaterinburg 620002, Russia\\
Department of Mathematics and Computer Science, faculty of science, Benha University, Benha 13511, Egypt}
\email{ahmed.hendy@fsc.bu.edu.eg}
\date{\today}

\begin{abstract}
This work considers the identification of spatial-dependent potential from boundary observations in subdiffusion using a stable and robust recovery method. Specifically, we develop an algorithm to minimize the Kohn-Vogelius cost function, which measures the difference between the solutions of two excitations. The inverse potential problem is recast into an optimization problem, where the objective is to minimize a Kohn-Vogelius-type functional within a set of admissible potentials. We establish the well-posedness of this optimization problem by proving the existence and uniqueness of a minimizer and demonstrating its stability under perturbations in the boundary data. Furthermore, we analyze the Fr\'echet differentiability of the Kohn-Vogelius functional and prove the Lipschitz continuity of its gradient. These theoretical results enable the development of a convergent conjugate gradient algorithm for numerical reconstruction. Numerical examples in one and two dimensions, including scenarios with noisy data, validate the effectiveness and robustness of the proposed method.
\end{abstract}

\maketitle
\pagestyle{myheadings} \thispagestyle{plain} \markboth{}{}

\section{Introduction}

The reconstruction of coefficients in normal or anomalous diffusion processes has a wide history in the literature of inverse problems from theoretical and numerical perspectives. The reconstruction of time dependent parameters in subdiffusion problems has been extensively investigated in literature \cite{hendy2022reconstruction1,hendy2022reconstruction2,jin2024numerical}. 
In particular, some studies have investigated the reconstruction of space-dependent potentials in subdiffusion inferred from lateral Cauchy data \cite{jin2021recovering,jing2021simultaneous,Kain2021MA,rundell2023uniqueness10}. In this paper, we aim to address the inverse problem of reconstructing a space-dependent potential $q^*$  that arises in subdiffusion, utilizing lateral Cauchy data. Our primary objective is to successfully reconstruct $q^*$ using a stable numerical method. In most of recent studies, the numerical algorithms developed for identifying the potential $q^*$ rely on a least-squares approach. However, this paper focuses on the mathematical analysis of an algorithm aimed at minimizing the so-called Kohn-Vogelius cost function with a Tikhonov regularization term that penalizes the potential to be reconstructed.

The Kohn--Vogelius cost functional quantifies the discrepancy between two auxiliary boundary value problems: one driven by boundary measurements and the other by boundary excitations. In the noise-free setting, this functional attains its minimum value of zero, and the corresponding minimizer coincides with the exact solution of the inverse problem. When the data are noisy or incompatible, the inverse problem itself may not admit a solution, yet the minimization problem remains well-posed and solvable. A key advantage of Kohn--Vogelius functionals lies in their robustness: they yield stable reconstructions even for ill-posed problems, owing to their definition over the entire domain, in contrast to classical least-squares functionals that are restricted to the boundary. This robustness has been confirmed in several studies, including those by Afraites et al. \cite{afraites2007detecting,afraites2006conformal} and Chaabane et al. \cite{chaabane2004stable}. For completeness, we briefly recall the historical development of Kohn--Vogelius type functionals. The foundational idea dates back to Wexler, Fry, and Neuman \cite{WexlerAO1985}, who in 1985 proposed a method for detecting unknown impedance from boundary measurements. In 1987, Kohn and Vogelius \cite{KohnCPAM1984} refined this approach by introducing a novel misfit gap-cost functional in an alternating-choice framework. Later, in 1990, Kohn and McKenney \cite{KohnIP1990} implemented this formulation numerically for inverse conductivity problems. Since then, the method has been extended to a wide range of inverse problems~\cite{abda2009topological,abda2019topological,abda2013recovering, aspri2022phase,CanelasIP2015,caubet2019data,eppler2005regularized,hrizi2019reconstruction,kahlaoui2025non,oulmelk2023inverse,prakash2021noniterative,roche1997numerical}, further establishing its versatility and effectiveness.

Recovering a potential in subdiffusion through the inverse problem has been extensively studied. In the one-dimensional setting, Rundell and Yamamoto \cite{rundell2018recovery} established the uniqueness of a non-negative potential $q^* \in C[0,1]$ using Gelfand-Levitan theory. They further proposed a Newton-based iterative reconstruction procedure and provided an empirical analysis of the singular value spectrum of the linearized forward map, underscoring the problem's severely ill-posed nature. This analysis was later extended to a broader class of solutions within an appropriate Sobolev space in time \cite{rundell2023uniqueness10}. Beyond recovering a single parameter, Jing and Yamamoto \cite{jing2021simultaneous} demonstrated the simultaneous identifiability of multiple parameters—including the fractional order, initial value, spatial potential, and Robin coefficients—from lateral Cauchy data at both endpoints in a one-dimensional subdiffusion model with homogeneous boundary conditions and no source term. A significant advancement was made by Jin and Zhou \cite{jin2021recovering}, who achieved the simultaneous reconstruction of both the fractional order and the space-dependent potential $q^* \in \mathcal{Q}$ using only endpoint data. Remarkably, their results hold even with only partial knowledge of the initial data or source, requiring only a mild constraint on the Neumann boundary condition for uniqueness. More recently, the focus has shifted to geometric inverse problems. In \cite{kahlaoui2024reconstruction}, the authors developed a one-shot algorithm based on topological derivatives to identify the location and shape of a potential $q^* = \chi_{\omega^*}$ supported on an unknown domain $\omega^* \subset \Omega$ from partial boundary measurements. For a comprehensive overview of related work, including the recovery of time- or space-dependent potentials from other data types such as terminal observations, over-posed final time data, or average flux measurements, we refer the reader to \cite{jin2020inverse,kaltenbacher2019inverse,kian2019reconstruction,miller2013coefficient,oulmelk2023inverse,zhang2017recovering,zheng2020identification}. On the numerical front, various approaches have been explored. For instance, Oulmelk et al. \cite{oulmelk2023artificial} investigated the use of Artificial Neural Networks to solve the inverse potential problem for the standard diffusion equation, offering insights applicable to the fractional case. Finally, it is worth noting that for the classical parabolic case ($\alpha=1$), this inverse problem has been thoroughly studied by numerous authors employing a wide array of methods \cite{ammari2017use,avdonin1995identification,BerettaSpace2011,cao2018determination,cao2018reconstruction,ImanuvilovNeumann2024,klibanov2013carleman,nguyen2017numerical,trucu2011reconstruction}.

In the present work, building on the Kohn--Vogelius formulation, we recast the inverse reconstruction problem as a regularized optimization problem. Specifically, the unknown potential $q^*$ is identified as the minimizer of a Kohn--Vogelius type cost function augmented with a Tikhonov regularization term. Our main contribution is to establish existence and uniqueness results for the Kohn--Vogelius formulation with Tikhonov regularization in the subdiffusion setting, and to propose a convergent conjugate gradient method (CGM) specifically adapted to this structure. We establish theoretical results showing that this regularized minimization problem admits a unique solution. Moreover, we prove that if a sequence of minimizers corresponds to perturbed observation data converging to the exact (noise-free) data, then the associated sequence of potentials converges to the solution of the considered inverse problem. To compute the minimizer, we derive the Fr\'echet gradient of the Kohn--Vogelius functional explicitly and employ an iterative CGM. We rigorously prove the Lipschitz continuity of the Fr\'echet gradient, which underpins the convergence analysis of the CGM. A distinctive and challenging numerical aspect of our method is the selection of the optimal step size within the CG iteration. Unlike standard implementations that solve a linear equation, our formulation necessitates solving a non-trivial quadratic equation, a complication we address with a carefully designed and efficient numerical procedure. 

For completeness, it is worth noting that in the context of electrical impedance tomography, 
Kohn and McKenney \cite{KohnIP1990} were the first to investigate numerically the performance of 
the classical Kohn--Vogelius formulation (i.e., without a Tikhonov regularization term) 
using a Newton-type algorithm. In its classical form, Newton's method minimizes 
a functional by iteratively approximating it with a quadratic, minimizing that quadratic, 
and updating the approximation. This procedure is efficient as long as the Hessian of the 
objective functional is positive definite. However, the Hessian of the Kohn--Vogelius functional 
in their work is not always positive definite, which can lead to numerical instabilities 
and unreliable descent directions. To address this, Kohn and McKenney introduced a 
\emph{regularized Hessian}, obtained by modifying unstable terms and adding a corrective 
factor of the form $(1+\varepsilon)^{-1}$ (where $\varepsilon>0$), analogous to the Levenberg--Marquardt regularization. 
This adjustment shifts the Hessian's eigenvalues to the positive side, ensuring a well-conditioned 
system and stable minimization. In contrast to the regularized Hessian approach of \cite{KohnIP1990}, the present paper proves, 
under a conditional assumption on the regularization parameter (see Theorem \ref{new-uniqueness-condition} and Remark~\ref{con-sta}), 
the uniqueness of the optimal solution for the considered Kohn--Vogelius-type cost functional. Importantly, this uniqueness and the resulting stability of the optimization problem critically rely on the presence of the Tikhonov regularization term; without it, these foundational properties cannot be guaranteed.

The manuscript is organized as follows. Section \ref{sec:2} presents the formulation of the inverse potential problem, while Section \ref{sec:3} establishes the well-posedness of the direct problem. In Section \ref{sec:4}, building on the Kohn–Vogelius framework, we reformulate the inverse problem as an optimization task, involving the minimization of a Kohn–Vogelius type cost functional over a set of admissible potentials. Section \ref{sec:6} provides key theoretical results on the existence and uniqueness of the solution to the optimization problem. The stability of the optimization procedure is discussed in Section \ref{sec:stability}, and in Section \ref{sec:algorithm}, we introduce a numerical conjugate gradient algorithm based on sensitivity and adjoint formulations. The convergence of the CGM with the Fletcher–Reeves formula \cite{fletcher1964function} is analyzed following the arguments in \cite{dai1996convergence}. Section \ref{sec:numerical-result} details the numerical implementation and illustrates the effectiveness of the proposed method through various experiments. Finally, concluding remarks are offered in Section \ref{sec:9}.

\section{Problem setting}\label{sec:2}

Let $\Omega$ be a bounded domain in $\mathbb{R}^d$ with boundary $\partial\Omega$ of class $\mathcal{C}^1$, where $d \leq 3$ denotes the spatial dimension. Let $T>0$ and $\nu =(\nu_1,\cdots, \nu_d)$ be the unit outward normal vector assigned to the boundary $\partial\Omega$. Moreover, we define $Q_T$ as the space-time domain $\Omega \times (0, T)$ and $\Sigma_T$ as its lateral boundary $\partial\Omega \times (0, T)$. Consider the following initial boundary value problem of subdiffusion \cite{JinRun,Barbook,LiSINUM22,LiMC24}
\begin{eqnarray} \label{direct-problem}
\left\{
  \begin{array}{rcll}
    \partial_{t} ^{\alpha} u-\Delta  u +q^*\,  u&=&0&\hbox{in}\ Q_T, \\
    \displaystyle \partial_\nu u&=&\varphi &\hbox{on}\  \Sigma_T, \\
   u(\cdot,0)&=&0  &\hbox{in}\ \Omega,
  \end{array}
\right.
\end{eqnarray}
where $\partial_\nu u=\nabla u\cdot\nu,$ the given Neumann data $\varphi$ is non-null, and $q^*$ is a space-dependent potential. Moreover, the notation $\partial_{t} ^{\alpha}$ denotes the so-called Caputo fractional derivative of order $0<\alpha<1$ in time and is defined by (see e.g. \cite{JinZhou2023})
\begin{align} \label{caputo}
\partial_t^\alpha  u(t)=
         \displaystyle\frac{1}{\Gamma(1-\alpha)}\int_0^t(t-\tau)^{-\alpha} \frac{\partial  u}{\partial \tau}(\tau)\, \displaystyle\mathrm{d}\tau,\quad0<t\leq T,
\end{align}
where $\Gamma$ denotes the Euler's Gamma function, which is defined on each complex number $z\in \mathbb{C}$ with positive real part (i.e. $\mathfrak{R}\{z\}>0$), by
\begin{equation*}
    \Gamma(z)=\int_0^\infty s^{z-1}e^{-s}\,\dss.
\end{equation*}
In the limiting case $\alpha = 1$, the Caputo derivative \eqref{caputo} coincides with the classical first-order derivative $\partial_t u$. Consequently, the subdiffusion \eqref{direct-problem} reduces to the standard parabolic diffusion equation, which models conventional Brownian motion.

The subdiffusion \eqref{direct-problem} has garnered significant attention across the physical, engineering, and mathematical literature due to its powerful ability to model \emph{anomalous diffusion} processes. These processes are characterized by a mean squared displacement of the form $\langle x^2(t) \rangle \sim t^{\alpha}$, which deviates from the linear growth $\langle x^2(t) \rangle \sim t$ associated with standard Fickian diffusion.
From a microscopic perspective, the model \eqref{direct-problem} can be derived as the macroscopic continuum limit of a Continuous Time Random Walk \cite{metzler2000random}. In such a stochastic process, the waiting time between successive particle jumps is governed by a heavy-tailed probability distribution $\psi(t) \sim t^{-(1+\alpha)}$ whose mean diverges. The probability density function of a particle being at location $x$ at time $t > 0$ is precisely described by the solution $u(x, t)$ to \eqref{direct-problem}, generalizing the classical connection between Brownian motion and the standard diffusion equation \cite{kian2021well}. The inclusion of the potential term $q^*(x)u$ further allows for the modeling of reaction processes or interactions with an external field.

\vspace{0.5cm}

The system \eqref{direct-problem} defines a well-posed direct problem when the potential $q^*$ and the Neumann boundary data $\varphi$ are prescribed. In this work, we address its inverse problem: the reconstruction of the unknown space-dependent potential $q^*$ from additional boundary measurements. To this end, we introduce the following set of admissible potentials:
\begin{align*}
    \Phi_{ad}=\Big\{ q\in H^1(\Omega), \hbox{ such that } 0< c \leq q \leq c^{\prime}\Big\},
\end{align*}
where $c$ and $c^{\prime}$ are given positive constants and $H^1(\Omega)$ denotes the classical Sobolev space \cite{brezis2010functional}. Hence, the inverse problem can be reformulated as follows: 
\begin{equation}\label{inverse}
    \left\{
  \begin{array}{lll}
    \hbox{Given a boundary observation}\; \phi,\hbox{ reconstruct the potential}\; q^*\in\Phi_{ad}\; \hbox{such that} \\

        \qquad\qquad\qquad\qquad\qquad \phi(x,t)= u(x,t),\;\;\;(x,t)\in\Sigma_T.
  \end{array}
\right.
\end{equation}
The boundary observational data $\phi$ for which this inverse problem has a solution $q^*$ is said to be compatible. The well-posedness of the inverse problem \eqref{inverse} has been studied in the literature: uniqueness was established by Kian et al. \cite{Kain2021MA} (see also \cite{rundell2018recovery}), and a local stability result was later obtained in \cite{kahlaoui2025stability}.

To numerically solve this inverse potential problem, a reconstruction method is developed within the Kohn–Vogelius minimization framework. This strategy reformulates the ill-posed inverse problem as a well-posed optimization problem. The next section establishes the analytical foundation of the approach: the necessary function spaces are recalled, and a detailed well-posedness analysis of the direct problem \eqref{direct-problem} is presented, which provides the essential groundwork for the subsequent analysis of the inverse problem.

\section{Well-posedness of the forward problem}\label{sec:3}

We briefly present a preliminary result on the well-posedness and regularity of the solution to the forward problem \eqref{direct-problem}. We first introduce some notation on functional spaces.

\subsection{Notations and auxiliary results} For any $0<\alpha< 1,$ we define $H^{\alpha}(\mathbb{R}),$ the Sobolev space of order $\alpha>0$ over $\mathbb{R},$ by (see, e.g., \cite{lions2012non})
\begin{equation*}
    H^{\alpha}(\mathbb{R})=\left\{ w(t) ; \; w \in L^{2}(\mathbb{R}) ;\; \left(1+|\lambda|^{2}\right)^{\frac{\alpha}{2}} \mathcal{F}(w)(\lambda) \in L^{2}(\mathbb{R})\right\},
\end{equation*}
endow with the norm:
\begin{equation*}
 \|w\|_{H^{\alpha}(\mathbb{R})}=\left\|\left(1+|\lambda|^{2}\right)^{\frac{\alpha}{2}} \mathcal{F}(w)(\lambda)\right\|_{L^{2}(\mathbb{R})},
\end{equation*}
where $\mathcal{F}(w)$ denotes the Fourier transform of $w$. In addition, we set
\begin{equation*}
    H^{\alpha}(0,T)=\left\{w \in L^{2}(0,T); \;  \exists\, \tilde{w} \in H^{\alpha}(\mathbb{R}) \text { such that }\left.\tilde{w}\right|_{(0,T)}=w\right\},
\end{equation*}
where
\begin{equation*}
    \|w\|_{H^{\alpha}(0,T)}=\inf \left\{ \|\tilde{w}\|_{H^{\alpha}(\mathbb{R})}
; \; \left.\tilde{w}\right|_{(0,T)}=w\right\}.
\end{equation*}
We also define the space
$${ }_0 C^{\infty}(0,\,T)=\Big\{w;\;\; w \in C^{\infty}(0,\,T)\;\hbox{ with compact support in}\; (0,T]\Big\}.$$
The space ${ }_0 H^\alpha(0,\,T)$ denotes the closure of ${ }_0 C^{\infty}(0,\,T)$ with respect to norm $\|\cdot\|_{H^{\alpha}(0,T)}.$ For  an abstract Sobolev space $Z$ with norm $\|\cdot\|_Z$, let
\begin{align*} \displaystyle
H^{\alpha}(0,T; Z)&:=\Big\{w ; \;\|w(., t)\|_{Z} \in H^{\alpha}(0,T)\Big\},  \\
{ }_{0} H^{\alpha}(0,T ; Z)&:=\Big\{w ; \;\|w(., t)\|_{Z} \in{ }_{0} H^{\alpha}(0,T)\Big\},
\end{align*}
endowed with the norm:
\begin{equation*}
     \|w\|_{H^\alpha(0,T; Z)} := \|\|w(.,t) \|_{Z}\|_{H^\alpha(0,T)}.
\end{equation*}
We will also need to recall some definitions of fractional derivatives.
\begin{definition}(see \cite{JinZhou2023}).
 Let $w\; :\mathbb{R}_+\longrightarrow\mathbb{R}.$ The left- and right-sided Riemann-Liouville derivatives of order $0<\alpha<1$ of $w$ are defined as
\begin{equation*}
    {}_0D^\alpha _{t} w(t) =\frac{1}{\Gamma(1-\alpha)} \ds \frac{d}{dt}\int_0^t (t-s)^{-\alpha} w(s) \hbox{ds},\quad 0\leq t\leq T, 
\end{equation*}
and 
\begin{equation*}
   {}_tD^{\alpha}_T w(t) =\frac{-1}{\Gamma(1-\alpha)} \ds \frac{d}{dt}\int_t^T (s-t)^{-\alpha} w(s) \hbox{ds},\quad 0\leq t\leq T,
\end{equation*}
respectively.
\end{definition}

We have the following lemma:
\begin{lemma}\label{lemma24}(see \cite[Lemma 2.6]{li2009space}). Let $\alpha\in(0,1)$. Then for any $w\in {}_0H^{\alpha}(0,T)$ and $v\in {}_0H^{\frac{\alpha}{2}}(0,T)$, we have
\begin{align*}
    \int_0^T \Big({}_0D^\alpha _{t} w(t)\Big) v(t)\dt &= \int_0^T \Big({}_0D^{\frac{\alpha}{2}} _{t} w(t)\Big)\Big( {}_tD^{\frac{\alpha}{2}}_T v(t)\Big)\dt.
\end{align*}
\end{lemma}
\begin{lemma} (see \cite[Lemma 2.2]{li2009space}). \label{cosinus}
For $0<\alpha<1$, $w\in {}_0H^\frac{\alpha}{2}(0,T)$, then
\begin{align*}
 \int_0^T \Big({}_0D^{\frac{\alpha}{2}} _{t} w\Big)\Big( {}_tD^{\frac{\alpha}{2}}_T w\Big)\dt \geq \cos\Big(\frac{\pi\alpha}{2}\Big)\|{}_0D^{\frac{\alpha}{2}}_t w\|_{L^2(0,T)}^2.
\end{align*}
\end{lemma}
Now, for $0<\alpha<1,$ we define the space
\begin{eqnarray*}
    B^\alpha(Q_T)&:={}_0H^{\alpha}(0,T;L^2(\Omega)) \cap L^2(0,T;H^1(\Omega)),
\end{eqnarray*}
equipped with the norm
\begin{equation*}
    \|w\|_{B^\alpha(Q_T)} =\Big( \|w\|^2_{H^\alpha(0,T;L^2(\Omega))} + \|w\|^2_{L^2(0,T;H^1(\Omega))}
\Big) ^\frac{1}{2}.
\end{equation*}
Referring to \cite{li2009space}, we know that $B^\alpha(Q_T)$ is a Hilbert space. We also introduce the following space:
\begin{align}
\label{H001}  {}_0 H^{\frac{1}{2},\frac{\alpha}{4}}\left(\Sigma_{T}\right):=L^{2}(0, T ; H^{\frac{1}{2}}(\partial\Omega)) \cap   {}_0H^{\frac{\alpha}{4}}\left(0, T ; L^{2}(\partial\Omega)\right).
\end{align}
We represent the dual space of ${}_0 H^{\frac{1}{2},\frac{\alpha}{4}}\left(\Sigma_{T}\right)$ as $H^{-\frac{1}{2},-\frac{\alpha}{4}}\left(\Sigma_{T}\right)$. Furthermore, we use the notation $\Big<\psi,w\Big>_{1/2,\alpha/4}$ to describe the duality pairing between $H^{-\frac{1}{2},-\frac{\alpha}{4}}\left(\Sigma_{T}\right)$ and ${}_0 H^{\frac{1}{2},\frac{\alpha}{4}}\left(\Sigma_{T}\right)$, where $\psi$ belongs to $H^{-\frac{1}{2},-\frac{\alpha}{4}}\left(\Sigma_{T}\right)$ and $w$ belongs to ${}_0 H^{\frac{1}{2},\frac{\alpha}{4}}\left(\Sigma_{T}\right).$
Based on the findings from \cite[Theorem 4.2]{lions2012non} and \cite[Theorem 2.1]{lionsv2}, we can now derive the following result.
\begin{lemma} \label{lm2}
The trace map $\gamma:\; u \longmapsto u|_{\Sigma_T} $ is continuous and surjective from $B^ \frac{\alpha}{2} (Q_T)$ to ${}_0H^{\frac{1}{2}, \frac{\alpha}{4} }(\Sigma_T)$.
\end{lemma}

\subsection{Well-posedness of the forward problem}

The well-posedness of the forward problem \eqref{direct-problem} with non-zero Neumann boundary conditions has been extensively investigated in the literature. Rundell and Yamamoto \cite{rundell2023uniqueness10} established the existence and uniqueness of a weak solution to the one-dimensional problem \eqref{direct-problem} with $q^* \in C(0,1)$ and Neumann data $\varphi \in H_\alpha(0,T)$, using a representation formula. Yamamoto \cite{yamamoto2018} extended this result to the multi-dimensional case, proving the unique existence of a weak solution via the transposition method under the assumptions $q^* \in C(\overline{\Omega})$ and $\varphi \in L^2(0,T;H^{-\gamma}(\partial\Omega))$ with $\gamma>0$. In \cite{cen2023recovery}, Cen \emph{et al.} further demonstrated the well-posedness of \eqref{direct-problem} when the Neumann data $\varphi$ is separable, i.e., $\varphi(x,t) = g(x)\mu(t)$ with $g \in H^{\frac{1}{2}}(\partial\Omega)$ and $\mu \in C^1(\mathbb{R}_+)$. Guidetti \cite{guidetti2019maximal} showed the existence of a strong solution $u$ to \eqref{direct-problem}, which is continuous in time, belongs to the class $C^2$ in space, and admits a Caputo fractional derivative $\partial_t^\alpha u$ that is H\"older continuous in both time and space.
For completeness, for a positive potential $q^*\in L^\infty(\Omega)$, Kian and Yamamoto \cite{kian2021well} studied the well-posedness of \eqref{direct-problem} for both weak and strong solutions. More precisely, they proved the existence of a weak solution with non-zero initial data, Neumann boundary data $\varphi \in L^1(0,T;H^{-\theta}(\partial\Omega))$ (with $\theta \in [1/2,+\infty)$), and non-homogeneous source terms lying in negative-order Sobolev spaces. For strong solutions, they introduced an optimal compatibility condition and established the existence of solutions under this setting.
On the other hand, several authors have analyzed the well-posedness of \eqref{direct-problem} with homogeneous Neumann or Dirichlet boundary conditions (i.e., $\varphi=0$) and non-zero source terms; see, for example, \cite{sakamoto2011initial,zacher2009weak11}.

Our main contribution is to establish the existence and uniqueness of a weak solution to the forward problem \eqref{direct-problem} in the case where the potential $q^*$ is positive and belongs to $L^\infty(\Omega)$ and the Neumann boundary data $\varphi\in H^{-\frac{1}{2},-\frac{\alpha}{4}}(\Sigma_{T})$. Before proceeding, and following \cite{li2009space}, we recall the definition of weak solutions to \eqref{direct-problem}.

\begin{definition}\label{weak-solution}(Weak solution of \eqref{direct-problem}). Let $q^* \in L^\infty(\Omega)$ be a positive potential and $\varphi\in H^{-\frac{1}{2},-\frac{\alpha}{4} }\left(\Sigma_{T}\right)$. A function  $ u \in B^{\frac{\alpha}{2}}(Q_T)$ is said to
be a weak solution of the problem \eqref{direct-problem} if the following holds:
\begin{equation} \label{F.V}
    \mathcal{A}( u,\vartheta)=\mathcal{L}(\vartheta)\quad\forall\vartheta\in B^{\frac{\alpha}{2}}(Q_T),
\end{equation}
where the bilinear form $\mathcal{A}(\cdot,\, \cdot)$ and the linear form $\mathcal{L}(\cdot)$ are defined by
\begin{align*}
    \mathcal{A}(w,\vartheta)&=\int_0^T\int_{\Omega}\Big({}_0D^{\frac{\alpha}{2}}_t w\; {}_tD^{\frac{\alpha}{2}}_T\vartheta +\nabla w.\nabla\vartheta +q^*\,  w\,\vartheta\Big)\, \dx \dt\quad\hbox{for all }\, w,\vartheta\in B^{\frac{\alpha}{2}}(Q_T),\\
    \mathcal{L}(\vartheta)&= \Big<\varphi,\,\vartheta \Big>_{1/2,\alpha/4}\;\;\hbox{for all }\, \vartheta\in B^{\frac{\alpha}{2}}(Q_T).
\end{align*}
\end{definition}

The ensuing result establishes the existence and uniqueness of a weak solution for problem \eqref{direct-problem} by employing the well-known Lax-Milgram lemma.

\begin{theorem}\label{well-posedness-solution}(Well-posedness of \eqref{direct-problem}).
Let $q^* \in L^\infty(\Omega)$ be a positive potential and $\varphi \in H^{-\frac{1}{2}, -\frac{\alpha}{4}}(\Sigma_T)$. Then the initial-boundary value problem \eqref{direct-problem} admits a unique weak solution 
\[
u \in B^{\frac{\alpha}{2}}(Q_T)
\]
in the sense of Definition \ref{weak-solution}. Moreover, there exists a constant $C_N = C(\alpha,\Omega) > 0$ such that the following stability estimate holds:
\begin{equation}\label{estimation-direct}
    \| u \|_{B^{\frac{\alpha}{2}}(Q_T)} \leq C_N \, \|\varphi\|_{H^{-\frac{1}{2}, -\frac{\alpha}{4}}(\Sigma_T)}.
\end{equation}
\end{theorem}

\begin{proof} 
Using the Cauchy-Schwarz inequality, we can show that
\begin{align*}
    \Big|\mathcal{L}(\vartheta) \Big| \leq \| \varphi \|_{H^{-\frac{1}{2}, -\frac{\alpha}{4} }(\Sigma_T)} \| \vartheta \|_{{}_0H^{\frac{1}{2}, \frac{\alpha}{4} }(\Sigma_T)} .
\end{align*}
From Lemma \ref{lm2}, there exists a constant $C_1>0$ such that
\begin{align} \label{estimate_L}
    \Big|\mathcal{L}(\vartheta) \Big| \leq C_1\| \varphi \|_{H^{-\frac{1}{2}, -\frac{\alpha}{4} }(\Sigma_T)} \| \vartheta \|_{B^{\frac{\alpha}{2}}(Q_T)},
\end{align}
which implies the continuity of the linear form $\mathcal{L}.$ 

We will now establish the coercivity of $\mathcal{A}.$ According to Lemma \ref{cosinus}, we get
\begin{equation*}
    \int_0^T\int_{\Omega}{}_0D^{\frac{\alpha}{2}}_t \vartheta\, {}_tD^{\frac{\alpha}{2}}_T\vartheta\,\dx\dt\geq \cos\Big( \displaystyle\frac{\pi\alpha}{2} \Big)\|{}_0D^{\frac{\alpha}{2}}_t\vartheta\|_{L^2(Q_T)}^2\qquad \forall\vartheta\in B^{\frac{\alpha}{2}}(Q_T).
\end{equation*}
Consequently, we have
\begin{align*}
    \mathcal{A}(\vartheta,\vartheta)&=\int_0^T\int_{\Omega}{}_0D^{\frac{\alpha}{2}}_t \vartheta\, {}_tD^{\frac{\alpha}{2}}_T\vartheta\,\dx\dt +\int_0^T\int_{\Omega}\Big|\nabla \vartheta\Big|^2\dx\dt +\int_0^T\int_{\Omega} q^*\Big|\vartheta\Big|^2\, \dx \dt\\
    &\geq \cos\Big( \displaystyle\frac{\pi\alpha}{2} \Big)\|{}_0D^{\frac{\alpha}{2}}_t\vartheta\|_{L^2(Q_T)}^2+\|\nabla\vartheta\|_{L^2(Q_T)}^2+\int_0^T\int_{\Omega} q^*\Big|\vartheta\Big|^2\, \dx \dt.
\end{align*}
Exploiting the positivity of $q^*$ in $\Omega$, we obtain
\begin{align*}
    \mathcal{A}(\vartheta,\vartheta)
    &\geq \cos\Big( \displaystyle\frac{\pi\alpha}{2} \Big)\|{}_0D^{\frac{\alpha}{2}}_t\vartheta\|_{L^2(Q_T)}^2+\|\nabla\vartheta\|_{L^2(Q_T)}^2.
\end{align*}
Moreover, as shown in \cite[Lemma 2.4]{li2009space} and \cite[Lemma 2.5]{li2009space}, the semi-norm $\|{}_0D^{\frac{\alpha}{2}}_t  \cdot \|_{L^2(0,T)}$ is equivalent to the norm $\|\cdot \|_{H^{\frac{\alpha}{2}}(0,T)}.$ As a result, there exists a constant $C_2>0$ such that
\begin{align*}
    \mathcal{A}(\vartheta,\vartheta)
    &\geq C_2\cos\Big( \displaystyle\frac{\pi\alpha}{2} \Big)\| \vartheta\|_{H^{\frac{\alpha}{2}}(0,T;L^2(\Omega))}^2+\|\nabla\vartheta\|_{L^2(Q_T)}^2\\
    &= \frac{C_2}{2}\cos\Big( \displaystyle\frac{\pi\alpha}{2} \Big)\| \vartheta\|_{H^{\frac{\alpha}{2}}(0,T;L^2(\Omega))}^2+\frac{C_2}{2}\cos\Big( \displaystyle\frac{\pi\alpha}{2} \Big)\| \vartheta\|_{H^{\frac{\alpha}{2}}(0,T;L^2(\Omega))}^2+\|\nabla\vartheta\|_{L^2(Q_T)}^2.
\end{align*}
Thanks to the compact embedding $H^{\alpha/2}(0,T) \subset L^2(0,T)$ (as established in \cite{kubica2020time}) for all $0<\alpha<1$, we can conclude:
\begin{align*} 
\begin{split}
    \mathcal{A}(\vartheta,\vartheta)
    &\geq \frac{C_2}{2}\cos\Big( \displaystyle\frac{\pi\alpha}{2} \Big)\| \vartheta\|_{H^{\frac{\alpha}{2}}(0,T;L^2(\Omega))}^2+\frac{C_2}{2}\cos\Big( \displaystyle\frac{\pi\alpha}{2} \Big)\| \vartheta\|_{L^2(Q_T)}^2+\|\nabla\vartheta\|_{L^2(Q_T)}^2\\
    &\geq \frac{C_2}{2}\cos\Big( \displaystyle\frac{\pi\alpha}{2} \Big)\| \vartheta\|_{H^{\frac{\alpha}{2}}(0,T;L^2(\Omega))}^2+\min\Big\{1,\frac{C_2}{2}\cos\Big( \displaystyle\frac{\pi\alpha}{2} \Big)\Big\}\| \vartheta\|_{L^{2}(0,T;H^1(\Omega))}^2.
    \end{split}
\end{align*}
Therefore, we can state
\begin{align} \label{estimate_A}
    \mathcal{A}(\vartheta,\vartheta)
    &\geq  C_\alpha\|\vartheta\|_{B^{\frac{\alpha}{2}}(Q_T)}^2,
\end{align}
which implies the coercivity of $\mathcal{A}(\cdot,\cdot).$ In the same way, we can verify that the bilinear form $\mathcal{A}(\cdot,\cdot)$ is continuous. Then according to Lax-Milgram lemma (see \cite{brezis2010functional}), there exists a unique weak solution $ u\in B^{\frac{\alpha}{2}}(Q_T)$ to \eqref{direct-problem}. Furthermore, by \eqref{estimate_L} and \eqref{estimate_A}, we deduce
\begin{equation}\label{DFHHH}
    \| u\|_{B^{\frac{\alpha}{2}}(Q_T)}\leq C_N\|\varphi\|_{H^{-\frac{1}{2}, -\frac{\alpha}{4} }(\Sigma_T)},
\end{equation}
which completes the proof.
\end{proof}
In the remainder of this paper, we assume that the measured data $\phi$ are compatible and belong to the space ${}_0H^{\frac{1}{2},\frac{\alpha}{4}}\left(\Sigma_{T}\right)$ as defined in \eqref{H001}. Moreover, the Neumann data $\varphi$ is assumed to be a nontrivial element of $H^{-\frac{1}{2},-\frac{\alpha}{4}}\left(\Sigma_{T}\right)$.

With these assumptions in place, we now turn to the central objective of this work: the development of a reliable and efficient numerical method for reconstructing the potential $q^*$. To this end, we reformulate the inverse potential problem \eqref{inverse} as an equivalent optimization problem. Specifically, we introduce an objective functional inspired by the Kohn–Vogelius formulation, whose minimization provides the desired reconstruction.

\section{Problem reformulation -A Kohn-Vogelius formulation}\label{sec:4}

The Kohn-Vogelius formulation provides a powerful framework for recasting the inverse potential problem \eqref{inverse} as an optimization problem. This approach relies on the introduction of two auxiliary boundary value problems defined for any admissible potential $q \in \Phi_{\text{ad}}$.\,\\

\noindent $-$ The first auxiliary problem, termed the \emph{Neumann problem}, incorporates the prescribed flux $\varphi$ as a boundary condition:
\begin{eqnarray} \label{PN}
\left\{
  \begin{array}{rcll}
    \partial_{t} ^{\alpha}  u_N[q]-\Delta  u_N[q] +q\, u_N[q]&=&0&\hbox{in}\ Q_T, \\
    \displaystyle \partial_\nu u_N[q]&=&\varphi &\hbox{on}\  \Sigma_T, \\
   u_N[q](\cdot,0)&=&0  &\hbox{in}\ \Omega.
  \end{array}
\right.
\end{eqnarray}
The existence and uniqueness of a weak solution to the Neumann problem \eqref{PN} have been established in Theorem \ref{weak-solution}.\,\\

\noindent $-$ The second auxiliary problem, referred to as the \emph{Dirichlet problem}, employs the measured boundary data $\phi$ as a Dirichlet condition:
\begin{eqnarray} \label{PD}
\left\{
  \begin{array}{rcll}
     \partial_{t} ^{\alpha}  u_D[q]-\Delta  u_D[q] +q\, u_D[q]&=&0&\hbox{in}\ Q_T, \\
    u_D[q]&=&\phi &\hbox{on}\  \Sigma_T, \\
   u_D[q](\cdot,0)&=&0  &\hbox{in}\ \Omega.
  \end{array}
\right.
\end{eqnarray}
The well-posedness of the Dirichlet problem \eqref{PD} was established by Kemppainen and Ruotsalainen \cite[Corollary 3.6]{kemppainen2009boundary}. Specifically, for $\phi \in {}_0H^{\frac{1}{2},\frac{\alpha}{4}}(\Sigma_{T})$, there exists a unique solution $u_D[q] \in B^{\frac{\alpha}{2}}(Q_T)$ satisfying the estimate
\begin{equation} \label{estimate-Direchlet}
    \| u_D[q] \|_{B^{\frac{\alpha}{2}}(Q_T)} \leq C_D \|\phi\|_{H^{\frac{1}{2}, \frac{\alpha}{4}}(\Sigma_T)},
\end{equation}
where $C_D := C(\alpha, T, \Omega) > 0$ is a constant independent of $q$. In \cite[Corollary 3.6]{kemppainen2009boundary}, the boundary data $\phi$ belongs to the anisotropic Sobolev space $\widetilde{H}^{\frac{1}{2},\frac{\alpha}{4}}(\Sigma_{T})$. Since $\frac{\alpha}{4} < \frac{1}{2}$ for $\alpha \in (0,1)$, it follows from \cite{lions2012non} that ${}_0H^{\frac{1}{2},\frac{\alpha}{4}}(\Sigma_{T})$ coincides with $\widetilde{H}^{\frac{1}{2},\frac{\alpha}{4}}(\Sigma_{T})$, which aligns with our framework.

A key observation underpinning the Kohn-Vogelius approach is that when $q$ coincides with the true potential $q^*$, the solutions to the auxiliary problems must satisfy $u_D[q^*] = u_N[q^*]$ throughout $Q_T$. This motivates the definition of the Kohn-Vogelius functional, which quantifies the discrepancy between these solutions:
\begin{equation}\label{dstance}
   \mathcal{K}( q)=\int_0^T \int_{\Omega}\Big|\nabla\Big(
 u_N[q]- u_D[q]\Big)\Big|^2\ \dx \dt +\int_0^T\int_{\Omega}q\Big| u_N[q]- u_D[q]\Big|^2 \dx \dt.
\end{equation}

The equivalence between minimizing $\mathcal{K}(q)$ over $\Phi_{\text{ad}}$ and solving the inverse problem \eqref{inverse} follows from two implications:
\begin{itemize}
    \item[$-$] If $q^* \in \Phi_{\text{ad}}$ solves the over-determined problem \eqref{direct-problem}--\eqref{inverse}, then $u_N[q^*] = u_D[q^*]$, and consequently $\mathcal{K}(q^*) = 0$.
    \item[$-$] Conversely, if $\mathcal{K}(q^*) = 0$ for some $q^* \in \Phi_{\text{ad}}$, then $u = u_N[q^*] = u_D[q^*]$ satisfies both the governing equation and boundary conditions of \eqref{direct-problem}--\eqref{inverse}.
\end{itemize}

However, the minimization of the Kohn-Vogelius functional $\mathcal{K}(q)$ constitutes an ill-posed inverse problem, exhibiting significant sensitivity to noise in the Dirichlet measurement data $\phi$. To mitigate this instability, we employ Tikhonov regularization, which leads to the following stabilized optimization problem: find the potential $q^* \in \Phi_{ad}$ that minimizes
\begin{equation}\label{stable} 
    \displaystyle\operatorname*{Minimize}_{q\in\Phi_{ad}}\
\mathcal{K}_\rho( q):=\mathcal{K}( q)+\rho\|q\|^2_{L^2(\Omega)},
\end{equation}
where $\rho>0$ is a regularization parameter. Next, we demonstrate the following properties regarding problem \eqref{stable}:
\begin{itemize}
    \item[(i)] Existence and uniqueness: There exists a unique solution to \eqref{stable}.
\item[(ii)] Stability: The solution of \eqref{stable} exhibits stability with respect to the measured Dirichlet data.
\end{itemize}

\section{Existence and uniqueness of optimal solutions}\label{sec:6}

This section establishes the existence and uniqueness of an optimal solution to problem \eqref{stable}. The analysis begins with a continuity result that is fundamental to the subsequent proofs.

\subsection{Continuity of the Dirichlet and Neumann solution operators}

Here, we establish the continuity of the solution operator \( q \longmapsto u_D[q] \) on \( \Phi_{\text{ad}} \). The continuity of the Neumann solution operator \( q \longmapsto u_N[q] \) follows by an analogous argument.

\begin{lemma} \label{conv_D} Let $q\in\Phi_{ad}$ and $q^n$ be a sequence of functions in $\Phi_{ad}.$ If $\{ q^n \}_n$ converges strongly to  $q$ in $L^2(\Omega)$. Then
\begin{equation*} 
    u_D[q^n]\longrightarrow u_D[q]\; \hbox{ strongly in }\; L^2(0,T;H^1(\Omega))\; \hbox{ as }\; n\to\infty, 
\end{equation*}
where $u_D[q]$ is the solution of \eqref{PD} and $u_D[q^n]$ is the solution of \eqref{PD} with $q^n$ replacing $q$.
\end{lemma}

\begin{proof} Consider the difference $w_n=  u_D[q^n]-  u_D[q]$, which is the solution to

\begin{eqnarray} \label{P_d}
\left\{
  \begin{array}{rcll}
    \partial_{t} ^{\alpha} w_n-\Delta w_n +q^n w_n&=&-(q^n -q)  u_D[q]&\hbox{in}\ Q_T, \\
    \displaystyle w_n&=&0 &\hbox{on}\  \Sigma_T, \\
    \displaystyle w_n(\cdot,0)&=&0  &\hbox{in}\ \Omega.
  \end{array}
\right.
\end{eqnarray}
By employing the weak formulation of the problem \eqref{P_d} and choosing $w_n$ as a test function, coupled with the application of Lemma \ref{cosinus}, we derive
\begin{align*}
    \cos( \frac{\pi\alpha}{2})\|{}_0D^{\frac{\alpha}{2}}_t  w_n\|_{L^2(Q_T)}^2 &+ \|\nabla w_n\|_{L^2(Q_T)}^2 + \int_0^T\int_{\Omega} q^n |w_n|^2 \dx \dt \\&\leq \int_0^T\int_{\Omega}\Big|(q^n -q) u_D[q]w_n\Big| \dx \dt.
\end{align*}
Considering the nonnegative nature of $\displaystyle\cos(\frac{\pi\alpha}{2})$ (because $0<\alpha<1$) and the presence of $ q^n>0$ within $\Omega,$ we obtain
\begin{align*}
     \|\nabla w_n\|_{L^2(Q_T)}^2  \leq \int_0^T\int_{\Omega}\Big|(q^n -q) u_D[q]w_n\Big| \dx \dt.
\end{align*}
With the help of the Poincar\'e  inequality, there exists a constant $C>0$ independent of $n$
such that
\begin{align} \label{h1+}
    \|w_n\|^2_{L^2(0,T;H^1(\Omega))} \leq C \int_0^T\int_{\Omega}\Big|(q^n -q) u_D[q]w_n\Big| \dx \dt,
\end{align}
Now, let's proceed to estimate the expression $\ds\int_0^T\int_{\Omega}\Big|(q^n -q) u_D[q]w_n\Big| \dx \dt.$ To this end, let $1< r,s<\infty$ be such that $1/r+2/s=1.$ Utilizing the H\"{o}lder inequality, we have
\begin{align*} 
\int_\Omega \Big|(q^n -q) u_D[q]\, w_n\Big|\dx\leq  \|q^n -q\|_{L^{r}(\Omega)}\| u_D[q]\|_{L^{s}(\Omega)}\|w_n\|_{L^{s}(\Omega)}.
\end{align*}
Choosing $r=2$ and $s=4$ and from the Sobolev embedding theorem (see, for example, \cite{brezis2010functional}), there exists a constant $C(\Omega,d)>0$ such that
\begin{align} \label{h11}
\int_0^T\int_\Omega\Big|(q^n &-q) u_D[q]\, w_n\Big|\dx\dt\nonumber\\&\leq C(\Omega,d)
\|q^n -q\|_{L^2(\Omega)}\| u_D[q]\|_{L^2(0,T;H^1(\Omega))}\|w_n\|_{L^2(0,T;H^1(\Omega))}.
\end{align}
Based on the estimate provided in \eqref{estimate-Direchlet}, we derive
\begin{align} \label{s0}
\int_0^T\int_\Omega\Big|(q^n -q) u_D[q]\, w_n\Big|\dx\dt&\leq C(\Omega,d)C_D
\|q^n -q\|_{L^2(\Omega)}\|\phi\|_{H^{\frac{1}{2}, \frac{\alpha}{4} }(\Sigma_T)}\|w_n\|_{L^2(0,T;H^1(\Omega))}.
\end{align}
By incorporating \eqref{s0} into \eqref{h1+}, one can deduce that there exists a constant $C(\alpha,T,\Omega,d):=C(\Omega,d)C_D C>0$ such that
\begin{align*}
  \|w_n\|_{L^2(0,T;H^1(\Omega))}\leq   C(\alpha,T,\Omega,d)
\|q^n -q\|_{L^2(\Omega)}\|\phi\|_{H^{\frac{1}{2}, \frac{\alpha}{4} }(\Sigma_T)}.
\end{align*}
Since  $\{ q^n \}_n$ converges strongly to  $q$ in $L^2(Q_T)$, we get

\begin{align*} 
   \|w_n\|_{L^2(0,T;H^1(\Omega))} \longrightarrow 0\; \hbox{ as }\; n\to \infty,
\end{align*} which eventually implies that
\begin{align*}
    u_D[q^n]\longrightarrow u_D[q]\; \hbox{ strongly in } \; L^2(0,T;H^1(\Omega)).
\end{align*}
\end{proof}

The continuity of the Neumann solution with respect to the potential $q\in\Phi_{ad}$ follows from an analysis analogous to that used in the proof of Lemma \ref{conv_D}.

\begin{lemma} \label{conv_N}
Let $q\in\Phi_{ad}$ and $q^n$ be a sequence of functions in $\Phi_{ad}.$ If $\{ q^n \}_n$ converges strongly to  $q$ in $L^2(\Omega)$. Then
\begin{equation} \label{convergancefaible-N}
    u_N[q^n]\longrightarrow u_N[q]\; \hbox{ strongly in }\; L^2(0,T;H^1(\Omega))\; \hbox{ as }\; n\to\infty, 
\end{equation}
where $u_N[q]$ is the solution of \eqref{PN} and $u_N[q^n]$ is the solution of \eqref{PN} with $q^n$ replacing $q$.
\end{lemma}

\subsection{Existence of optimal solutions}

In this context, we aim to prove the existence of an optimal solution to the optimization problem \eqref{stable}.

\begin{theorem}[Existence]\label{theorem-existence}
For every $\rho>0$ there exists at least one solution of the optimization
problem \eqref{stable}.
\end{theorem}
\begin{proof}
The Kohn-Vogelius functional $\mathcal{K}_\rho$ from \eqref{stable} is non-negative, implying that $\displaystyle\inf\mathcal{K}_\rho$ is a finite quantity in the class $\Phi_{ad}.$ Consequently, we can establish the existence of a minimizing sequence $\{q^n\}\subset\Phi_{ad}$ such that:
\begin{align*}
\lim_{n\rightarrow\infty}\mathcal{K}_\rho(q^n)=\inf_{q\in\Phi_{ad}}\mathcal{K}_\rho(q).
\end{align*}
From the definition of the admissible solution set $\Phi_{ad}$, the sequence ${q^n}$ is bounded in $H^1(\Omega)$. Consequently, there exists a subsequence, denoted by $n$, for which $q^n$ weakly converges to some $q^0$ in $H^1(\Omega).$ We show that $q^0$ is in fact a minimum for $\mathcal{K}_\rho.$

Being a proper closed convex subset of $H^1(\Omega)$, $\Phi_{ad}$ is also a closed subset of $H^1(\Omega)$ for the weak topology (see, for example, \cite{brezis2010functional}), from which we derive that $q^0$ belongs to the class $\Phi_{ad}$

 Firstly, utilizing the compactness property derived from the Rellich-Kondrachov theorem (see, for instance, \cite{brezis2010functional}), we obtain the following convergence result:
\begin{equation}\label{cv-0}
q^n\longrightarrow q^0\; \hbox{strongly in}\; L^2(\Omega)\; \hbox{as}\; n\rightarrow\infty.
\end{equation}
Subsequently, leveraging the convergence outcomes described in Lemmas \ref{conv_D} and \ref{conv_N}, we obtain
\begin{align} \label{a+1}
     u_N[q^n]- u_D[q^n] \longrightarrow  u_N[q^0]- u_D[q^0]\; \hbox{strongly in}\; L^2(0,T;H^1(\Omega))\; \hbox{as}\; n\rightarrow\infty.
\end{align}
We shall now establish the convergence result:
\begin{align*}
    \mathcal{I}_n \longrightarrow  0 \;\; \hbox{as}\; n\rightarrow\infty,
\end{align*}
where the term $\mathcal{I}_n$ is defined by
\begin{align*}
  \mathcal{I}_n=  \int_0^T \int_{\Omega}q^n\Big|
 u_N[q^n]- u_D[q^n]\Big|^2\ \dx \dt -  \int_0^T \int_{\Omega}q^0\Big|
 u_N[q^0]- u_D[q^0]\Big|^2\ \dx \dt.
\end{align*}
To this end, we decompose the term $\mathcal{I}_n$ in two terms as
\begin{align*}
  \mathcal{I}_n&=  \mathcal{I}_{n,1}+\mathcal{I}_{n,2},
\end{align*}
where
\begin{align*}
  \mathcal{I}_{n,1}&=  \int_0^T \int_{\Omega}q^n\Big(\Big|
 u_N[q^n]- u_D[q^n]\Big|^2 -\Big|
 u_N[q^0]- u_D[q^0]\Big|^2\Big) \ \dx \dt, \\
 \mathcal{I}_{n,2}&=  \int_0^T \int_{\Omega}\Big(q^n-q^0\Big)\Big|
 u_N[q^0]- u_D[q^0]\Big|^2\ \dx \dt.
\end{align*}
Since $q^n \in \Phi_{ad}$, then $q^n\leq c^\prime$ in $\Omega$ for each $n$. So
\begin{align*}
  \Big|\mathcal{I}_{n,1}\Big|&\leq  c^\prime\Big|\int_0^T \int_{\Omega}\Big(\Big|
 u_N[q^n]- u_D[q^n]\Big|^2 -\Big|
 u_N[q^0]- u_D[q^0]\Big|^2 \Big) \dx \dt\Big|.
\end{align*}
Utilizing the convergence result presented in \eqref{a+1}, we can conclude that
\begin{align} \label{cv-0+}
    \mathcal{I}_{n,1} \longrightarrow  0 \;\; \hbox{as}\; n\rightarrow\infty.
\end{align}

We now turn our attention to examining the convergence of $\mathcal{I}_{n,2}$. Employing the H\"older inequality, we can derive
\begin{align*}
    \mathcal{I}_{n,2} \leq \|q^n-q^0\|_{L^2(\Omega)}\|u_N[q^0]- u_D[q^0]\|^2_{L^2(0,T;L^4(\Omega))}.
\end{align*}
With the assistance of the Sobolev embedding theorem, we can infer:
\begin{align*} 
 \mathcal{I}_{n,2} \leq C(\Omega,d)\|q^n-q^0\|_{L^2(\Omega)}\|u_N[q^0]- u_D[q^0]\|^2_{L^2(0,T;H^1(\Omega))}.
\end{align*}
By employing the triangle inequality in conjunction with the estimates \eqref{estimation-direct} and \eqref{estimate-Direchlet}, we obtain
\begin{align*} 
     \mathcal{I}_{n,2} \leq C(\Omega,d)\|q^n-q^0\|_{L^2(\Omega)}\Big(C_N\|\varphi\|_{H^{-\frac{1}{2},-\frac{\alpha}{4} }(\Sigma_{T})} + C_D\|\phi\|_{H^{\frac{1}{2},\frac{\alpha}{4} }(\Sigma_{T})}\Big)^2.
\end{align*}
Therefore, from \eqref{cv-0}, we can show that
\begin{align} \label{cv-0++}
    \mathcal{I}_{n,2} \longrightarrow  0 \;\; \hbox{as}\; n\rightarrow\infty.
\end{align}
As a result, by combining both \eqref{cv-0+} and \eqref{cv-0++}, we arrive at the following convergence result:
\begin{align} \label{c2}
    \lim_{n\rightarrow\infty} \int_0^T \int_{\Omega}q^n\Big|
 u_N[q^n]- u_D[q^n]\Big|^2\ \dx \dt = \int_0^T \int_{\Omega}q^0\Big|
 u_N[q^0]- u_D[q^0]\Big|^2\ \dx \dt.
\end{align}

Consequently, by combining the convergence results \eqref{cv-0}, \eqref{a+1}, and \eqref{c2}, along with the lower semi-continuity of the $L^2$-norm, we can establish the following inequality:
\begin{align*}
\mathcal{K}_\rho(q^0) &=\int_0^T \int_{\Omega}\Big|\nabla\Big(
 u_N[q^0]- u_D[q^0]\Big)\Big|^2\ \dx \dt + \int_0^T \int_{\Omega}q^0\Big|
 u_N[q^0]- u_D[q^0]\Big|^2\ \dx\dt+\rho\int_0^T \int_{\Omega}\Big|q^0\Big|^2\ \dx\dt\\
&\leq\lim_{n\rightarrow\infty}\inf\int_0^T \int_{\Omega}\Big|\nabla\Big(
 u_N[q^n]- u_D[q^n]\Big)\Big|^2\ \dx \dt \qquad\qquad\\
&+ \lim_{n\rightarrow\infty}\inf\int_0^T \int_{\Omega}q^n\Big|
 u_N[q^n]- u_D[q^n]\Big|^2\ \dx\dt+ \rho\lim_{n\rightarrow\infty}\inf\int_0^T \int_{\Omega}\Big|q^n\Big|^2\ \dx\dt\\
&\leq\lim_{n\rightarrow\infty}\inf\mathcal{K}_\rho(q^n)=\inf_{q\in\Phi_{ad}}\mathcal{K}_\rho(q).
\end{align*}
This inequality implies that $q^0$ serves as a minimizer for the optimization problem \eqref{stable}.
\end{proof}

\subsection{Uniqueness of optimal solutions}\label{sec:77}
In this section, we address the uniqueness of the optimal solution to the minimization problem \eqref{stable}. To this end, we begin by analyzing the convexity of the Kohn–Vogelius functional $\K_\rho$. This analysis requires first establishing the Lipschitz continuity of the forward operators and their differentiability with respect to the potential $q$.

\subsubsection{\bf Lipschitz continuity analysis of the forward operators}

Here, we will delve into the examination of the Lipschitz continuity property exhibited by the Dirichlet and Neumann solution operators.

\begin{lemma}\label{lemma-L-C-D}
The mapping $q\longmapsto u_D[q]$ is Lipschitz continuous from $\Phi_{ad}$ to $L^2(0,T;H^1(\Omega))$, i.e., for any $q$, $q+\delta q \in \Phi_{ad}$ and the corresponding $u_D[q], u_D[q+\delta q] \in B^{\frac{\alpha}{2}}(Q_T)$, there holds
\begin{align} \label{lip_D}
    \| u_D[q + \delta q] - u_D[q]\|_{L^2(0,T;H^1(\Omega))}&\leq \frac{C(\Omega,d)C_D\|\phi\|_{H^{\frac{1}{2}, \frac{\alpha}{4} }(\Sigma_T)}}{\min\{1,\frac{T^{-\alpha}}{2\Gamma(1-\alpha)}\}}
\|\delta q\|_{L^2(\Omega)},
\end{align}
where $C_D$ denotes the constant specified in \eqref{estimate-Direchlet}, while $C(\Omega,d)>0$ represents another constant depending only on the domain $\Omega$ and the space's dimension $d$.
\end{lemma}

\begin{proof}
Denote $\delta \vartheta_{D}=u_D[q+\delta q]-u_D[q],$
then $\delta \vartheta_{D}$ satisfies the problem
\begin{eqnarray} \label{P_d+}
\left\{
  \begin{array}{rcll}
    \partial^\alpha _{t} \delta \vartheta_{D}   -\Delta \delta \vartheta_{D} +(q+\delta q)\delta \vartheta_{D}&=&-\delta q u_D[q]&\hbox{in}\ Q_T, \\
    \displaystyle \delta \vartheta_{D} &=&0 &\hbox{on}\  \Sigma_T, \\
    \displaystyle \delta \vartheta_{D} (\cdot,0)&=&0  &\hbox{in}\ \Omega.
  \end{array}
\right.
\end{eqnarray}
Since $q\in \Phi_{ad}$ and $q + \delta q \in \Phi_{ad}$, it follows that $|\delta q| \leq c^\prime - c$. Moreover, according to \eqref{estimate-Direchlet}, $u_D[q] \in L^2(0,T;L^2(\Omega))$, implying $\delta q \, u_D[q] \in L^2(0,T;L^2(\Omega))$. By applying \cite[Lemma 2.4]{jiang2017weak}, problem \eqref{P_d+} admits a unique weak solution $\delta\vartheta_D$ and
satisfying the regularity property:
\begin{equation}\label{es-regu}
    \delta \vartheta_{D} \in {}_0H^\alpha(0,T;L^2(\Omega))\cap L^2(0,T;H^2(\Omega)\cap H^1_0(\Omega)).
    \end{equation}

Now, employing the weak formulation of the problem \eqref{P_d+} and choosing $\delta \vartheta_{D}$ as a test function, we derive
\begin{align} \label{laq}
\begin{split}
\Big|\int_0^T\int_{\Omega}\delta \vartheta_{D}\;\partial_t^\alpha \delta \vartheta_{D}\dx\dt+\int_0^T\int_{\Omega}\Big|\nabla \delta \vartheta_{D}\Big|^2\dx\dt&+\int_0^T\int_{\Omega}(q+\delta q)  \Big|\delta \vartheta_{D}\Big|^2 \dx \dt\Big|\\
&\quad\leq \int_0^T\int_{\Omega}\Big|\delta q \,u_D[q] \delta \vartheta_{D}\Big| \dx \dt.
\end{split}
\end{align}
According to \cite[Theorem 3.3]{kubica2020time}, we know that
\begin{equation}\label{kubica-estimation}
    \int_0^T\int_\Omega\Big(\partial_t^\alpha g\Big)\, g\,\dx\dt\geq \frac{T^{-\alpha}}{2\Gamma(1-\alpha)}\|g \|^2_{L^2(0,T;\, L^2(\Omega))}
\end{equation}
for all $g\in L^2(0,T;H^1_0(\Omega))\cap {}_0H^\alpha(0,T;L^2(\Omega)).$
Using the regularity property \eqref{es-regu} and the above coercivity inequality \eqref{kubica-estimation}, we derive
\begin{align}\label{enq}
\frac{T^{-\alpha}}{2\Gamma(1-\alpha)}\|\delta \vartheta_{D} \|^2_{L^2(0,T;\, L^2(\Omega))}\leq\int_0^T\int_\Omega \delta \vartheta_{D} \partial_{t}^\alpha\delta \vartheta_{D}\dx\dt.
\end{align}
Applying the H\"older inequality, we have
\begin{align*} 
\Big|\int_\Omega\delta q u_D[q]\, \delta \vartheta_{D}\dx\Big|&\leq
\|\delta q\|_{L^2(\Omega)}\| u_D[q]\|_{L^4(\Omega)}\|\delta \vartheta_{D}\|_{L^4(\Omega)}.
\end{align*}
On the other hand, from the Sobolev embedding theorem, there exists a constant $C(\Omega,d)>0$ such that
\begin{align} \label{h1++}
\Big|\int_0^T\int_\Omega\delta q u_D[q]\, \delta \vartheta_{D}\dx\dt\Big|&\leq C(\Omega,d)
\|\delta q\|_{L^2(\Omega)}\| u_D[q]\|_{L^2(0,T;H^1(\Omega))}\|\delta \vartheta_{D}\|_{L^2(0,T;H^1(\Omega))}.
\end{align}
Using \eqref{estimation-direct}, we get
\begin{align} \label{s0+}
\int_0^T\int_\Omega\Big|\delta q u_D[q]\, \delta \vartheta_{D}\Big|\dx\dt&\leq C(\Omega,d)C_D
\|\delta q\|_{L^2(\Omega)}\|\phi\|_{H^{\frac{1}{2}, \frac{\alpha}{4} }(\Sigma_T)}\|\delta \vartheta_{D}\|_{L^2(0,T;H^1(\Omega))}.
\end{align}
By incorporating \eqref{s0+} and \eqref{enq} into \eqref{laq} and taking into account the positivity of $q+\delta q$ within $\Omega$, we can infer that
\begin{align*}
    \|\delta \vartheta_{D}\|_{L^2(0,T;H^1(\Omega))}&\leq \frac{ C(\Omega,d)C_D\|\phi\|_{H^{\frac{1}{2}, \frac{\alpha}{4} }(\Sigma_T)}}{\min\{1,\frac{T^{-\alpha}}{2\Gamma(1-\alpha)}\}}
\|\delta q\|_{L^2(\Omega)}.
\end{align*}
The proof is completed.
\end{proof}
By invoking the same argument employed in the proof of Lemma \ref{lemma-L-C-D}, we can succinctly state the ensuing Lipschitz continuity conclusion regarding the Neumann solution.

\begin{lemma} \label{lipshitz}
The mapping $q\longmapsto u_N[q]$ is Lipschitz continuous from $\Phi_{ad}$ to $B^{\frac{\alpha}{2}}(Q_T)$, i.e., for any $q$, $q+\delta q \in \Phi_{ad}$ and the corresponding $u_N[q], u_N[q+\delta q] \in B^{\frac{\alpha}{2}}(Q_T)$, there holds
\begin{align} \label{lip_N}
     \|u_N[q+\delta q]-u_N[q]\|_{L^2(0,T;H^1(\Omega))}\leq \frac{C(\Omega,d)C_N\|\varphi\|_{H^{-\frac{1}{2}, -\frac{\alpha}{4} }(\Sigma_T)}}{\min\{1,\frac{T^{-\alpha}}{2\Gamma(1-\alpha)}\}} \| \delta q \|_{L^2(\Omega)},
\end{align}
where $C_N$ denotes the constant specified in \eqref{estimation-direct}, while $C(\Omega,d)>0$ represents another constant depending only on the domain $\Omega$ and the space's dimension $d$.
\end{lemma}

\subsubsection{\bf Fr\'echet differentiability analysis of the forward operators}

Here, we demonstrate that the solutions $u_D[q]$ and $u_N[q]$ are Fr\'echet differentiable with respect to the potential $q\in \Phi_{ad}.$

\begin{lemma} \label{taylor_d}
The mapping $q\longmapsto u_D[q]$ from $\Phi_{ad}$ to $B^{\frac{\alpha}{2}}(Q_T)$ is Fr\'echet
differentiable in the sense that for any $\delta q$, $q+\delta q \in \Phi_{ad}$  there exists a bounded linear operator $\mathcal{U}_D[q]: L^2(\Omega) \mapsto B^{\frac{\alpha}{2}}(Q_T) $ such that
\begin{equation*}
    u_D [q+\delta q]=u_D [q]+\mathcal{U}_{D}[q] (\delta q)+o_{B^{\frac{\alpha}{2}}(Q_T)}\Big(\|\delta q\|_{L^2(\Omega)}\Big),
\end{equation*}
where $\mathcal{U}_{D}[q] (\delta q)$ satisfies the sensitive problem:
\begin{eqnarray} \label{P_d++}
\left\{ 
  \begin{array}{rcll}
    \partial_{t} ^{\alpha} \mathcal{U}_{D}[q] (\delta q)-\Delta \mathcal{U}_{D}[q] (\delta q) +q\, \mathcal{U}_{D} [q](\delta q)&=& -\delta q \,u_D [q] &\hbox{in}\ Q_T, \\
    \mathcal{U}_{D}[q] (\delta q) &=&0 &\hbox{on}\  \Sigma_T, \\
    \displaystyle \mathcal{U}_{D}[q] (\delta q)&=&0  &\hbox{in}\ \Omega\times\{0\}.
  \end{array}
\right.
\end{eqnarray}
Moreover, there exists a constant $C(\alpha,T,\Omega,d)>0$ such that
\begin{align*} 
    \|\mathcal{U}_{D}[q] (\delta q)\|_{B^{\frac{\alpha}{2}}(Q_T)} \leq C(\alpha,\Omega,d)\|\phi\|_{H^{\frac{1}{2}, \frac{\alpha}{4} }(\Sigma_T)}\|\delta q\|_{L^2(\Omega)}.
\end{align*}
\end{lemma}

\begin{proof}
Consider the problem
\begin{eqnarray} \label{w__}
\left\{
  \begin{array}{rcll}
    \partial_{t} ^{\alpha}  w-\Delta w +q\,  w&=& -\delta q\, u_D[q] &\hbox{in}\ Q_T, \\
      w &=&0 &\hbox{on}\  \Sigma_T, \\
    \displaystyle  w(\cdot,0)&=&0  &\hbox{in}\ \Omega,
  \end{array}
\right.
\end{eqnarray}
where $\delta q\in L^2(\Omega)$ such that $q+\delta q \in \Phi_{ad}$. In a manner analogous to \eqref{F.V}, we can derive a weak formulation for the problem \eqref{w__} as follows: Find $w\in B^{\frac{\alpha}{2}}(Q_T)$ such that
\begin{align}\label{GFJJ}\int_0^T\int_{\Omega}\Big({}_0D^{\frac{\alpha}{2}}_t w\; {}_tD^{\frac{\alpha}{2}}_T\vartheta +\nabla w.\nabla\vartheta +q\,  w\,\vartheta\Big)\, \dx \dt=-\int_0^T\int_{\Omega}\delta q u_D[q]\vartheta \dx \dt,
\end{align}
for all $\vartheta\in B^{\frac{\alpha}{2}}(Q_T)$ and $w_{|_{\Sigma_T}}=0$ on $\Sigma_T$. Following the same analysis as in the proof of Theorem \ref{weak-solution}, one can deduce the existence of a unique solution $w\in B^\frac{\alpha}{2}(Q_T)$ to \eqref{w__}.
\vspace{0.1 cm}

\noindent As the inequality \eqref{s0+}, we derive
\begin{align*} 
\int_0^T\int_\Omega\Big|\delta q u_D [q]\, \vartheta\Big|\dx\dt&\leq C(\Omega,d)C_D
\|\delta q\|_{L^2(\Omega)}\|\phi\|_{H^{\frac{1}{2}, \frac{\alpha}{4} }(\Sigma_T)}\|\vartheta\|_{L^2(0,T;H^1(\Omega))}.
\end{align*}
By taking $\vartheta=w$ as a test function in \eqref{GFJJ} and following the same steps as in the proof of \eqref{DFHHH}, one can establish the existence of a constant $C(\alpha,T,\Omega,d)>0$ such that
\begin{equation*}
    \| w\|_{B^{\frac{\alpha}{2}}(Q_T)}\leq C(\alpha,T,\Omega,d)\|\phi\|_{H^{\frac{1}{2}, \frac{\alpha}{4} }(\Sigma_T)}\|\delta q \|_{ L^2(\Omega)}.
\end{equation*}
Thus, the map $\delta q \mapsto w$ from $L^2(\Omega)$ to $B^{\frac{\alpha}{2}}(Q_T)$ defines a bounded linear operator $\mathcal{U}_D[q]$.

Let $V:= u_D [q+\delta q]- u_D [q]- \mathcal{U}_{D}[q] (\delta q)= \delta \vartheta_{D} - w,$ where $\delta \vartheta_{D}$ satisfies the problem \eqref{P_d++}. Hence, $V$ satisfies the problem
\begin{eqnarray} \label{P_www}
\left\{
  \begin{array}{rcll}
    \partial_{t} ^{\alpha} V-\Delta V +q V&=& -\delta q \,\delta \vartheta_{D} &\hbox{in}\ Q_T, \\
    V &=&0 &\hbox{on}\  \Sigma_T, \\
    \displaystyle V (\cdot,0)&=&0  &\hbox{in}\ \Omega.
  \end{array}
\right.
\end{eqnarray}
From the weak formulation of the problem \eqref{P_www}, taking $V$ as a test function, and utilizing techniques similar to those used in the proof of \eqref{estimate_A}, one can deduce that
\begin{align} \label{1}
    \|V\|_{B^{\frac{\alpha}{2}}(Q_T)}^2 \leq C(\alpha,\Omega) \int_0^T\int_{\Omega}\Big|\delta q\, \delta \vartheta_{D} V\Big| \dx \dt.
\end{align}
Similar to \eqref{h1++}, one can deduce
\begin{align} \label{12}
\int_0^T\int_{\Omega}\Big|\delta q\, \delta \vartheta_{D} V\Big| \dx \dt
\leq C(\Omega,d)\|\delta q\|_{L^2(\Omega)}\|\delta \vartheta_{D}\|_{L^2(0,T;H^1(\Omega))}\|V\|_{L^2(0,T;H^1(\Omega))}.
\end{align}
Inserting \eqref{12} into \eqref{1}, we get
\begin{align*}
    \|V\|_{B^{\frac{\alpha}{2}}(Q_T)} \leq C(\alpha,\Omega,d)\|\delta q\|_{L^2(\Omega)}\|\delta \vartheta_{D}\|_{L^2(0,T;H^1(\Omega))}.
\end{align*}
Using \eqref{lip_D}, we obtain 
\begin{align*}
    \|V\|_{B^{\frac{\alpha}{2}}(Q_T)} \leq \frac{C(\alpha,\Omega,d)\|\phi\|_{H^{\frac{1}{2}, \frac{\alpha}{4} }(\Sigma_T)}}{\min\{1,\frac{T^{-\alpha}}{2\Gamma(1-\alpha)}\}}\|\delta q\|_{L^2(\Omega)}^2,
\end{align*}
we complete the proof.
\end{proof}

\begin{lemma} \label{taylor_n}
The mapping $q\longmapsto u_N[q]$ from $\Phi_{ad}$ to $B^{\frac{\alpha}{2}}(Q_T)$ is Fr\'echet
differentiable in the sense that for any $\delta q$, $q+\delta q \in \Phi_{ad}$  there exists a bounded linear operator $\mathcal{U}_N[q]: L^2(\Omega) \mapsto B^{\frac{\alpha}{2}}(Q_T) $ such that
\begin{equation*}
    u_N [q+\delta q]=u_N [q]+\mathcal{U}_{N}[q] (\delta q)+o_{B^{\frac{\alpha}{2}}(Q_T)}\Big(\|\delta q\|_{L^2(\Omega)}\Big),
\end{equation*}
where $\mathcal{U}_{N}[q] (\delta q)$ satisfies the sensitive problem:
\begin{eqnarray} \label{P_n++}
\left\{
  \begin{array}{rcll}
    \partial_{t} ^{\alpha} \mathcal{U}_{N}[q] (\delta q)-\Delta \mathcal{U}_{N}[q] (\delta q) +q\, \mathcal{U}_{N}[q] (\delta q)&=& -\delta q\, u_N [q] &\hbox{in}\ Q_T, \\
    \partial_\nu \mathcal{U}_{N}[q] (\delta q) &=&0 &\hbox{on}\  \Sigma_T, \\
    \displaystyle \mathcal{U}_{N} [q](\delta q)&=&0  &\hbox{in}\ \Omega\times\{0\}.
  \end{array}
\right.
\end{eqnarray}
Moreover, there exists a constant $C(\alpha,\Omega,d)>0$ such that
\begin{align} \label{esq_N}
    \|\mathcal{U}_{N}[q] (\delta q)\|_{B^{\frac{\alpha}{2}}(Q_T)} \leq C(\alpha,\Omega,d)\|\varphi\|_{H^{-\frac{1}{2}, - \frac{\alpha}{4} }(\Sigma_T)}\|\delta q\|_{L^2(\Omega)}.
\end{align}
\end{lemma}

\begin{proof}
    Following the same analysis as in the proof of Lemma \ref{taylor_d}, we can establish Lemma \ref{taylor_n}.
\end{proof}

\subsubsection{\bf The Fr\'echet derivative of $\K_\rho$} 

In this section, we analyze the Fr\'echet differentiability of the Kohn-Vogelius type function $q \longmapsto \mathcal{K}(q)$ on the class $\Phi_{ad}$.  To this end, we decompose the Kohn-Vogelius functional $\mathcal{K}_\rho$ into four terms as follows

\begin{align*}
    \mathcal{K}_\rho(q)= \mathcal{K}_{NN}(q) + \mathcal{K}_{DD}(q) -2\mathcal{K}_{ND}(q)+\rho\|q\|^2_{L^2(\Omega)},
\end{align*}
where $\mathcal{K}_{NN}(q)$ is the Neumann term given by
\begin{align*}
    \mathcal{K}_{NN}(q)=\int_0^T\int_\Omega\Big|\nabla u_N[q]\Big|^2\dx\dt +\int_0^T\int_{\Omega}q\Big| u_N[q]\Big|^2 \dx \dt,
\end{align*}
$\mathcal{K}_{DD}(q)$ is the Dirichlet term defined as
\begin{align*}
    \mathcal{K}_{DD}(q)=\int_0^T\int_\Omega\Big|\nabla u_D[q]\Big|^2\dx\dt +\int_0^T\int_{\Omega}q\Big| u_D[q]\Big|^2 \dx \dt,
\end{align*}
and $\mathcal{K}_{ND}(q)$ is the Neumann/Dirichlet term
\begin{align*}
        \mathcal{K}_{ND}(q)=\int_0^T\int_{\Omega}\nabla u_N[q] .\nabla u_D[q] \dx \dt
    +\int_0^T\int_{\Omega}q  u_N[q] u_D[q] \dx \dt.
\end{align*}
Consequently, to compute the Fr\'echet differentiability of $\K_\rho$, we analyze the Fr\'echet differentiability of each of its components with respect to $q$: $\mathcal{K}_{NN}$, $\mathcal{K}_{DD}$, and $\mathcal{K}_{ND}$. To begin, we examine the Fr\'echet differentiability of the function $q \mapsto \mathcal{K}_{NN}(q)$. 

\begin{lemma}\label{lemm10}
The functional $\mathcal{K}_{NN}$ is Fr\'echet differentiable. Moreover, the Fr\'echet derivative of the functional $\mathcal{K}_{NN}$ at $q\in\Phi_{ad}$ in the direction $\delta q$ is given by
\begin{align*}
    \mathcal{K}_{NN}^{\prime}(q)\delta q&=2\int_0^T\int_\Omega\nabla u_N[q]\cdot\nabla \mathcal{U}_N[q](\delta q)\dx\dt +2\int_0^T\int_\Omega q\,u_N[q]\; \mathcal{U}_N[q](\delta q)\dx\dt\\
    &\qquad+\int_0^T\int_{\Omega}  \delta q\; |u_N[q]|^2 \dx.
\end{align*}
\end{lemma}

\begin{proof}
Let $\delta q\in L^2(\Omega)$ be such that $q+\delta q \in \Phi_{ad}$. By direct calculation and utilizing Lemma \ref{taylor_n}, we have
\begin{align}\label{as-rr}
    \mathcal{K}_{NN}(q+\delta q) -\mathcal{K}_{NN}(q)- \mathcal{K}_{NN}^{\prime}(q)\delta q=\sum_{i=1}^4\mathcal{R}_i+o\Big(\|\delta q\|_{L^2(\Omega)}\Big),
\end{align}
where
\begin{align*}
\mathcal{R}_1&=2\int_0^T\int_\Omega \delta q\,u_N[q]\; \mathcal{U}_N[q](\delta q)\dx\dt, \\
\mathcal{R}_2&= 2\Big(\int_0^T\int_\Omega\nabla u_N[q]\cdot\nabla \mathcal{U}_N[q](\delta q)\dx\dt\Big)\times o\Big(\|\delta q\|_{L^2(\Omega)}\Big), \\
\mathcal{R}_3&= 2\Big(\int_0^T\int_\Omega (q+\delta q)u_N[q]\; \mathcal{U}_N[q](\delta q)\dx\dt\Big)\times o\Big(\|\delta q\|_{L^2(\Omega)}\Big), \\
\mathcal{R}_4&= \int_0^T\int_\Omega\Big|\nabla \mathcal{U}_N[q](\delta q)\Big|^2\dx\dt + \int_0^T\int_{\Omega}(q+\delta q)\Big| \mathcal{U}_N[q](\delta q)\Big|^2 \dx \dt.
\end{align*}
Next we give an estimate for each term $\mathcal{R}_i$ ($i\in\{1,2,3,4\}$). Using a similar argument as in the derivation of \eqref{h1++}, and utilizing \eqref{estimation-direct} and \eqref{esq_N}, one can deduce
\begin{align*}
\Big|\int_0^T\int_\Omega \delta q\,  u_N[q]\; \mathcal{U}_N[q](\delta q)\dx\dt\Big| \leq C(\alpha,\Omega,d)\|\varphi\|_{H^{-\frac{1}{2}, - \frac{\alpha}{4} }(\Sigma_T)}^2\|\delta q\|_{L^2(\Omega)}^2.
\end{align*}
This implies that 
\begin{equation}\label{r10}
    \mathcal{R}_1=o\Big(\|\delta q\|_{L^2(\Omega)}\Big).
\end{equation}
Applying the Cauchy-Schwarz inequality and using the inequalities \eqref{estimation-direct} and \eqref{esq_N}, we get
\begin{align}\label{similar}
\Big|\int_0^T\int_\Omega\nabla u_N[q]\cdot\nabla \mathcal{U}_N[q](\delta q)\dx\dt\Big| \leq C(\alpha,\Omega,d)\|\varphi\|_{H^{-\frac{1}{2}, - \frac{\alpha}{4} }(\Sigma_T)}^2\|\delta q\|_{L^2(\Omega)}.
\end{align}
Consequently, 
\begin{align} \label{c_0}
     \mathcal{R}_2= o\Big(\|\delta q\|_{L^2(\Omega)}\Big).
\end{align}

Now, we estimate the term $ \mathcal{R}_3$. Utilizing the fact that $q\in\Phi_{ad}$ (i.e., $q(x)\leq c^\prime$ for all $x\in\Omega$) and employing the same technique described in the proof of \eqref{similar}, one can demonstrate that
\begin{align}
\begin{split}\label{rrrrt}
\Big|\int_0^T\int_\Omega q\, u_N[q]\; \mathcal{U}_N[q](\delta q)\dx\dt\Big|&\leq c^\prime \int_0^T\int_\Omega \Big|u_N[q]\, \mathcal{U}_N[q](\delta q)\Big|\dx\dt\\
&\leq c^\prime C(\alpha,\Omega,d)\|\varphi\|_{H^{-\frac{1}{2}, - \frac{\alpha}{4} }(\Sigma_T)}^2\|\delta q\|_{L^2(\Omega)}.
\end{split}
\end{align}
On the other hand, we can observe that $\mathcal{R}_3=\Big(\mathcal{R}_1+\displaystyle\int_0^T\int_\Omega q\,u_N[q]\; \mathcal{U}_N[q](\delta q)\dx\dt\Big)\times o\Big(\|\delta q\|_{L^2(\Omega)}\Big).$ Therefore, from \eqref{r10} and \eqref{rrrrt}, we deduce
\begin{eqnarray}\label{rr3}
     \mathcal{R}_3=o\Big(\|\delta q\|_{L^2(\Omega)}\Big).
\end{eqnarray}
Next, we estimate the term $\mathcal{R}_4.$ Since $q+\delta q\in\Phi_{ad}$, we have $|q+\delta q|\leq c^\prime.$ From this and Lemma \ref{taylor_n},
\begin{align*}
  \Big|  \mathcal{R}_4\Big|\leq \max\{1,c^\prime \} \|\mathcal{U}_{N}[q] (\delta q)\|_{B^{\frac{\alpha}{2}}(Q_T)}^2\leq C(\alpha,\Omega,d)^2\|\varphi\|_{H^{-\frac{1}{2}, - \frac{\alpha}{4} }(\Sigma_T)}^2\|\delta q\|_{L^2(\Omega)}^2.
\end{align*}
Therefore,
\begin{eqnarray}\label{rrlastly}
    \mathcal{R}_4=o\Big(\|\delta q\|_{L^2(\Omega)}\Big).
\end{eqnarray}

Inserting \eqref{r10}, \eqref{c_0}, \eqref{rr3}, and \eqref{rrlastly} into \eqref{as-rr}, we get 
\begin{align*}
    \mathcal{K}_{NN}(q+\delta q) -\mathcal{K}_{NN}(q)- \mathcal{K}_{NN}^{\prime}(q)\delta q=o\Big(\|\delta q\|_{L^2(\Omega)}\Big).
\end{align*}
Finally, it is easy to verify that the linear mapping $\delta q\longmapsto\mathcal{K}_{NN}^{\prime}(q)\delta q$ is continuous on $L^2(\Omega)$. This completes the proof of Lemma \ref{lemm10}.
\end{proof}

Utilizing the same analysis as described in the proof of Lemma \ref{lemm10}, one can establish the Fr\'echet differentiability of the functions $q\longmapsto\mathcal{K}_{DD}(q)$ and $q\longmapsto\mathcal{K}_{ND}(q)$, as detailed in the following lemmas.

\begin{lemma}\label{lemm100}
The functional $\mathcal{K}_{DD}$ is Fr\'echet differentiable. In addition, the Fr\'echet derivative of the functional $\mathcal{K}_{DD}$ at $q\in\Phi_{ad}$ in the direction $\delta q$ is given by
\begin{align*}
    \mathcal{K}_{DD}^{\prime}(q)\delta q&=2\int_0^T\int_\Omega\nabla u_D[q]\cdot\nabla \mathcal{U}_D[q](\delta q)\dx\dt +2\int_0^T\int_\Omega q\,u_D[q]\; \mathcal{U}_D[q](\delta q)\dx\dt\\
    &\qquad+\int_0^T\int_{\Omega}  \delta q\; |u_D[q]|^2 \dx.
\end{align*}
\end{lemma}

\begin{lemma}\label{lemm1000}
The functional $\mathcal{K}_{ND}$ is Fr\'echet differentiable. Moreover, the Fr\'echet derivative of the functional $\mathcal{K}_{ND}$ at $q\in\Phi_{ad}$ in the direction $\delta q$ is given by
\begin{align*}
    \mathcal{K}_{ND}^{\prime}(q)\delta q&=\int_0^T\int_{\Omega}\nabla \mathcal{U}_N[q](\delta q)\cdot \nabla u_D[q] \dx \dt + \int_0^T\int_{\Omega}  \nabla u_N[q]\cdot\nabla \mathcal{U}_D[q](\delta q) \dx \dt \\
  &\qquad+\int_0^T\int_{\Omega}  q\; \mathcal{U}_N[q](\delta q)\;u_D[q] \dx \dt +\int_0^T\int_{\Omega}  q\; u_N[q]\;\mathcal{U}_D[q](\delta q) \dx \dt\\
  &\qquad +\int_0^T\int_{\Omega} \delta q\;  u_N[q] u_D[q] \dx \dt.
\end{align*}
\end{lemma}

Referring to Lemmas \ref{lemm10}, \ref{lemm100}, and \ref{lemm1000}, it can be concluded that the Kohn-Vogelius functional $\K_\rho$ is Fr\'echet differentiable.  Furthermore, we have the following theorem.

\begin{theorem}\label{K-D-Rho}
The Fr\'echet derivative of $\mathcal{K}_{\rho}$ at $q\in\Phi_{ad}$ in the direction $\delta q$ is:
\begin{equation}\label{frechet-K-rho}
    \K^\prime_\rho(q)\delta q=\mathcal{K}_{NN}^{\prime}(q)\delta q+\mathcal{K}_{DD}^{\prime}(q)\delta q-2\mathcal{K}_{ND}^{\prime}(q)\delta q+2\rho\int_{\Omega} q\,\delta q\;\dx.
\end{equation}
\end{theorem}

\subsubsection{\bf Uniqueness} 
Here we establish the strict convexity of the Kohn-Vogelius functional $\K_\rho$ and demonstrate the uniqueness of its minimizer within the class $\Phi_{ad}$. This convexity result is presented in Theorem \ref{new-uniqueness-condition}. The proof relies on a classical property stating that the functional $\K_\rho$ is strictly convex if the condition $\mathcal{K}_\rho(q+\delta q)-\mathcal{K}_\rho(q) -\mathcal{K}^\prime_\rho(q)\delta q>0$ holds (cf. \cite[Lemma 15.4]{IntroductionKO}), where $\mathcal{K}^\prime_\rho(q)$ denotes the Fr\'echet derivative of $\K_\rho$.

\begin{theorem}[Uniqueness]\label{new-uniqueness-condition}
    For any $\rho>0$, there exists a positive constant $M$ independent of $q\in\Phi_{ad}$ such that for $M<\rho$, $\mathcal{K}_\rho(q)$ is strictly convex.
\end{theorem}

\begin{proof}
    Let $\delta q\in L^2(\Omega)$ be such that $q+\delta q \in \Phi_{ad}$. Then, we have
\begin{align*} 
    \mathcal{K}_\rho(q+\delta q)-\mathcal{K}_\rho(q) -\mathcal{K}^\prime_\rho(q)\delta q=K_1+K_2-2K_3+ \rho\|\delta q\|^2_{L^2(\Omega)},
\end{align*}
where
\begin{align*}
   {K}_1&= \mathcal{K}_{NN}(q+\delta q)-\mathcal{K}_{NN}(q) -\mathcal{K}_{NN}^\prime(q)\delta q, \\
    {K}_2 &= \mathcal{K}_{DD}(q+\delta q)-\mathcal{K}_{DD}(q) -\mathcal{K}_{DD}^\prime(q)\delta q, \\
   {K}_3&=\mathcal{K}_{ND}(q+\delta q)-\mathcal{K}_{ND}(q) -\mathcal{K}_{ND}^\prime(q)\delta q.
\end{align*}
Now, we proceed to analyze the lower bound of each term $K_1$, $K_2$, and $K_3$. Conducting straightforward calculations, we obtain the following expression
\begin{align*}
    K_1 = \sum_{i=1}^3 r_i,
\end{align*}
where
\begin{align*} 
   r_1   &= \int_0^T\int_{\Omega} \Big|\nabla\Big( u_N[q+\delta q]- u_N[q]\Big)\Big|^2\dx\dt +\int_0^T\int_{\Omega} (q+\delta q)\Big| u_N[q+\delta q]- u_N[q]\Big|^2\dx\dt,\\
   r_2  &=2\int_0^T\int_{\Omega} \delta q u_N[q] \Big( u_N[q+\delta q]- u_N[q]\Big) \dx\dt,\\
    r_3&= 2\int_0^T\int_{\Omega} \nabla u_N[q] . \nabla\Big( u_N[q+\delta q]- u_N[q]-\mathcal{U}_N[q](\delta q)\Big) \dx\dt\\
    &\qquad + 2\int_0^T\int_{\Omega} q u_N[q] \Big( u_N[q+\delta q]- u_N[q]-\mathcal{U}_N[q](\delta q)\Big) \dx\dt.
\end{align*}
To control the term $K_1$ by $\|\delta q \|_{L^{2}(\Omega)}^2$, we examine the components $r_1$, $r_2$, and $r_3$. Firstly, since $q+\delta q \in \Phi_{ad}$, it is clear that $r_1\geq 0$, hence
\begin{align} \label{r_1}
    K_1\geq r_2+r_3.
\end{align}
Next, we examine $r_2$. Applying the H\"older inequality, we get
\begin{align*} 
\int_\Omega \Big|\delta q u_N[q]\, \Big( u_N[q+\delta q]- u_N[q]\Big)\Big|\dx\leq  \|\delta q \|_{L^{2}(\Omega)}\| u_N[q]\|_{L^{4}(\Omega)}\| u_N[q+\delta q]- u_N[q]\|_{L^{4}(\Omega)}.
\end{align*}
By integrating over time and applying the Cauchy-Schwarz inequality along with Sobolev embedding theorem, we deduce
\begin{align*} 
\int_0^T\int_\Omega \Big|\delta q u_N[q]\, \Big( &u_N[q+\delta q]- u_N[q]\Big)\Big|\dx\dt\\&\leq C(\Omega,d) \|\delta q \|_{L^{2}(\Omega)}\| u_N[q]\|_{L^2(0,T;H^{1}(\Omega))}\| u_N[q+\delta q]- u_N[q]\|_{L^2(0,T;H^{1}(\Omega))}.
\end{align*}
Leveraging Lemma \ref{lipshitz} and inequality \eqref{estimation-direct}, we get
\begin{align}
    \label{r_2}
    \Big|r_2\Big| \leq \frac{2C(\Omega,d)^2 C_N^2\|\varphi\|^2_{H^{-\frac{1}{2}, -\frac{\alpha}{4} }(\Sigma_T)}}{\min\{1,\frac{T^{-\alpha}}{2\Gamma(1-\alpha)}\}} \|\delta q \|_{L^{2}(\Omega)}^2.
\end{align}
We now turn our attention to examining the last term $r_3$. By applying the same strategy developed to deduce \eqref{lip_N}, we can conclude that
\begin{align*}
    \|u_N[q+\delta q]- u_N[q]-\mathcal{U}_N[q](\delta q)\|_{L^2(0,T;H^1(\Omega))} \leq \frac{C(\Omega,d)^2C_N\|\varphi\|_{H^{-\frac{1}{2}, -\frac{\alpha}{4} }(\Sigma_T)}}{\min\{1,\frac{T^{-\alpha}}{2\Gamma(1-\alpha)}\}^2}\|\delta q\|_{L^2(\Omega)}^2.
\end{align*}
Using the above inequality and the Cauchy-Schwarz inequality, we can deduce that
\begin{align*}
    \Big|\int_0^T\int_{\Omega} \nabla u_N[q] . \nabla\Big( u_N[q+\delta q]- &u_N[q]-\mathcal{U}_N[q](\delta q)\Big) \dx\dt\Big|\\&\leq \frac{C(\Omega,d)^2C_N\|\varphi\|_{H^{-\frac{1}{2}, -\frac{\alpha}{4} }(\Sigma_T)}}{\min\{1,\frac{T^{-\alpha}}{2\Gamma(1-\alpha)}\}^2}\|\nabla u_N[q]\|_{L^2(0,T;L^2(\Omega))}\|\delta q\|_{L^2(\Omega)}^2.
\end{align*}
Based on the estimate provided in \eqref{estimation-direct}, we derive
\begin{align}
\begin{split}
    \Big|\int_0^T\int_{\Omega} \nabla u_N[q] . \nabla\Big( u_N[q+\delta q]- u_N[q]&-\mathcal{U}_N[q](\delta q)\Big) \dx\dt\Big|\\
    &\leq \frac{C(\Omega,d)^2C_N^2\|\varphi\|^2_{H^{-\frac{1}{2}, -\frac{\alpha}{4} }(\Sigma_T)}}{\min\{1,\frac{T^{-\alpha}}{2\Gamma(1-\alpha)}\}^2}\|\delta q\|_{L^2(\Omega)}^2.
\label{nj_3}
\end{split}
\end{align}
Using the fact that $0<q(x)\leq c^\prime$ for all $x\in\Omega$,  and similarly to \eqref{nj_3}, we get
\begin{align}\notag
    \Big|\int_0^T\int_{\Omega} q u_N[q] \Big( u_N[q+\delta q]- u_N[q]&-\mathcal{U}_N[q](\delta q)\Big) \dx\dt\Big|\\&\leq\frac{c^\prime C(\Omega,d)^2C_N^2\|\varphi\|^2_{H^{-\frac{1}{2}, -\frac{\alpha}{4} }(\Sigma_T)}}{\min\{1,\frac{T^{-\alpha}}{2\Gamma(1-\alpha)}\}^2}\|\delta q\|_{L^2(\Omega)}^2.\label{nj_3+}
\end{align}
Combining \eqref{nj_3} and \eqref{nj_3+}, we get
\begin{align} \label{r_3}
    \Big|r_3\Big| \leq \frac{ C(\Omega,d)^2C_N^2\|\varphi\|^2_{H^{-\frac{1}{2}, -\frac{\alpha}{4} }(\Sigma_T)}}{\min\{1,\frac{T^{-\alpha}}{2\Gamma(1-\alpha)}\}^2}\Big(2+2c^\prime\Big)\|\delta q\|_{L^2(\Omega)}^2.
\end{align}
Using \eqref{r_1}, \eqref{r_2} and \eqref{r_3}, we can establish the following inequality
\begin{equation} \label{L_1}
    K_1 \geq -\frac{C(\Omega,d)^2}{\min\{1,\frac{T^{-\alpha}}{2\Gamma(1-\alpha)}\}^2}M_1  \|\delta q\|_{L^2(\Omega)}^2,
\end{equation}
where 
\begin{align*}
    M_1=  2(2+c^\prime) C_N^2\|\varphi\|_{H^{-\frac{1}{2}, -\frac{\alpha}{4} }(\Sigma_T)}^2.
\end{align*}
In the same way, one can obtain
\begin{align}
    \label{L_2}
     K_2 \geq -\frac{C(\Omega,d)^2 }{\min\{1,\frac{T^{-\alpha}}{2\Gamma(1-\alpha)}\}^2}  M_2\|\delta q\|_{L^2(\Omega)}^2,
\end{align}
where 
\begin{align*}
    M_2=  2(2+c^\prime) C_N^2\|\varphi\|_{H^{-\frac{1}{2}, -\frac{\alpha}{4} }(\Sigma_T)}^2.
\end{align*}
     \begin{align}
    \label{L_3}
    \Big|K_3\Big|  &\leq\frac{C(\Omega,d)^2}{\min\{1,\frac{T^{-\alpha}}{2\Gamma(1-\alpha)}\}^2} M_3 \|\delta q\|_{L^2(\Omega)}^2,
\end{align}
where
\begin{align*}
    M_3= 3(c^\prime+2)C_D C_N\|\varphi\|_{H^{\frac{1}{2}, \frac{\alpha}{4} }(\Sigma_T)}\|\phi\|_{H^{-\frac{1}{2}, -\frac{\alpha}{4} }(\Sigma_T)}.
\end{align*}

Finally, from \eqref{L_1}, \eqref{L_2}, and  \eqref{L_3}, one can readily deduce that 
\begin{align*}
    \mathcal{K}_\rho(q+\delta q)-\mathcal{K}_\rho(q) -\mathcal{K}^\prime_\rho(q)\delta q \geq(\rho -M) \|\delta q\|_{L^2(\Omega)}^2
\end{align*}
where
\begin{align} \label{constant_M}
    M=\frac{C(\Omega,d)^2}{\min\{1,\frac{T^{-\alpha}}{2\Gamma(1-\alpha)}\}^2}\Big(M_1+M_2+2M_3\Big).
\end{align}
If we let $\rho>M$, then $\mathcal{K}_\rho(q+\delta q)-\mathcal{K}_\rho(q) -\mathcal{K}^\prime_\rho(q)\delta q>0,$ implying that the functional $q\longmapsto\mathcal{K}_\rho(q)$ is strictly convex.
\end{proof}

\begin{remark}\label{con-sta}We notice that given $\rho>0$, we can choose $\|\varphi\|_{H^{\frac{1}{2}, \frac{\alpha}{4} }(\Sigma_T)}$ and $\|\phi\|_{H^{-\frac{1}{2}, -\frac{\alpha}{4} }(\Sigma_T)}$ (sufficiently small)  such that
condition $\rho>M$ holds. Furthermore, in this condition (i.e., $\rho>M$), one can also let $\varphi\in H^{\frac{1}{2}, \frac{\alpha}{4} }(\Sigma_T)$ and $\phi\in H^{-\frac{1}{2}, -\frac{\alpha}{4} }(\Sigma_T)$ be arbitrary and select the penalization parameter $\rho > 0$ such that the condition is satisfied.
\end{remark}

Next, we want to investigate the stable dependence of the minimizers of the Kohn-Vogelius functional with respect to the perturbation in observation data.

\section{Stability and convergence of the Kohn and Vogelius method}\label{sec:stability}

The inverse problem \eqref{inverse} may lack stability, but in this section, we establish the stability of the optimization problem \eqref{stable}. This stability arises because the Kohn and Vogelius cost function incorporates solutions computed from both the prescribed and measured data, extending beyond just the boundary. Consequently, it prohibits unstable behavior away from the prescribed part of the boundary, which could significantly impact the cost function. To address this, we introduce a sequence of ``measurements", denoted as $\phi^n$ in ${}_0H^{\frac{1}{2},\frac{\alpha}{4} }\left(\Sigma_{T}\right)$, with the property:
\begin{align}\label{lim-seq-sat}
    \|\phi^n - \phi\|_{H^{\frac{1}{2},\frac{\alpha}{4} }\left(\Sigma_{T}\right)} \longrightarrow 0\;\; \hbox{as}\; n\rightarrow\infty.
\end{align}
Furthermore, for simplicity in further analysis, we define $u_D[q,\phi^n]$ as the solution to the problem \eqref{PD}, where $q$ serves as the potential, and $\phi^n$ acts as the Dirichlet boundary data on $\Sigma_T.$ Additionally, for each $n$, we denote by $q^n\in\Phi_{ad}$ the solution to the perturbed optimization problem: 
\begin{equation*} 
 (\mathcal{O}\mathcal{P}^n_\rho):\;\;\;   \displaystyle\operatorname*{Minimize}_{q\in\Phi_{ad}}\
\mathcal{K}_\rho^n( q):=\mathcal{N}(q,\phi^n)+\rho\|q\|_{L^2(\Omega)}^2.
\end{equation*}
Here, the function $\mathcal{N}$ is defined as follows:
$$\begin{array} {rclll}
 \mathcal{N}:
 &\Phi_{ad}\times {}_0H^{\frac{1}{2},\frac{\alpha}{4} }\left(\Sigma_{T}\right)& \longrightarrow \mathbb{R} \\
 &(q,f)& \longmapsto \displaystyle \int_0^T \int_{\Omega}\Big|\nabla\Big(
 u_N[q]- u_D[q,f]\Big)\Big|^2\ \dx \dt +\int_0^T\int_{\Omega}q\Big| u_N[q]- u_D[q,f]\Big|^2 \dx \dt. 
\end{array}$$
Recall that $u_D[q,f]$ represents the solution of the problem \eqref{PD} with $q$ as a potential and $f$ as Dirichlet boundary data on $\Sigma_T$. In addition, $u_N[q]$ denotes the usual solution of the problem \eqref{PN} with $q$ as a potential. 

Two main challenges arise in this setting:    
\begin{itemize}
    \item {\bf Existence of perturbed solutions}: When the Dirichlet data $\phi$ is replaced by $\phi^n$, the pair $(q^*, \phi^n)$ is no longer compatible, meaning that the inverse problem \eqref{inverse} may not have a solution. Consequently, the perturbed optimization problem $(\mathcal{O}\mathcal{P}_\rho^n)$ is no longer equivalent to the inverse problem. Nevertheless, we are interested in the solutions $q^n$ of the optimization problem itself.
    
    \item {\bf Convergence of perturbed solutions}: Assuming that $(\mathcal{O}\mathcal{P}_\rho^n)$ has solutions, we aim to establish whether $q^n$ converges, in a suitable function space, to the solution of the inverse problem corresponding to the exact measurements $\phi$.
\end{itemize}

{\color{black}
\subsection{\bf Stability}
Here, we justify the stability of the optimization problem $(\mathcal{O}\mathcal{P}_\rho^n)$ with respect to perturbations in the observation data $\phi$.

\begin{theorem} Let $\rho>0$ satisfy the uniqueness condition in Theorem \ref{new-uniqueness-condition} and let $\{\phi^n\} \subset {}_0H^{\frac{1}{2},\frac{\alpha}{4}}(\Sigma_T)$ be a sequence of perturbed measurements satisfying \eqref{lim-seq-sat}. Let $\{ q^n \}\subset\Phi_{ad}$ be a sequence of minimizers of problems $(\mathcal{O}\mathcal{P}_{\rho}^n)$ for $n=1,2,\cdots$. Then there exists a subsequence $\{ q^{n_\ell} \}$ of $\{ q^n \}$, such that 
\begin{equation*}
q^{n_\ell}\longrightarrow q^0\; \hbox{strongly in}\; L^2(\Omega)\; \hbox{as}\; \ell\rightarrow\infty,
\end{equation*}
where $q^0\in\mathcal{Q}$ is the unique minimizer of the optimization problem \eqref{stable}.
\end{theorem}

\begin{proof}
The unique existence of each $q^n\in\Phi_{ad}$ is ensured by Theorems \ref{theorem-existence} and \ref{new-uniqueness-condition}. Employing a similar technique as in the proof of the convergence result \eqref{cv-0}, we establish that the sequence $q^n$ strongly converges to $q^0\in\Phi_{ad}$ in $L^2(\Omega)$. In other words, we have
\begin{equation}\label{cvv-st}
q^n\longrightarrow q^0\; \hbox{strongly in}\; L^2(\Omega)\; \hbox{as}\; n\rightarrow\infty.
\end{equation}

Now it suffices to show that $q^0$ is indeed the unique minimizer of \eqref{stable}. From \eqref{cvv-st} and Lemma \ref{conv_N}, one can deduce that 
\begin{equation} \label{convergancefaible-NNN}
    u_N[q^n]\longrightarrow u_N[q^0]\; \hbox{ strongly in }\; L^2(0,T;H^1(\Omega))\; \hbox{ as }\; n\to\infty, 
\end{equation}
where $u_N[q^0]$ is the solution to \eqref{PN} with $q=q^0$, and $u_N[q^n]$ is the solution to \eqref{PN} with $q=q^n$.

Next, we prove that
\begin{equation}\label{convergancefaible-DDD}
    u_D[q^n, \phi^n] \longrightarrow u_D[q^0, \phi] \quad \text{stongly in} \quad L^2(0, T; L^2(\Omega)) \;\; \text{as} \;\; n \longrightarrow \infty.
\end{equation}
To establish the above convergence result, we define $\vartheta_n = u_D[q^n, \phi^n] - u_D[q^0, \phi]$. It is straightforward to deduce that $\vartheta_n$ is the solution to
\begin{eqnarray*} 
\left\{
  \begin{array}{rcll}
    \partial_t^\alpha \vartheta_n - \Delta \vartheta_n + q^0 \vartheta_n &=& -(q^n - q^0) u_D[q^n, \phi^n] & \text{in} \; Q_T, \\
    \vartheta_n &=& \phi^n - \phi & \text{on} \; \Sigma_T, \\
    \vartheta_n(\cdot, 0) &=& 0 & \text{in} \; \Omega.
  \end{array}
\right.
\end{eqnarray*}
We split $\vartheta_n$ into $\vartheta_n^1 + \vartheta_n^2$ such that
\begin{eqnarray} \label{P_d30}
\left\{
  \begin{array}{rcll}
    \partial_t^\alpha \vartheta_n^1 - \Delta \vartheta_n^1 + q^0 \vartheta_n^1 &=& -(q^n - q^0) u_D[q^n, \phi^n] & \text{in} \; Q_T, \\
    \vartheta_n^1 &=& 0 & \text{on} \; \Sigma_T, \\
    \vartheta_n^1(\cdot, 0) &=& 0 & \text{in} \; \Omega.
  \end{array}
\right.
\end{eqnarray}
\begin{eqnarray*} 
\left\{
  \begin{array}{rcll}
    \partial_t^\alpha \vartheta_n^2 - \Delta \vartheta_n^2 + q^0 \vartheta_n^2 &=& 0 & \text{in} \; Q_T, \\
    \vartheta_n^2 &=& \phi^n - \phi & \text{on} \; \Sigma_T, \\
    \vartheta_n^2(\cdot, 0) &=& 0 & \text{in} \; \Omega.
  \end{array}
\right.
\end{eqnarray*}
By taking $\vartheta_n^1$ as a test function in the weak formulation of \eqref{P_d30}, we get
\begin{align*}
    \int_0^T \int_\Omega \Big(\partial_t^\alpha \vartheta_n^1\Big) \vartheta_n^1 \dx \dt + \int_0^T \int_\Omega \Big|\nabla \vartheta_n^1\Big|^2 \dx \dt + \int_0^T \int_\Omega q^0 \Big|\vartheta_n^1\Big|^2 \dx \dt \\
    = \int_0^T \int_\Omega \Big(q^n - q^0\Big) \vartheta_n^1 u_D[q^n, \phi^n] \, \dx \dt.
\end{align*}
Employing a similar technique as in the proof of inequality \eqref{h1+}, we obtain
\begin{equation} \label{dd_1}
    \|\vartheta_n^1\|_{L^2(0, T; H^1(\Omega))}^2 \leq C(\Omega) \int_0^T \int_\Omega \Big|\Big(q^n - q^0\Big) \vartheta_n^1 u_D[q^n, \phi^n]\Big| \dx \dt.
\end{equation}
Similar to \eqref{h11}, one can deduce
\begin{equation}\label{dd-ff}
    \int_0^T \int_\Omega \Big|\Big(q^n - q^0\Big) \vartheta_n^1 u_D[q^n, \phi^n]\Big| \dx \dt \leq C(\Omega) \|q^n - q^0\|_{L^2(\Omega)} \|\vartheta_n^1\|_{L^2(0, T; H^1(\Omega))} \| u_D[q^n, \phi^n] \|_{L^2(0, T; H^1(\Omega))}.
\end{equation}
From the inequality \eqref{estimate-Direchlet}, we obtain
\begin{align*}
    \| u_D[q^n, \phi^n] \|_{L^2(0, T; H^1(\Omega))} &\leq C(\alpha, T, \Omega) \|\phi^n\|_{H^{\frac{1}{2}, \frac{\alpha}{4}} (\Sigma_T)} \\
    &\leq C(\alpha, T, \Omega) \Big(\|\phi^n - \phi\|_{H^{\frac{1}{2}, \frac{\alpha}{4}} (\Sigma_T)} + \|\phi\|_{H^{\frac{1}{2}, \frac{\alpha}{4}} (\Sigma_T)}\Big).
\end{align*}
Using the fact that \(\phi^n \to \phi\) in \({}_0 H^{\frac{1}{2}, \frac{\alpha}{4}} (\Sigma_T)\) as \(n \to \infty\), we obtain
\begin{align}\label{insert}
    \| u_D[q^n, \phi^n] \|_{L^2(0, T; H^1(\Omega))} \leq C(\alpha, T, \Omega, \phi).
\end{align}
Inserting \eqref{insert} into \eqref{dd-ff}, we obtain
\begin{equation}\label{dd-ff2}
    \int_0^T \int_\Omega \Big|\Big(q^n - q^0\Big) \vartheta_n^1 u_D[q^n, \phi^n]\Big| \dx \dt \leq C(\alpha, T, \Omega, \phi) \|q^n - q^0\|_{L^2(\Omega)} \|\vartheta_n^1\|_{L^2(0, T; H^1(\Omega))}.
\end{equation}
Inserting \eqref{dd-ff2} into \eqref{dd_1}, we get
\begin{equation}\label{dd-ff3}
    \|\vartheta_n^1\|_{L^2(0, T; H^1(\Omega))} \leq C(\alpha, T, \Omega,\phi) \|q^n - q^0\|_{L^2(\Omega)}.
\end{equation}
Moreover, from the inequality \eqref{estimate-Direchlet}, we obtain
\begin{align} \label{dd_ff4}
    \|\vartheta_n^2\|_{L^2(0, T; H^1(\Omega))} \leq C(\alpha, T, \Omega) \|\phi^n - \phi\|_{L^2(0, T; L^2(\partial\Omega))}.
\end{align}
Combining \eqref{dd-ff3} with \eqref{dd_ff4}, and using the convergence result \eqref{cvv-st} and the fact that $\phi^n \longrightarrow \phi$ in ${}_0 H^{\frac{1}{2}, \frac{\alpha}{4}} (\Sigma_T)$ as $n \longrightarrow \infty$, we get
\begin{equation*}
    \|\vartheta_n\|_{L^2(0, T; H^1(\Omega))} \longrightarrow 0 \;\; \text{as} \;\; n \longrightarrow \infty.
\end{equation*}

On the other hand, following the same strategy developed in the proof of \eqref{c2}, we obtain
\begin{align} \label{c2++}
    \lim_{n \rightarrow \infty} \int_0^T \int_\Omega q^n \Big| u_N[q^n] - u_D[q^n, \phi^n] \Big|^2 \dx \dt = \int_0^T \int_\Omega q^0 \Big| u_N[q^0] - u_D[q^0, \phi] \Big|^2 \dx \dt.
\end{align}

Consequently, by combining the convergence results \eqref{cvv-st}, \eqref{convergancefaible-NNN}, \eqref{convergancefaible-DDD}, and \eqref{c2++} along with the lower semi-continuity of the $L^2$-norm, we conclude

\begin{align*}
\mathcal{K}_\rho(q^0) &=\int_0^T \int_{\Omega}\Big|\nabla\Big(
 u_N[q^0]- u_D[q^0]\Big)\Big|^2\ \dx \dt + \int_0^T \int_{\Omega}q^0\Big|
 u_N[q^0]- u_D[q^0]\Big|^2\ \dx\dt\\
&\;\;\quad +\rho\int_0^T \int_{\Omega}\Big|q^0\Big|^2\ \dx\dt\\
&\leq\lim_{n\rightarrow\infty}\inf\int_0^T \int_{\Omega}\Big|\nabla\Big(
 u_N[q^n]- u_D[q^n,\phi^n]\Big)\Big|^2\ \dx \dt \\
&\;\;\quad + \lim_{n\rightarrow\infty}\inf\int_0^T \int_{\Omega}q^n\Big|
 u_N[q^n]- u_D[q^n,\phi^n]\Big|^2\ \dx\dt+ \rho\lim_{n\rightarrow\infty}\inf\int_0^T \int_{\Omega}\Big|q^n\Big|^2\ \dx\dt\\
&\leq \lim_{n\rightarrow\infty}\inf \mathcal{K}^n_\rho(q^n).
\end{align*}
Moreover, since \(\{ q^n \} \subset \Phi_{ad}\) is a sequence of minimizers of problems \((\mathcal{O}\mathcal{P}_{\rho}^n)\), and by using \(u_D[q,\phi^n] \longrightarrow u_D[q,\phi]\) in \(L^2(0, T; H^1(\Omega))\) as \(n \longrightarrow \infty\), we obtain
\begin{align*} 
\mathcal{K}_\rho(q^0)&\leq\lim_{n\rightarrow\infty}\Big(\int_0^T \int_{\Omega}\Big|\nabla\Big(
 u_N[q]- u_D[q,\phi^n]\Big)\Big|^2\ \dx \dt + \int_0^T \int_{\Omega}q\Big|
 u_N[q]- u_D[q,\phi^n]\Big|^2\ \dx\dt \\
 &\;\; \quad + \rho\int_0^T \int_{\Omega}\Big|q\Big|^2\ \dx\dt \Big)\\
&=\mathcal{K}_\rho(q), \quad \;\;\; \hbox{for any} \; q \in \Phi_{ad},
\end{align*}
which verifies that $q^0$ is the minimizer of \eqref{stable}.
\end{proof}

\subsection{\bf Convergence of the Kohn-Vogelius method}
Here, we demonstrate that the solution of the minimization problem $(\mathcal{O}\mathcal{P}_{\rho}^n)$ converges to the solution of the inverse problem as $\rho\to0.$

\begin{lemma} \label{leM}
Let $ \phi^n \in {}_0H^{\frac{1}{2},\frac{\alpha}{4} }\left(\Sigma_{T}\right)$ be a sequence of observation data such that $\displaystyle\lim_{n \to \infty}\| \phi^n -\phi \|_{H^{\frac{1}{2},\frac{\alpha}{4} }\left(\Sigma_{T}\right)} = 0.$ Let $q^n\in\Phi_{ad}$ be a solution to the optimization problem $(\mathcal{O}\mathcal{P}_{\rho_n}^n);$ i.e.,
$$\inf_{q\in\Phi_{ad}}\K_{\rho_n}^n(q)=\mathcal{N}(q^n,\phi^n)+\rho_n\|q^n\|_{L^2(\Omega)}^2,$$ 
where $\rho_n>0$ for all $n\in\mathbb{N}$ and $\rho_n\longrightarrow 0$ as $n\to\infty$. Consequently, it follows that:
$$\mathcal{N}(q^n,\phi^n)
\longrightarrow 0 \; \; \hbox{  as  } \; n \rightarrow \infty.
$$ 
\end{lemma}
\begin{proof} Let $q^*$ be a solution of the inverse problem \eqref{inverse}. As $q^n$ is a minimizer of the optimization problem $(\mathcal{O}\mathcal{P}^n_{\rho_n}),$ we have by the minimization property
\begin{align} \label{PA}
    0\leq \mathcal{N}(q^n,\phi^n) \leq \mathcal{N}(q^*,\phi^n)+\rho_n\|q^*\|_{L^2(\Omega)}^2.
\end{align}
Let us prove now that $\displaystyle\lim_{n \to \infty}\mathcal{N}(q^n,\phi^n) =0.$ To obtain this result, let us define the following function
$$w_n:=  u_D[q^*,\phi^n]-  u_N[q^*]=  u_D[q^*,\phi^n]-  u_D[q^*,\phi].
$$
Here, $u_D[q^*,\phi^n]$ denotes the solution to \eqref{PD} with $q=q^*$ and $\phi=\phi^n$, $u_N[q^*]$ represents the solution to \eqref{PN} with $q=q^*$, and $u_D[q^*,\phi]$ signifies the solution to \eqref{PD} with $q=q^*$.  We can easily deduce that $w_n$ is a solution to
\begin{eqnarray*}
\left\{
  \begin{array}{rcll}
    \partial_{t} ^{\alpha} w_n-\Delta w_n +q^*w_n&=&0 &\hbox{in}\ Q_T, \\
    \displaystyle w_n&=&\phi^n-\phi &\hbox{on}\  \Sigma_T, \\
    \displaystyle w_n(.,0)&=&0  &\hbox{in}\ \Omega.
  \end{array}
\right.
\end{eqnarray*}
From \eqref{estimate-Direchlet}, we can check that there exists a constant $C_D>0$ (independent of $n$ and $\phi$) such that
\begin{equation*}
    \|w_n\|_{B^{\frac{\alpha}{2}}(Q_T)}\leq C_D \| \phi^n -\phi \|_{H^{\frac{1}{2}, \frac{\alpha}{4} }(\Sigma_T)}.
\end{equation*}
and since $\displaystyle\lim_{n \to \infty}\| \phi^n -\phi \|_{H^{\frac{1}{2}, \frac{\alpha}{4} }(\Sigma_T)} = 0$, we get
\begin{align*}
    \displaystyle\lim_{n \to \infty} \|w_n\|_{B^{\frac{\alpha}{2}}(Q_T)}
 =0.
\end{align*}
Consequently, 
\begin{align} \label{nnw}
    \displaystyle\lim_{n \to \infty} \ds\| w_n\|_{L^2(0,T;H^1(\Omega))}^2
 =0.
\end{align}
On the other hand, we have
\begin{equation*}
  \begin{array}{rll}
    \mathcal{N}(q^*,\phi^n) &= \ds\int_0^T \int_{\Omega}\Big|\nabla\Big(
 u_N[q^*]- u_D[q^*,\phi^n]\Big)\Big|^2\ \dx \dt \\&\qquad+ \displaystyle\int_0^T \int_{\Omega}q^*\Big|
 u_N[q^*]- u_D[q^*,\phi^n]\Big|^2\ \dx \dt\\
\ &\leq \ds\|\nabla\Big( u_N[q^*]- u_D[q^*,\phi^n]\Big)\|_{L^2(Q_T)}^2 +c^{\prime} \ds\| u_N[q^*]- u_D[q^*,\phi^n]\|_{L^2(Q_T)}^2\\
\ &\leq \max\{1,c^{\prime}\} \ds\| u_N[q^*]- u_D[q^*,\phi^n]\|_{L^2(0,T;H^1(\Omega))}^2 \\
\ &= \max\{1,c^{\prime}\} \ds\| w_n\|_{L^2(0,T;H^1(\Omega))}^2.
  \end{array}
\end{equation*}
Exploiting the convergence $\rho_n \longrightarrow 0$ as $n\to\infty$ and employing \eqref{nnw}, we can infer that
\begin{align*}
    \displaystyle\lim_{n \to \infty}\mathcal{N}(q^*,\phi^n)+\rho_n\|q^*\|^2_{L^2(\Omega)} =0
\end{align*}
which, when combined with the inequality \eqref{PA}, yields:
\begin{align*}
    \mathcal{N}(q^n,\phi^n)
\longrightarrow 0 \; \; \hbox{  as  } \; n \rightarrow \infty.
\end{align*}
Thus we complete the proof.
\end{proof}

We can now present the convergence result of our optimization problem.

\begin{theorem}Assume a solution $q^*\in\Phi_{ad}$ to the inverse potential problem \eqref{inverse} corresponding to the datum $\phi$ exists.
Let $\{ \phi^n \} \subset {}_0H^{\frac{1}{2},\frac{\alpha}{4} }\left(\Sigma_{T}\right)$ be a sequence such that 
\begin{align*}
    \| \phi^n -\phi \|_{H^{\frac{1}{2},\frac{\alpha}{4} }\left(\Sigma_{T}\right)} \longrightarrow 0 \; \; \hbox{  as  } \; n \rightarrow \infty,
\end{align*}
and $\{\rho_n\}$ be a positive sequence such that $\ds\lim_{n\to\infty}\rho_n=0$. Let $\{ q^n \}\subset\Phi_{ad}$ be a sequence of minimizers of problems $(\mathcal{O}\mathcal{P}_{\rho_n}^n)$. Then, there exists a subsequence of $\{ q^n \}$, still denoted
by $\{ q^n \}$, such that
\begin{align*}
q^n\longrightarrow q^* \; \hbox{  strongly in}\; L^2(\Omega)\; \hbox{as}\; n\rightarrow\infty.
\end{align*}
\end{theorem}

\begin{proof}
Actually, repeating the same argument as that in \eqref{cv-0}, we can derive that there exist $\overline{q}\in\Phi_{ad}$ and a subsequence of $q^n$, still denoted by $q^n$ such that
\begin{align*}
q^n\longrightarrow \overline{q} \; \hbox{  strongly in}\; L^2(\Omega)\; \hbox{as}\; n\rightarrow\infty.
\end{align*}
We have
\begin{align*}
    \mathcal{N}(q^n,\phi^n) &=\ds\int_0^T \int_{\Omega}\Big|\nabla\Big(
 u_N[q^n]- u_D[q^n,\phi^n]\Big)\Big|^2\ \dx \dt \nonumber\\&\qquad+ \int_0^T \int_{\Omega}q^n\Big|
 u_N[q^n]- u_D[q^n,\phi^n]\Big|^2\ \dx\dt \notag\\
&\geq \ds\int_0^T \int_{\Omega}\Big|\nabla\Big(
 u_N[q^n]- u_D[q^n,\phi^n]\Big)\Big|^2\ \dx \dt  \nonumber\\&\qquad+c \int_0^T \int_{\Omega}\Big|
 u_N[q^n]- u_D[q^n,\phi^n]\Big|^2\ \dx \dt\notag\\
&\geq \min\{1,c\} \| u_N[q^n]- u_D[q^n,\phi^n]\|_{L^2(0,T;H^1(\Omega))}^2. 
\end{align*} 
According to Lemma \ref{leM}, we have $\displaystyle\lim_{n\to\infty}\mathcal{N}(q^n,\phi^n)=0$, and then:
\begin{align}\label{convergence3}
\| u_N[q^n]- u_D[q^n,\phi^n]\|_{L^2(0,T;H^1(\Omega))}^2 \longrightarrow 0 \; \; \hbox{  as  } \; n \rightarrow \infty.
\end{align}
Since $q^n\longrightarrow \overline{q} \; \hbox{  strongly in}\; L^2(\Omega)\; \hbox{as}\; n\rightarrow\infty,$ we can derive from \eqref{convergancefaible-N} the following:
\begin{equation*} 
    u_N[q^n]\longrightarrow u_N[\overline{q}]\; \hbox{ strongly in }\; L^2(0,T;H^1(\Omega))\; \hbox{ as }\; n\to\infty. 
\end{equation*}
So, we use \eqref{convergence3} to get
\begin{equation*}
    u_D[q^n,\phi^n]\longrightarrow u_N[\overline{q}]\; \hbox{ strongly in }\; L^2(0,T;H^1(\Omega))\; \hbox{ as }\; n\to\infty. 
\end{equation*}
Then, we use the trace theorem to deduce
\begin{equation*}
    u_D[q^n,\phi^n]\Big|_{\Sigma_T}\longrightarrow u_N[\overline{q}]\Big|_{\Sigma_T}\; \hbox{ strongly in }\; L^2(0,T;H^\frac{1}{2}(\partial\Omega))\; \hbox{ as }\; n\to\infty. 
\end{equation*}
Moreover, since $ u_D[q^n,\phi^n]\Big|_{\Sigma_T}= \phi^n $ and $ \phi^n \longrightarrow \phi $ in $H^{\frac{1}{2},\frac{\alpha}{4} }\left(\Sigma_{T}\right)$, 
we get
\begin{equation*}
     u_N[\overline{q}]\Big|_{\Sigma_T} = \phi.
\end{equation*} 
This implies that $\overline{q}$ solves the inverse problem \eqref{inverse}. Finally, by the uniqueness result (for the inverse problem \eqref{inverse}) established by Kian \emph{et al.} in \cite{Kain2021MA}, we conclude that $\overline{q}=q^*$.
\end{proof}

\begin{remark}
    The issue of robustness, which pertains to how the retrieved solutions behave in the presence of additional noise belonging to $L^2(0,T;L^2(\partial\Omega))$ instead of ${}_0H^{\frac{1}{2},\frac{\alpha}{4} }\left(\Sigma_{T}\right)$, still requires further investigation.
\end{remark}

\section{Numerical algorithm}\label{sec:algorithm}

As outlined in the introduction, we solve the optimization problem \eqref{stable} using the CGM to compute a minimizer of \eqref{dstance}. A critical component of this algorithm is the efficient computation of the gradient of $\mathcal{K}_\rho$. Although Theorem \ref{K-D-Rho} provides an expression for the Fr\'echet derivative (see \eqref{frechet-K-rho}), this form is not amenable to direct numerical computation. Consequently, we derive a computationally efficient expression for the gradient using the adjoint method, as detailed in the next section. The subsequent analysis establishes the Lipschitz continuity of the gradient of $\mathcal{K}$ (Section~\ref{sec:L-C-K}), a key property for proving convergence. In Section~\ref{sec:I-PP}, we present the iterative procedure for the CGM method. The convergence of this iterative scheme is then rigorously analyzed in Section~\ref{sec:cv-vvvv}. Finally, we address the practical implementation by deriving an explicit formula for the step size selection within the iterative process and summarize the complete reconstruction algorithm.

\subsection{The gradient of the functional $\K_\rho$ with the adjoint method}\label{sec:7}
Recall that the Kohn-Vogelius functional $\mathcal{K}_\rho$ can be decomposed as
\begin{align*}
    \mathcal{K}_\rho(q)= \mathcal{K}_{NN}(q) + \mathcal{K}_{DD}(q) -2\mathcal{K}_{ND}(q)+\rho\| q\|^2_{L^2(\Omega)},
\end{align*}
where
\begin{align*}
    \mathcal{K}_{NN}(q)&=\int_0^T\int_\Omega\Big|\nabla u_N[q]\Big|^2\dx\dt +\int_0^T\int_{\Omega}q\Big| u_N[q]\Big|^2 \dx \dt,\\
    \mathcal{K}_{DD}(q)&=\int_0^T\int_\Omega\Big|\nabla u_D[q]\Big|^2\dx\dt +\int_0^T\int_{\Omega}q\Big| u_D[q]\Big|^2 \dx \dt,\\
        \mathcal{K}_{ND}(q)&=\int_0^T\int_{\Omega}\nabla u_N[q] \cdot\nabla u_D[q] \dx \dt
    +\int_0^T\int_{\Omega}q  u_N[q] u_D[q] \dx \dt.
\end{align*}

Let us now introduce the definitions of the Riemann-Liouville fractional left integral $I_{0}^\alpha w(t)$ and right integral $I_{T}^\alpha w(t)$ respectively as
\begin{equation*}
I_{0}^\alpha w(t)=\frac{1}{\Gamma(\alpha)} \int_0^t (t-s)^{\alpha-1}w(s)\dss, \quad
I_{T}^\alpha w(t)=\frac{1}{\Gamma(\alpha)} \int_t^T (s-t)^{\alpha-1}w(s)\dss
\end{equation*}
for $w\in L^2(0,\,T).$ From this, the Caputo fractional derivative and the right Riemann-Liouville derivative ${}_tD^{\alpha}_T$ can be rephrased as 
\begin{equation}\label{operater-adj}
  \partial_t^\alpha w(t)=I_{0}^{1-\alpha} \frac{\mathrm{d}}{\dt}w(t),\qquad  {}_tD^{\alpha}_T w(t):=-\frac{\mathrm{d}}{\dt} I^{1-\alpha}_{T}w(t).
\end{equation}

The next lemma establishes the Fr\'echet derivative of the function $q\longmapsto\mathcal{K}_{NN}(q)$ using an adjoint state.

\begin{lemma}\label{lemm1}
The function $\mathcal{K}_{NN}$ is Fr\'echet differentiable and its gradient is
\begin{align} \label{frechlet_N}
    \mathcal{K}_{NN}^{\prime}(q)= \int_{0}^T \Big| u_N[q](x,t)\Big|^2 \dt +\int_{0}^T   u_N[q](x,t) V_N[q](x,t)\dt,
\end{align}
where $V_N[q]$ is the solution of the following adjoint problem:
\begin{eqnarray}\label{adjoint}
 \left\{
   \begin{array}{rclcl}
      {}_tD^{\alpha}_T V_N[q]-\Delta  V_N[q] +q\, V_N[q] &=&-2 J_{u_N} & \hbox{in} & Q_T, \\
 \partial_\nu V_N[q]&=&0 &\hbox{on} & \Sigma_T,\\
 I^{1-\alpha}_{T}V_N[q]&=&0 &\hbox{in} & \Omega\times\{T\},
   \end{array}
 \right.
 \end{eqnarray}
where the linear mapping $\eta\longmapsto J_{u_N}(\eta)=\displaystyle\int_\Omega\Big(\nabla u_N[q]\cdot\nabla \eta+q\,u_N[q]\eta\Big)\dx$ for all $\eta\in B^\alpha(Q_T).$
\end{lemma}

\begin{proof}
Let $\delta q\in L^2(\Omega)$ be such that $q+\delta q \in \Phi_{ad}$.
A straightforward calculation, using the expression of $\mathcal{K}_{NN}$, gives
\begin{align} \label{m}
    \mathcal{K}_{NN}(q+\delta q) - \mathcal{K}_{NN}(q) =\int_0^T\int_{\Omega}\delta q \Big| u_N[q]\Big|^2 \dx \dt+\sum_{i=1}^3\mathcal{I}_i,
\end{align}
where
\begin{align} 
    \notag
    \mathcal{I}_1&= \int_0^T\int_{\Omega} \Big|\nabla\Big( u_N[q+\delta q]- u_N[q]\Big)\Big|^2\dx\dt +\int_0^T\int_{\Omega} q\Big|  u_N[q+\delta q]- u_N[q]\Big|^2\dx\dt,\\
    \notag
    \mathcal{I}_2&=\int_0^T\int_{\Omega} \delta q\Big| u_N[q+\delta q]- u_N[q]\Big|^2\dx\dt + 2\int_0^T\int_{\Omega} \delta q u_N[q] \Big( u_N[q+\delta q]- u_N[q]\Big) \dx\dt,\\
        \notag
    \mathcal{I}_3&= 2\int_0^T\int_{\Omega} \nabla u_N[q] \cdot \nabla\Big( u_N[q+\delta q]- u_N[q]\Big) \dx\dt + 2\int_0^T\int_{\Omega} q u_N[q] \Big( u_N[q+\delta q]- u_N[q]\Big) \dx\dt.
\end{align}

First, let's estimate the term $\mathcal{I}_1$. Due to Lemma \ref{lipshitz}, we get
\begin{equation} \label{m0_1}
   \int_0^T \int_{\Omega} \Big|\nabla\Big( u_N[q+\delta q]- u_N[q]\Big)\Big|^2\dx\dt
    =o\Big(\|\delta q\|_{L^2(\Omega)}\Big).
\end{equation}
Now, regarding the second part of the term $\mathcal{I}_1$. Since $q \in \Phi_{ad}$, hence
\begin{align*} 
\int_0^T\int_{\Omega} q\Big|  u_N[q+\delta q]- u_N[q]\Big|^2\dx\dt\leq  c^\prime \| u_N[q+\delta q]- u_N[q]\|^2_{L^2(0,T;L^{2}(\Omega))}.
\end{align*}
From Lemma \ref{lipshitz}, one can deduce
\begin{equation} \label{m0_2}
    \int_0^T\int_{\Omega} q\Big|  u_N[q+\delta q]- u_N[q]\Big|^2\dx\dt
    =o\Big(\|\delta q\|_{L^2(\Omega)}\Big).
\end{equation}
Combining \eqref{m0_1} and \eqref{m0_2}, we get
\begin{equation} \label{m0}
    \mathcal{I}_1
    =o\Big(\|\delta q\|_{L^2(\Omega)}\Big).
\end{equation}

Let's proceed to estimate the term $\mathcal{I}_2$. Once more, we can employ the H\"{o}lder inequality  to obtain
\begin{align*} 
\int_\Omega \Big|\delta q u_N[q]\, \Big( u_N[q+\delta q]- u_N[q]\Big)\Big|\dx\leq  \|\delta q \|_{L^{2}(\Omega)}\| u_N[q]\|_{L^{4}(\Omega)}\| u_N[q+\delta q]- u_N[q]\|_{L^{4}(\Omega)}.
\end{align*}
Next, by integrating over time and applying the Cauchy-Schwarz inequality in conjunction with the Sobolev embedding theorem, we can deduce that
\begin{align*} 
\int_\Omega \Big|\delta q u_N[q]\, \Big( u_N[q+\delta q]- u_N[q]\Big)\Big|\dx\leq C(\Omega,d) \|\delta q \|_{L^{2}(\Omega)}\| u_N[q]\|_{H^{1}(\Omega)}\| u_N[q+\delta q]- u_N[q]\|_{H^{1}(\Omega)}.
\end{align*}
Subsequently, leveraging Lemma \ref{lipshitz} and Theorem \ref{weak-solution}, we acquire
\begin{align}
    2\int_0^T\int_{\Omega} \delta q u_N[q] \Big( u_N[q+\delta q]- u_N[q]\Big) \dx\dt
    \leq C(\Omega,d)C_N \|\varphi\|_{H^{-\frac{1}{2}, -\frac{\alpha}{4} }(\Sigma_T)} \|\delta q \|_{L^{2}(\Omega)}^2 =o\Big(\|\delta q\|_{L^2(\Omega)}\Big). \label{m1_1}
\end{align}
Likewise, we can establish that
\begin{equation} \label{m2_2}
    \int_0^T\int_{\Omega} \delta q\Big| \Big( u_N[q+\delta q]- u_N[q]\Big)\Big|^2\dx\dt
    =o\Big(\|\delta q\|_{L^2(\Omega)}\Big).
\end{equation}
Gathering \eqref{m1_1} and \eqref{m2_2}, we conclude
\begin{equation} \label{m2}
    \mathcal{I}_2
    =o\Big(\|\delta q\|_{L^2(\Omega)}\Big).
\end{equation}

To complete the proof, let's proceed with the estimation of the term $\mathcal{I}_3.$ From \eqref{PN}, we deduce that $w_N=\Big( u_N[q+\delta q]- u_N[q]\Big)$ is solution to
\begin{eqnarray} \label{cvx} 
\left\{
  \begin{array}{rcll}
    \partial_{t} ^{\alpha} w_N-\Delta w_N +q\, w_N&=& -\delta q\, u_N [q+\delta q] &\hbox{in}\ Q_T, \\
    \partial_\nu w_N &=&0 &\hbox{on}\  \Sigma_T, \\
    \displaystyle w_N(\cdot,0)&=&0  &\hbox{in}\ \Omega.
  \end{array}
\right.
\end{eqnarray}
Thanks to Lemma 2.2 in \cite{jiang2020numerical}, we can establish that problem \eqref{cvx} possesses a unique weak solution $w_N\in {}_0H^\alpha(0,T;L^2(\Omega))\cap L^2(0,T;H^2(\Omega))$. By utilizing the weak formulation of problem \eqref{cvx} and choosing $V_N[q]$ as a test function, we obtain
\begin{align}
\begin{split}\label{adjoint2}
\int_0^T\int_{\Omega}\Big(
\partial_t^\alpha (w_N)\;V_N[q] +
\nabla w_N\cdot\nabla V_N[q] &+ q\;w_N\;V_N[q] \Big)\dx \dt
\\
&=-\int_0^T\int_{\Omega}
\delta q u_N[q+\delta q]V_N[q] \dx \dt.
\end{split}
\end{align}
On the other hand, from \eqref{operater-adj}, we have
\begin{equation*}
    \int_{\Omega} \int_0^{T} \partial_t^\alpha (w_N)V_N[q]\,\dt\dx =\int_{\Omega} \int_0^{T} \Big(I_{0}^{1-\alpha}\partial_t w_N\Big)\; V_N[q]\,\dt\dx. 
\end{equation*}
Further, by the integration by parts, the relation in \cite[Lemma 4.1]{jiang2017weak} and noting that $w_N(\cdot,0)=\big(I^{1-\alpha}_TV_N[q]\big)(\cdot,\,T)=0$, we derive
\begin{align*}
\int_{\Omega} \int_0^{T} I_{0}^{1-\alpha}\Big( \partial_t w_N\Big) V_N[q]\,\dt\dx& =\int_{\Omega} \int_0^{T} \partial_t w_N \big(I_{T}^{1-\alpha} V_N[q]\big)\, \dt\dx \\
& =-\int_{\Omega} \int_0^{T} w_N \frac{\mathrm{d}}{\mathrm{d} t} \big(I_{T}^{1-\alpha} V_N[q]\big)\, \dt\dx \\
& =\int_{\Omega} \int_0^{T} w_N \Big( {}_tD^{\alpha}_T  V_N[q]\Big)\, \dt\dx.
\end{align*}
This implies that
\begin{equation} \label{int_partie}
    \int_0^T\int_{\Omega} \Big(\partial_t^\alpha  w_N \Big) V_N[q] \ \dx \dt =\int_0^T\int_\Omega \Big( {}_tD^{\alpha}_T V_N[q] \Big) w_N \dx\dt.
\end{equation}
Inserting \eqref{int_partie} into \eqref{adjoint2}, we have
\begin{align}
\begin{split}\label{adjoint20202}
\int_0^T\int_{\Omega}\Big(
{}_tD^{\alpha}_T V_N[q] \Big) w_N +
\nabla w_N\cdot\nabla V_N[q] &+ q\,w_N\,V_N[q] \Big)\dx \dt
\\
&=-\int_0^T\int_{\Omega}
\delta q\, u_N[q+\delta q]V_N[q] \dx \dt.
\end{split}
\end{align}
Now, by multiplying the first equation of \eqref{adjoint} by $w_N$ and integrating over space and time, we obtain
\begin{align} \notag
\int_0^T\int_{\Omega}\Big(
{}_tD^{\alpha}_T V_N[q]\,w_N +
\nabla V_N[q]\cdot&\nabla w_N + q\,V_N[q]\,w_N \Big)\dx \dt\\ &=-2\int_0^T\int_{\Omega}
\Big(\nabla u_N[q]\cdot\nabla w_N +q\,u_N[q]\,w_N\Big) \dx \dt.\label{adjoint3}
\end{align}
By subtracting \eqref{adjoint3} from \eqref{adjoint20202}, we deduce
\begin{align*} 
    \mathcal{I}_3=\int_0^T\int_{\Omega} \delta q \,w_N \,V_N[q] \dx\dt+\int_0^T\int_{\Omega} \delta q\, u_N[q] V_N[q] \dx\dt.
\end{align*}
Using the same analysis described in establishing \eqref{m2}, we can infer that
\begin{equation*}
    \int_0^T\int_{\Omega} \delta q \;w_N \;V_N[q] \dx\dt=o\Big(\|\delta q\|_{L^2(\Omega)}\Big).
\end{equation*}
Therefore,
\begin{align}
     \mathcal{I}_3= \int_0^T\int_{\Omega} \delta q u_N[q] V_N[q] \dx\dt+o\Big(\|\delta q\|_{L^2(\Omega)}\Big).\label{m3}
\end{align}

Finally, by substituting \eqref{m0}, \eqref{m2}, and \eqref{m3} into \eqref{m}, we have
\begin{align}\notag
    \mathcal{K}_{NN}(q+\delta q) - \mathcal{K}_{NN}(q) = \Big<\int_0^T\Big| u_N[q]\Big|^2 \dt + \int_0^T  u_N[q] V_N[q]\dt,\,\delta q\Big>_{L^2(\Omega)}+o\Big(\|\delta q\|_{L^2(\Omega)}\Big),
\end{align}
where $\Big<\cdot,\cdot\Big>_{L^2(\Omega)}$ is the usual scalar product on $L^2(\Omega).$ This means that the function $\mathcal{K}_{NN}$ is Fr\'echet differentiable, and its gradient is given by \eqref{frechlet_N}. The proof is completed.
\end{proof}

Building on the methodology established in Lemma \ref{lemm1}, we now derive the Fr\'echet derivatives of the operators $q\longmapsto\mathcal{K}_{DD}(q)$ and $q\longmapsto\mathcal{K}_{ND}(q)$. The derivatives are expressed in terms of solutions to associated adjoint problems, as detailed in the following lemmas.

\begin{lemma}\label{lemm2}
The function $\mathcal{K}_{DD}$ is Fr\'echet differentiable and its gradient is
\begin{align*}
    \mathcal{K}_{DD}^{\prime}(q)= \int_{0}^T \Big| u_D[q](x,t)\Big|^2 \dt +\int_{0}^T   u_D[q](x,t) V_D[q](x,t)\dt,
\end{align*}
where $V_D[q]$ is the solution of the following adjoint problem:
\begin{eqnarray} \label{adjoint_D}
 \left\{
   \begin{array}{rclcl}
      {}_tD^{\alpha}_T V_D[q]-\Delta  V_D[q] +q\, V_D[q] &=&-2 J_{u_D} & \hbox{in} & Q_T, \\
 V_D[q]&=&0 &\hbox{on} & \Sigma_T,\\
 I^{1-\alpha}_{T}V_D[q]&=&0 &\hbox{in} & \Omega\times\{T\},
   \end{array}
 \right.
 \end{eqnarray}
where the linear mapping $\eta\longmapsto J_{u_D}(\eta)=\displaystyle\int_\Omega\Big(\nabla u_D[q]\cdot\nabla \eta+q\,u_D[q]\eta\Big)\dx$ for all $\eta\in B^\alpha(Q_T).$
\end{lemma}

\begin{lemma} \label{lemm3}
The Neumann/Dirichlet function $\mathcal{K}_{ND}$ is Fr\'echet differentiable and its gradient is
\begin{align} 
\begin{split}
\label{frechlet_ND}
    \mathcal{K}_{ND}^{\prime}(q)&= \int_{0}^T  u_N[q](x,t)u_D[q](x,t) \dt +\int_{0}^T   u_D[q](x,t) \Theta_N[q](x,t)\dt\\
    &\qquad\quad+\int_{0}^T   u_N[q](x,t) \Theta_D[q](x,t)\dt,
    \end{split}
\end{align}
where $\Theta_D[q]$ and $\Theta_N[q]$, respectively, are the solutions of the following auxiliary adjoint problems

\begin{eqnarray} \label{adjoint_ND}
 \left\{
   \begin{array}{rclcl}
      {}_tD^{\alpha}_T \Theta_D[q]-\Delta  \Theta_D[q] +q\,\Theta_D[q] &=&- J_{u_N} & \hbox{in} & Q_T, \\
 \Theta_D[q]&=&0 &\hbox{on} & \Sigma_T,\\
 I^{1-\alpha}_{T}\Theta_D[q]&=&0 &\hbox{in} & \Omega\times\{T\},
   \end{array}
 \right.
 \end{eqnarray}
and
\begin{eqnarray} \label{adjoint_ND2}
 \left\{
   \begin{array}{rclcl}
      {}_tD^{\alpha}_T \Theta_N[q]-\Delta  \Theta_N[q] +q\, \Theta_N[q] &=&- J_{u_D} & \hbox{in} & Q_T, \\
 \partial_\nu \Theta_N[q]&=&0 &\hbox{on} & \Sigma_T,\\
 I^{1-\alpha}_{T}\Theta_N[q]&=&0 &\hbox{in} & \Omega\times\{T\}.
   \end{array}
 \right.
 \end{eqnarray}
Here, the mappings $\eta\mapsto J_{u_N}(\eta)$ and $\eta\mapsto J_{u_D}(\eta)$ are defined in Lemmas \ref{lemm1} and \ref{lemm2}, respectively.
\end{lemma}
\begin{proof}
Taking any $\delta q \in L^2(\Omega)$  such that $q+\delta q \in \Phi_{ad},$ we have
\begin{equation*}
    \mathcal{K}_{ND}(q+\delta q) - \mathcal{K}_{ND}(q) =\int_0^T\int_{\Omega} \delta q u_N[q] u_D[q]\dx\dt+ \sum_{i=1}^4\mathcal{T}_i,
\end{equation*}
where
\begin{align*}
    \mathcal{T}_1 &= \int_0^T\int_{\Omega} \nabla\Big( u_N[q+\delta q]- u_N[q]\Big)\cdot\nabla\Big( u_D[q+\delta q]- u_D[q]\Big)\dx\dt\\
    \;&\;\;\; +\int_0^T\int_{\Omega} \Big(q+\delta q\Big)\Big( u_N[q+\delta q]- u_N[q]\Big)\Big( u_D[q+\delta q]- u_D[q]\Big)\dx\dt,\\
    \mathcal{T}_2 &= \int_0^T\int_{\Omega} \delta q\; u_N[q]\Big( u_D[q+\delta q]- u_D[q]\Big)\dx\dt+\int_0^T\int_{\Omega} \delta q\; u_D[q]\Big( u_N[q+\delta q]- u_N[q]\Big)\dx\dt,\\
    \mathcal{T}_3 &= \int_0^T\int_{\Omega} \nabla\Big( u_N[q+\delta q]- u_N[q]\Big)\cdot\nabla u_D[q]\dx\dt+  \int_0^T\int_{\Omega} q\; u_D[q]\Big( u_N[q+\delta q]- u_N[q]\Big)\dx\dt,\\
    \mathcal{T}_4 &= \int_0^T\int_{\Omega} \nabla u_N[q]\cdot\nabla\Big( u_D[q+\delta q]- u_D[q]\Big)\dx\dt + \int_0^T\int_{\Omega} q\; u_N[q]\Big( u_D[q+\delta q]- u_D[q]\Big)\dx\dt.
\end{align*}

Similarly to \eqref{m0} and \eqref{m2}, we can deduce that
\begin{align} \label{nn0}
\mathcal{T}_1 =o\Big(\|\delta q\|_{L^2(\Omega)}\Big)\;\;\hbox{and}\;\;\mathcal{T}_2 =o\Big(\|\delta q\|_{L^2(\Omega)}\Big). 
\end{align}
Similarly to \eqref{m3}, one can obtain
\begin{align}
\mathcal{T}_3&= \int_0^T\int_{\Omega} \delta q u_N[q] \Theta_D[q] \dx\dt+o\Big(\|\delta q\|_{L^2(\Omega)}\Big),\label{nn2}\\
     \mathcal{T}_4&= \int_0^T\int_{\Omega} \delta q u_D[q] \Theta_N[q] \dx\dt+o\Big(\|\delta q\|_{L^2(\Omega)}\Big),\label{nn3}
\end{align}
where $\Theta_D[q]$ and $\Theta_N[q]$ solve respectively problems \eqref{adjoint_ND} and \eqref{adjoint_ND2}. 

Combining the above equalities \eqref{nn0}, \eqref{nn2}, \eqref{nn3}, we have
\begin{align*}
    &\mathcal{K}_{ND}(q+\delta q) - \mathcal{K}_{ND}(q) = \Big<\int_{0}^T  u_N[q](x,t)u_D[q](x,t)\dt\\
    &\qquad+\int_{0}^Tu_D[q](x,t) \Theta_N[q](x,t)\dt+\int_{0}^Tu_N[q](x,t) \Theta_D[q](x,t)\dt,\;\delta q\Big>_{L^2(\Omega)}\\
&\qquad\qquad\qquad+o\Big(\|\delta q\|_{L^2(\Omega)}\Big).
\end{align*}
This means that the function $\mathcal{K}_{ND}$ is Fréchet differentiable, and its gradient is given by \eqref{frechlet_ND}. The lemma is proved.
\end{proof}

Now, we will establish the Fr\'echet derivative of the Kohn-Vogelius functional $\mathcal{K}_\rho$. Prior to delving into this, we define
\begin{align*}
   \xi_D[q]= V_D[q] -2 \Theta_D[q] \;\;\hbox{and}\;\; \xi_N[q]= V_N[q] -2 \Theta_N[q].
\end{align*}
From \eqref{adjoint}, \eqref{adjoint_D}, \eqref{adjoint_ND}, and \eqref{adjoint_ND2}, it follows that $\xi_D$ and $\xi_N$ respectively solve
\begin{eqnarray}  \label{adjoint22}
 \left\{
   \begin{array}{rclcl}
      {}_tD^{\alpha}_T \xi_D[q]-\Delta  \xi_D[q] +q\, \xi_D[q] &=&2\, J & \hbox{in} & Q_T, \\
 \xi_D[q]&=&0 &\hbox{on} & \Sigma_T,\\
 I^{1-\alpha}_{T}\xi_D[q]&=&0 &\hbox{in} & \Omega\times\{T\},
   \end{array}
 \right.
 \end{eqnarray}
where the linear mapping $\eta\longmapsto J(\eta)=\displaystyle\int_{\Omega}
 \nabla \Big(u_N[q]-u_D[q]\Big).\nabla \eta\,\dx+\int_{\Omega}q \Big(u_N[q]-u_D[q]\Big)\eta\, \dx$ for all $\eta\in B^\alpha(Q_T)$, and
\begin{eqnarray}   \label{adjoint222}
 \left\{
   \begin{array}{rclcl}
      {}_tD^{\alpha}_T \xi_N[q]-\Delta  \xi_N[q] +q\, \xi_N[q] &=&-2\, J & \hbox{in} & Q_T, \\
 \partial_\nu \xi_N[q]&=&0 &\hbox{on} & \Sigma_T,\\
 I^{1-\alpha}_{T}\xi_N[q]&=&0 &\hbox{in} & \Omega\times\{T\}.
   \end{array}
 \right.
 \end{eqnarray}
Then, using the result of Lemmas \ref{lemm1}, \ref{lemm2}, and \ref{lemm3}, we deduce the following result.

\begin{theorem}\label{frechlet_K_theorem}
The Kohn-Vogelius functional $\mathcal{K}_\rho$ is Fr\'echet differentiable, and its gradient is given by
\begin{align}
\begin{split}
    \mathcal{K}_\rho^{\prime}(q)&=\int_{0}^T \Big| u_N[q](x,t)-u_D[q](x,t)\Big|^2 \dt +\int_{0}^T \;  u_D[q](x,t) \xi_N[q](x,t)\dt \\
    &\qquad\quad+\int_{0}^T  \; u_N[q](x,t) \xi_D[q](x,t)\dt + 2\rho \;q \label{frechlet_K}.
    \end{split}
\end{align}
\end{theorem}

\subsection{ \bf Lipschitz continuous of $\mathcal{K}_\rho^{\prime}(q)$  }\label{sec:L-C-K}

Here we establish that the gradient of $\K$ is Lipschitz continuous.

\begin{theorem} \label{theo_f_gradient}
  The Fréchet gradient \(\mathcal{K}_\rho'(q)\) is Lipschitz continuous, i.e., there exists a positive constant \(C_L\) such that
   \begin{equation*}
       \| \mathcal{K}_\rho^{\prime}(q_1) -\mathcal{K}_\rho^{\prime}(q_2)\|_{L^2(\Omega)}\leq C_L \|q_1 - q_2\|_{L^2(\Omega)}, \quad \forall q_1, q_2 \in \Phi_{\mathrm{ad}}.
   \end{equation*}
\end{theorem}
\begin{proof} 
Let \(q_1, q_2 \in \Phi_{\mathrm{ad}}\), and define \(\delta \mathcal{K}_\rho' := \mathcal{K}_\rho'(q_1) - \mathcal{K}_\rho'(q_2)\). From \eqref{frechlet_K}, we have
\begin{equation}
    \label{delta_k}
\left\| \delta \mathcal{K}_\rho' \right\|_{L^2(\Omega)} \leq \left\| \delta \mathcal{K}_1 \right\|_{L^2(\Omega)} + \left\| \delta \mathcal{K}_2 \right\|_{L^2(\Omega)} + 2\rho \left\| q_1 - q_2 \right\|_{L^2(\Omega)},
\end{equation}
where
$$
\begin{aligned}
\left\| \delta \mathcal{K}_1 \right\|_{L^2(\Omega)} &=
\left\| \int_0^T u_N[q_1] (u_N[q_1] - u_N[q_2]) \, \mathrm{d}t \right\|_{L^2(\Omega)} + \left\| \int_0^T u_N[q_2] (u_N[q_1] - u_N[q_2]) \, \mathrm{d}t \right\|_{L^2(\Omega)} \\
&\;+ \left\| \int_0^T u_D[q_1] (u_D[q_1] - u_D[q_2]) \, \mathrm{d}t \right\|_{L^2(\Omega)} + \left\| \int_0^T u_D[q_2] (u_D[q_1] - u_D[q_2]) \, \mathrm{d}t \right\|_{L^2(\Omega)} \\
&\;+ 2 \left\| \int_0^T u_D[q_2] (u_N[q_1] - u_N[q_2]) \, \mathrm{d}t \right\|_{L^2(\Omega)} + 2 \left\| \int_0^T u_N[q_1] (u_D[q_1] - u_D[q_2]) \, \mathrm{d}t \right\|_{L^2(\Omega)},\\
\left\| \delta \mathcal{K}_2 \right\|_{L^2(\Omega)} &=
\left\| \int_0^T u_D[q_2] (\xi_N[q_1] - \xi_N[q_2]) \, \mathrm{d}t \right\|_{L^2(\Omega)} + \left\| \int_0^T \xi_N[q_1] (u_D[q_1] - u_D[q_2]) \, \mathrm{d}t \right\|_{L^2(\Omega)} \\
&\;+ \left\| \int_0^T u_N[q_2] (\xi_D[q_1] - \xi_D[q_2]) \, \mathrm{d}t \right\|_{L^2(\Omega)} + \left\| \int_0^T \xi_D[q_1] (u_N[q_1] - u_N[q_2]) \, \mathrm{d}t \right\|_{L^2(\Omega)}.
\end{aligned}
$$
We begin by estimating the first term of \(\delta \mathcal{K}_1\):
\begin{equation*}
I_1 = \left\| \int_0^T u_N[q_1] (u_N[q_1] - u_N[q_2]) \, \mathrm{d}t \right\|_{L^2(\Omega)}.
\end{equation*}
By Minkowski's integral inequality \cite{folland1999real}:
\begin{align*}
    \Big\| \int_0^T u_N[q_1](\cdot,t) (u_N[q_1] &- u_N[q_2])(\cdot,t) \; \dt \Big\|_{L^2(\Omega)}\\
&\leq \int_0^T \Big\| u_N[q_1](\cdot,t) (u_N[q_1] - u_N[q_2])(\cdot,t) \Big\|_{L^2(\Omega)} \dt.
\end{align*}
Apply H\"older's inequality:
\begin{align*}
    \Big\| u_N[q_1](\cdot,t) (u_N[q_1] &- u_N[q_2])(\cdot,t) \Big\|_{L^2(\Omega)} \\&\leq \Big\|u_N[q_1](\cdot,t)\Big\|_{L^4(\Omega)}  \Big\|(u_N[q_1] - u_N[q_2])(\cdot,t)\Big\|_{L^4(\Omega)}.
\end{align*}
By the Sobolev embedding theorem \cite{brezis2010functional}, since $d \leq 3$ and $\Omega$ has Lipschitz boundary, we have the continuous embedding $H^1(\Omega) \hookrightarrow L^4(\Omega)$. Then, we get
\begin{align*}
    \Big\|u_N[q_1](\cdot,t)\Big\|_{L^4(\Omega)}  \Big\|(u_N[q_1] &- u_N[q_2])(\cdot,t)\Big\|_{L^4(\Omega)} \\ &\leq C \Big\|u_N[q_1](\cdot,t)\Big\|_{H^1(\Omega)}  \Big\|(u_N[q_1] - u_N[q_2])(\cdot,t)\Big\|_{H^1(\Omega)}.
\end{align*}
Combining these estimates:
\begin{align*}
    \Big\| \int_0^T u_N[q_1](\cdot,t) (u_N[q_1] &- u_N[q_2])(\cdot,t)  \,\dt \Big\|_{L^2(\Omega)}  \\
&\leq C \int_0^T \Big\|u_N[q_1](\cdot,t)\Big\|_{H^1(\Omega)}  \Big\|(u_N[q_1] - u_N[q_2])(\cdot,t)\Big\|_{H^1(\Omega)}  \dt.
\end{align*}
By the Cauchy–Schwarz inequality, it follows that
\begin{align*}
    \int_0^T \Big\|u_N[q_1](\cdot,t)\Big\|_{H^1(\Omega)}  \Big\|(u_N[q_1] &- u_N[q_2])(\cdot,t)\Big\|_{H^1(\Omega)}  \dt\\
&\leq C \Big\|u_N[q_1]\Big\|_{L^2(0,T; H^1(\Omega))}  \Big\|u_N[q_1] - u_N[q_2]\Big\|_{L^2(0,T; H^1(\Omega))}.
\end{align*}
This implies
\begin{equation}\label{sc_0}
I_1 \leq C\left\| u_N[q_1] \right\|_{L^2(0,T; H^1(\Omega))} \left\| u_N[q_1] - u_N[q_2] \right\|_{L^2(0,T; H^1(\Omega))}.
\end{equation}
From Theorem \ref{well-posedness-solution}, there exists \(C > 0\) such that
\begin{equation}
    \label{sc_1}
\left\| u_N[q_1] \right\|_{L^2(0,T; H^1(\Omega))} \leq C \left\| \varphi \right\|_{H^{-\frac{1}{2}, -\frac{\alpha}{4}}(\Sigma_T)}.
\end{equation}
By Lemma \ref{lipshitz}, we have
\begin{equation}
    \label{sc_2}
\left\| u_N[q_1] - u_N[q_2] \right\|_{L^2(0,T; H^1(\Omega))} \leq C \left\| q_1 - q_2 \right\|_{L^2(\Omega)}.
\end{equation}
By substituting equations \eqref{sc_1} and \eqref{sc_2} into equation \eqref{sc_0}, we get
\begin{equation} \label{sa_0}
    I_1 \leq  C \| q_1-q_2 \|_{L^2(\Omega)},
\end{equation}
The remaining terms in \(\delta \mathcal{K}_1\) are estimated similarly, using the boundedness of \(u_N[q_1]\),  \(u_N[q_2]\), \(u_D[q_1]\), \(u_D[q_2]\) in \(L^2(0,T; H^1(\Omega))\) and the Lipschitz continuity of \(u_N[q]\) and \(u_D[q]\) with respect to \(q\). Therefore,
\begin{equation}
    \label{delta_k_1}
\left\| \delta \mathcal{K}_1 \right\|_{L^2(\Omega)} \leq C \left\| q_1 - q_2 \right\|_{L^2(\Omega)}.
\end{equation}

We now estimate the term \(\left\| \delta \mathcal{K}_2 \right\|_{L^2(\Omega)}\). We begin by establishing that the maps \(q \mapsto \xi_N[q]\) and \(q \mapsto \xi_D[q]\) are bounded and Lipschitz continuous.

Firstly, we proceed to estimate $\xi_N$. Although the backward Riemann-Liouville derivative ${}_tD^{\alpha}_T$ appears in \eqref{adjoint222}, the solution $\xi_N[q_1]$ is equivalently characterized as the solution to the following problem featuring a backward Caputo derivative:
\begin{eqnarray}\label{adjointq000}
\left\{
  \begin{array}{rcll}
    \partial^\alpha_{T-}  \xi_N[q_1]-\Delta \xi_N[q_1] +q_1\, \xi_N[q_1] &=& -2J,  &\text{in } Q_T, \\
    \partial_\nu\xi_N[q_1]&=& 0,  &\text{on } \Sigma_T, \\
     \xi_N[q_1](\cdot,T)&=&0, &\text{in } \Omega,
  \end{array}
\right.
\end{eqnarray}
where the backward Caputo derivative is defined by
\[
\partial^\alpha_{T-}w(t) := -\frac{1}{\Gamma(1-\alpha)}\int_t^T (s-t)^{-\alpha} \frac{\partial w}{\partial s}(s) \, \mathrm{d}s.
\]
Applying the change of variable $t \mapsto T - t$ to system \eqref{adjointq000} yields the equivalent forward-in-time system:
\begin{eqnarray*}
\left\{
  \begin{array}{rcll}
    \partial^\alpha_{t}  \overline{\xi}_N[q_1]-\Delta \overline{\xi}_N[q_1] +q_1\, \overline{\xi}_N[q_1] &=& -2\overline{J},  &\text{in } Q_T, \\
    \partial_\nu\overline{\xi}_N[q_1]&=& 0,  &\text{on } \Sigma_T, \\
    \overline{\xi}_N[q_1](\cdot,0)&=&0, &\text{in } \Omega,
  \end{array}
\right.
\end{eqnarray*}
where we have defined $\overline{\xi}_N[q_1](\cdot,t) = \xi_N[q_1](\cdot,T-t)$. The source term $\overline{J}$ is defined by the linear mapping
\begin{align*}
    \eta \longmapsto \overline{J}(\eta(\cdot,t)) =& \int_{\Omega} \nabla \Big(u_N[q_1](x,T-t)-u_D[q_1](x,T-t)\Big) \cdot \nabla \eta  \,\mathrm{d}x \\ &+ \int_{\Omega} q_1 \Big(u_N[q_1](x,T-t)-u_D[q_1](x,T-t)\Big)\eta  \,\mathrm{d}x,
\end{align*}
which holds for all $\eta \in B^\alpha(Q_T)$.

\noindent Following the same strategy developed in the proof of Theorem \ref{well-posedness-solution}, we deduce the following a priori estimate:
\begin{equation*}
\left\| \overline{\xi}_N[q_1] \right\|_{B^{\frac{\alpha}{2}}(Q_T)} \leq C \left\| u_N[q_1]- u_D[q_1] \right\|_{L^2(0,T;H^1(\Omega))},
\end{equation*}
where the constant $C > 0$ is independent of the potential $q_1$.\\
Using the bounds from \eqref{estimation-direct} and \eqref{estimate-Direchlet}, and the relation \(\overline{\xi}_N[q_1](\cdot, t) = \xi_N[q_1](\cdot, T - t)\), we obtain
\begin{equation*}
    \left\| \xi_N[q_1] \right\|_{L^2(0,T;H^1(\Omega))} \leq C\left( \left\|\phi\right\|_{H^{\frac{1}{2},\frac{\alpha}{4}}(\Sigma_T)} + \left\|\varphi\right\|_{H^{-\frac{1}{2}, -\frac{\alpha}{4}}(\Sigma_T)} \right).
\end{equation*}
A parallel argument demonstrates the analogous boundedness result for the map \(q_1 \mapsto \xi_D[q_1]\):
\begin{equation*} 
    \left\| \xi_D[q_1] \right\|_{L^2(0,T;H^1(\Omega))} \leq C\left( \left\|\phi\right\|_{H^{\frac{1}{2},\frac{\alpha}{4}}(\Sigma_T)} + \left\|\varphi\right\|_{H^{-\frac{1}{2}, -\frac{\alpha}{4}}(\Sigma_T)} \right).
\end{equation*}
Let \(\delta \overline{\xi}_D = \overline{\xi}_D[q_1] - \overline{\xi}_D[q_2]\). Then \(\delta \overline{\xi}_D\) satisfies
 \begin{eqnarray*} 
\left\{
  \begin{array}{rcll}
    \partial^\alpha _{t} \delta \overline{\xi}_D   -\Delta \delta \overline{\xi}_D +q_2\;\delta \overline{\xi}_D&=&(q_2-q_1) \overline{\xi}_D[q_2] +2\Big(J_{ND}[q_1] -J_{ND}[q_2]\Big)&\hbox{in}\ Q_T, \\
    \displaystyle  \delta \overline{\xi}_D &=&0 &\hbox{on}\  \Sigma_T, \\
    \displaystyle \delta \overline{\xi}_D (\cdot,0)&=&0  &\hbox{in}\ \Omega,
  \end{array}
\right.
\end{eqnarray*}
where the linear mapping $\eta\longmapsto J_{ND}[q_i](\eta)=\displaystyle\int_{\Omega}
 \nabla \Big(u_N[q_i]-u_D[q_i]\Big)\cdot\nabla \eta\,\dx+\int_{\Omega}q_i \Big(u_N[q_i]-u_D[q_i]\Big)\eta\, \dx$ for all $\eta\in B^\alpha(Q_T)$ ($i\in\{1,2\}$).
By reiterating the same argument employed in the proof of Lemma \ref{lemma-L-C-D}, we can deduce
\begin{equation*}
\|\delta\overline{\xi}_D\|_{L^2(0,T;H^1(\Omega))}\leq C
\| q_1-q_2\|_{L^2(\Omega)},
\end{equation*}
which implies
\begin{equation} \label{lip_adj_D}
\|\xi_D[q_1]-\xi_D[q_2]\|_{L^2(0,T;H^1(\Omega))}\leq C
\| q_1-q_2\|_{L^2(\Omega)}.
\end{equation}
Similarly,
\begin{equation*} 
\|\xi_N[q_1]-\xi_N[q_2]\|_{L^2(0,T;H^1(\Omega))}\leq C
\| q_1-q_2\|_{L^2(\Omega)}.
\end{equation*}
Now consider the first term in \(\delta \mathcal{K}_2\):
\[
M_1 = \left\| \int_0^T u_D[q_2] (\xi_N[q_1] - \xi_N[q_2]) \, \mathrm{d}t \right\|_{L^2(\Omega)}.
\]
Using the boundedness of \(u_D[q_2]\) in \(L^2(0,T; H^1(\Omega))\) established in \eqref{estimate-Direchlet}, along with the Lipschitz continuity from \eqref{lip_adj_D}, and applying the same estimation technique developed for \eqref{sa_0}, we obtain
\[
M_1 \leq C \left\| q_1 - q_2 \right\|_{L^2(\Omega)}.
\]
The remaining terms in \(\delta \mathcal{K}_2\) are bounded similarly, using Lipschitz continuity of \(\xi_N\), \(\xi_D\), \(u_N\) and \(u_D\) with respect to \(q\) and the boundedness of \(u_N[q_2]\), \(u_D[q_2]\), \(\xi_N[q_1]\), and \(\xi_D[q_1]\). Hence,
\begin{equation} \label{delta_k_2}
    \left\| \delta \mathcal{K}_2 \right\|_{L^2(\Omega)} \leq C \left\| q_1 - q_2 \right\|_{L^2(\Omega)}.
\end{equation}
Substituting \eqref{delta_k_1} and \eqref{delta_k_2} into \eqref{delta_k} yields
\[
\left\| \delta \mathcal{K}_\rho' \right\|_{L^2(\Omega)} \leq (C + 2\rho) \left\| q_1 - q_2 \right\|_{L^2(\Omega)},
\]
which completes the proof.
\end{proof}

Next, to solve the optimization problem \eqref{stable}, we employ a CGM.

\subsection{Conjugate gradient method}\label{sec:I-PP}

In this section, we introduce the CGM as the numerical approach for solving the optimization problem under consideration. The following iterative process based on the CGM is used for the estimation of $q(x)$ by
minimizing the Kohn-Vogelius functional $\mathcal{K}_\rho:$
\begin{align} \label{qn}
    q_{n+1}(x)=q_n(x) - \beta_n P_n(x),\;\; n=0,\,1,\,2,\cdots
\end{align}
where:
\begin{itemize}
    \item[-] the subscript $n$ denotes the number of iterations;
    
    \item[-] the function $q_0(x)$ is the initial guess for $q(x)$;
    
    \item[-] $\beta_n$ is the step search size in passing from iteration $n+1$;
    
    \item[-] the function $P_n(x)$ is the direction of descent given by
\begin{equation} \label{qn+2}
     P_0(x)=\mathcal{K}_\rho ^\prime(q_0(x)),\;\; P_n(x)=\mathcal{K}_\rho ^\prime(q_n(x))+\gamma_n P_{n-1}(x),\;\; n=1,\,2,\cdots\,.
\end{equation}
\end{itemize}
Various expressions are available for the
 conjugate coefficient $\gamma_n$, such as in the Fletcher–Reeves method (see, for example, \cite{daniel1971approximate,fletcher1964function}):
\begin{equation} \label{qn+3}
    \gamma_0=0,\;\; \gamma_n=\frac{\|\mathcal{K}_\rho '(q_n)\|^2_{L^2(\Omega)}}{\|\mathcal{K}_\rho '(q_{n-1})\|^2_{L^2(\Omega)}},\;\; n=1,\,2,\cdots\,.
\end{equation}
The step size $\beta_n$ is determined by the exact line search method:
\begin{equation}\label{min_beta_n}
\mathcal{K}_\rho(q_{n+1}) = \min_{\beta \geq 0} \mathcal{K}_\rho(q_n - \beta P_n).
\end{equation}
This implies the first-order optimality condition for the function $\beta\longmapsto\psi(\beta) := \mathcal{K}_\rho (q_n - \beta P_n)$ at $\beta = \beta_n$. By the chain rule and the Fréchet differentiability of $\mathcal{K}_\rho$ (Theorem \ref{frechlet_K_theorem}), we have:
\[
\psi^\prime(\beta_n) = -\int_{\Omega} \mathcal{K}_\rho^\prime(q_{n+1}) P_n \;\dx = 0.
\]
Thus, the exact line search yields the orthogonality condition:
\begin{equation} \label{min_beta_n_2}
\int_{\Omega}  \mathcal{K}_\rho'(q_{n+1}) P_n \;\dx = 0.
\end{equation}

\subsection{Convergence of the iterative process}\label{sec:cv-vvvv}

In this subsection, inspired by the arguments in \cite{dai1996convergence}, we discuss the convergence of the above iterative process \eqref{qn}-\eqref{qn+3}.

\begin{theorem} \label{theo_f_gradient_2}
    Let $\{q_n\}_{n\geq 0} \subset \Phi_{ad}$ be generated by \eqref{qn} with $\beta_n$ satisfying \eqref{min_beta_n}. Then $\{\mathcal{K}_\rho(q_n)\}_{n\geq 0}$ is monotonically decreasing and convergent.
\end{theorem}
\begin{proof}
From \eqref{min_beta_n_2} and the definition of $P_n$ in \eqref{qn+2}, we obtain for all $n \geq 0$:
\begin{equation*}
    \int_\Omega\mathcal{K}^\prime_\rho(q_n)P_n\;\dx=-\|\mathcal{K}^\prime_\rho(q_n) \|^2_{L^2(\Omega)}.
\end{equation*}
By \eqref{min_beta_n}, for any $\beta \geq 0$:
\begin{equation*}
    \mathcal{K}_\rho(q_{n+1}) \leq \mathcal{K}_\rho(q_n - \beta P_n).
\end{equation*}
Using the mean value theorem and Lipschitz continuity (Theorem \ref{theo_f_gradient}), we obtain
\begin{align*}
    \mathcal{K}_\rho(q_{n}) - \mathcal{K}_\rho(q_n - \beta P_n) &= -\beta \int_0^1\int_\Omega \mathcal{K}_\rho^\prime (q_n - \theta\beta P_n)P_n\;\dx\mathrm{d}\theta\\
    &= -\beta \int_\Omega \mathcal{K}_\rho^\prime (q_n )P_n\;\dx -\beta \int_0^1\int_\Omega \Big(\mathcal{K}_\rho^\prime (q_n - \theta\beta P_n) -\mathcal{K}_\rho^\prime (q_n )\Big)P_n\;\dx\mathrm{d}\theta\\
    &\geq  \beta \|\mathcal{K}^\prime_\rho(q_n) \|^2_{L^2(\Omega)} -\beta \|P_n\|_{L^2(\Omega)}\int_0^1 \|\mathcal{K}_\rho^\prime (q_n - \theta\beta P_n) -\mathcal{K}_\rho^\prime (q_n )\|_{L^2(\Omega)} \mathrm{d}\theta \\
    &\geq \beta \|\mathcal{K}^\prime_\rho(q_n) \|^2_{L^2(\Omega)} -\frac{1}{2}\beta^2 C_L \|P_n\|^2_{L^2(\Omega)}.
\end{align*}
The maximum decrease occurs at
\begin{equation*}
    \beta^*= \frac{\|\mathcal{K}^\prime_\rho(q_n) \|^2_{L^2(\Omega)}}{C_L\|P_n\|^2_{L^2(\Omega)}},
\end{equation*}
yielding
\begin{equation*}
    \mathcal{K}_\rho(q_{n}) - \mathcal{K}_\rho(q_n - \beta^* P_n)  \geq \frac{\|\mathcal{K}^\prime_\rho(q_n) \|^2_{L^2(\Omega)}}{2C_L\|P_n\|^2_{L^2(\Omega)}}.
\end{equation*}
Thus $\mathcal{K}_\rho(q_{n+1}) \leq \mathcal{K}_\rho(q_n)$ for all $n \geq 0$, and since $\mathcal{K}_\rho$ is bounded below, the sequence converges.
\end{proof}

\begin{theorem}
    Under the assumption of Theorem \ref{theo_f_gradient_2}, the CGM \eqref{qn}-\eqref{qn+3} either terminates at a stationary point or converges in the following sense
    \begin{equation*}
        \lim_{n\to\infty}\inf \|\mathcal{K}^\prime_\rho(q_n) \|_{L^2(\Omega)}=0.
    \end{equation*}
\end{theorem}

\begin{proof}
First, by exploiting the Lipschitz continuity of the gradient $\mathcal{K}^\prime(q)$ 
and following the arguments in \cite{dai1996convergence,fletcher1964function}, 
we obtain that
\begin{equation} \label{key_inequality}
\sum_{n=0}^{\infty} \frac{\|\mathcal{K}_\rho'(q_n)\|_{L^2(\Omega)}^4}{\|P_n\|_{L^2(\Omega)}^2} < \infty.
\end{equation}
Next, we show by contradiction that  
\[
\inf_{n \in \mathbb{N}} \|\mathcal{K}_\rho^\prime(q_n)\|_{L^2(\Omega)} \to 0 
\quad \text{as } n \to \infty.
\]  
Suppose, on the contrary, that this is not the case. Then  
\[
\lim_{n \to \infty}\inf \|\mathcal{K}_\rho^\prime(q_n)\|_{L^2(\Omega)} > 0.
\]  
Hence, there exist a constant $\varepsilon > 0$ and an index $N \in \mathbb{N}$ such that, 
for all $n \geq N$, 
\begin{equation} \label{lower_bound}
\|\mathcal{K}_\rho'(q_n)\|_{L^2(\Omega)} \geq \varepsilon.
\end{equation}
Taking the inner product with $P_n$ and using the orthogonality $\langle \mathcal{K}_\rho'(q_n), P_{n-1} \rangle_{L^2(\Omega)} = 0$ (which follows from \eqref{min_beta_n_2} and the update rule), we obtain:
\[
\|P_n\|_{L^2(\Omega)}^2 = \langle \mathcal{K}_\rho'(q_n), P_n \rangle_{L^2(\Omega)} + \gamma_n \langle P_{n-1}, P_n \rangle_{L^2(\Omega)} = \|\mathcal{K}_\rho'(q_n)\|_{L^2(\Omega)}^2 + \gamma_n^2 \|P_{n-1}\|_{L^2(\Omega)}^2,
\]
where the last equality uses the fact that $\langle P_{n-1}, P_n \rangle_{L^2(\Omega)} = \gamma_n \|P_{n-1}\|_{L^2(\Omega)}^2$ (which can be shown by induction). Substituting the definition of $\gamma_n$ from \eqref{qn+3} yields:
\begin{equation} \label{recurrence}
\|P_n\|_{L^2(\Omega)}^2 = \|\mathcal{K}_\rho'(q_n)\|_{L^2(\Omega)}^2 + \frac{\|\mathcal{K}_\rho'(q_n)\|_{L^2(\Omega)}^4}{\|\mathcal{K}_\rho'(q_{n-1})\|_{L^2(\Omega)}^4} \|P_{n-1}\|_{L^2(\Omega)}^2.
\end{equation}
Divide both sides of \eqref{recurrence} by $\|\mathcal{K}_\rho'(q_n)\|^4$:
\[
\frac{\|P_n\|_{L^2(\Omega)}^2}{\|\mathcal{K}_\rho'(q_n)\|_{L^2(\Omega)}^4} = \frac{1}{\|\mathcal{K}_\rho'(q_n)\|_{L^2(\Omega)}^2} + \frac{\|P_{n-1}\|_{L^2(\Omega)}^2}{\|\mathcal{K}_\rho'(q_{n-1})\|_{L^2(\Omega)}^4}.
\]
Define $a_n = \frac{\|P_n\|_{L^2(\Omega)}^2}{\|\mathcal{K}_\rho'(q_n)\|_{L^2(\Omega)}^4}$. Then the above becomes:
\[
a_n = a_{n-1} + \frac{1}{\|\mathcal{K}_\rho'(q_n)\|_{L^2(\Omega)}^2}.
\]
This is a telescoping recurrence. Applying it iteratively for $n \geq N$ gives:
\begin{equation} \label{telescoping}
a_n = a_N + \sum_{k=N+1}^{n} \frac{1}{\|\mathcal{K}_\rho'(q_k)\|_{L^2(\Omega)}^2}.
\end{equation}
Using the lower bound \eqref{lower_bound}, we have for all $k \geq N+1$:
\[
\frac{1}{\|\mathcal{K}_\rho'(q_k)\|_{L^2(\Omega)}^2} \geq \frac{1}{\varepsilon^2}.
\]
Substituting into \eqref{telescoping} yields:
\[
a_n \geq a_N + \frac{n - N}{\varepsilon^2}.
\]
Recalling the definition of $a_n$, this implies:
\begin{equation} \label{lower_bound_P}
\|P_n\|_{L^2(\Omega)}^2 \geq \|\mathcal{K}_\rho'(q_n)\|_{L^2(\Omega)}^4 \left( a_N + \frac{n - N}{\varepsilon^2} \right).
\end{equation}
Since the sequence $\{q_n\}$ is bounded in $L^2(\Omega)$ (as it lies in $\Phi_{ad}$) and $\mathcal{K}_\rho(q)^\prime$ is Lipschitz continuous, there exists a constant $M > 0$ such that for all $n$,
\begin{equation} \label{upper_bound}
\|\mathcal{K}_\rho'(q_n)\|_{L^2(\Omega)} \leq M.
\end{equation}
This follows because the image of a bounded set under a Lipschitz continuous map is bounded.

From \eqref{lower_bound_P} and \eqref{upper_bound}, we have:
\[
\|P_n\|_{L^2(\Omega)}^2 \geq \varepsilon^4 \left( a_N + \frac{n - N}{\varepsilon^2} \right) = \varepsilon^4 a_N + \varepsilon^2 (n - N).
\]
Therefore,
\[
\frac{1}{\|P_n\|_{L^2(\Omega)}^2} \leq \frac{1}{\varepsilon^4 a_N + \varepsilon^2 (n - N)}.
\]
Now, consider the terms in the series \eqref{key_inequality} for $n \geq N$:
\[
\frac{\|\mathcal{K}_\rho'(q_n)\|_{L^2(\Omega)}^4}{\|P_n\|_{L^2(\Omega)}^2} \geq \frac{\varepsilon^4}{\|P_n\|_{L^2(\Omega)}^2} \geq \frac{\varepsilon^4}{\varepsilon^4 a_N + \varepsilon^2 (n - N)} = \frac{\varepsilon^2}{\varepsilon^2 a_N + (n - N)}.
\]
Summing these terms from $n = N$ to $\infty$ gives:
\[
\sum_{n=N}^{\infty} \frac{\|\mathcal{K}_\rho'(q_n)\|^4}{\|P_n\|^2} \geq \varepsilon^2 \sum_{n=N}^{\infty} \frac{1}{\varepsilon^2 a_N + (n - N)}.
\]
The series on the right-hand side is a harmonic series, which diverges. This contradicts the convergence of the series in \eqref{key_inequality}.
The contradiction implies that our initial assumption \eqref{lower_bound} is false. Therefore,
\[
\liminf_{n\to\infty} \|\mathcal{K}_\rho'(q_n)\|_{L^2(\Omega)} = 0,
\]
which completes the proof.
\end{proof}

\subsection{Iterative algorithm}\label{sec:Iterative-algo}

According to the above discussion, all the parameters of the iterative process \eqref{qn} are explicitly defined, except for the search step size, denoted by $\beta_n$. The step size $\beta_n$ can be determined through the following deduction.
Initially, setting $\delta q_n = P_n$, we can apply Lemma \ref{taylor_n} to linearize the estimated Neumann solution $u_N[q_{n+1}] = u_N[q_n - \beta_n P_{n}]$ using a Taylor series expression in the form:
\begin{align}\label{ef1}
   u_N[q_n -\beta_n P_{n}] \approx 
   u_N[q_{n}] - \beta_n  \mathcal{U}_{N}[q_n](P_n).
\end{align}
Likewise, by utilizing Lemma \ref{taylor_d}, we obtain a similar linearization for the Dirichlet solution $u_D[q_n - \beta_n P_n]$:
\begin{align}\label{ef2}
   u_D[q_n - \beta_n P_n] \approx u_D[q_{n}] - \beta_n  \mathcal{U}_{D}[q_n](P_n).
\end{align}
In these equations \eqref{ef1} and \eqref{ef2}, $\mathcal{U}_{N}[q_n](P_n)$ and $\mathcal{U}_{D}[q_n](P_n)$ solve respectively problems \eqref{P_n++} and \eqref{P_d++} with $\delta q_n = P_n$ and $q=q_n$. Therefore, with a straightforward calculation, we can derive:
\begin{align*}
    \mathcal{K}_\rho (q_n - \beta_n P_n)&\approx \int_0^T \int_{\Omega}\Big|\nabla\Big(
 u_N[q_{n}]- u_D[q_{n}] - \beta_n \Big( \mathcal{U}_{N}[q_n](P_n)-\mathcal{U}_{D}[q_n](P_n)\Big)\Big)\Big|^2\ \dx \dt \\
 &\quad+\int_0^T\int_{\Omega}q_{n}\Big| u_N[q_{n}]- u_D[q_{n}] - \beta_n  \Big( \mathcal{U}_{N}[q_n](P_n)-\mathcal{U}_{D}[q_n](P_n)\Big)\Big|^2 \dx \dt\\
 &\quad-\beta_n  \int_0^T\int_{\Omega}P_{n}\Big| u_N[q_{n}]- u_D[q_{n}] - \beta_n  \Big( \mathcal{U}_{N}[q_n](P_n)-\mathcal{U}_{D}[q_n](P_n)\Big)\Big|^2 \dx \dt\\
 &\quad+ \rho\int_{\Omega}\Big|q_n-\beta_n P_n\Big|^2\dx.
\end{align*}
As a result, the determination of the step size $\beta_n$ involves minimizing the functional
$$
\mathcal{K}_\rho (q_{n+1})=\mathcal{K}_\rho (q_n - \beta_n P_n).
$$
To achieve this, we set up the following condition to find $\beta_n$:
\begin{equation*}
    \frac{\partial \mathcal{K}_\rho }{\partial \beta_n}(q_n - \beta_n P_n)=0,
\end{equation*}
which leads to the quadratic equation (with respect to the size $\beta_n$),
\begin{equation}\label{beta-equation}
   -3A_n \beta_n^2 + 2B_n \beta_n -C_n =0. 
\end{equation}
Here, the coefficients $A_n,$ $B_n,$ and $C_n$ are defined as follows
\begin{align*}
A_n&= \int_0^T\int_\Omega P_{n} \Big| \mathcal{U}_{N}[q_n](P_n)-\mathcal{U}_{D}[q_n](P_n)\Big|^2 \dx\dt, \\
B_n&= \int_0^T\int_\Omega \Big|\nabla\Big( \mathcal{U}_{N}[q_n](P_n)-\mathcal{U}_{D}[q_n](P_n)\Big)\Big|^2 \dx\dt + \int_0^T\int_\Omega q_{n} \Big| \mathcal{U}_{N}[q_n](P_n)-\mathcal{U}_{D}[q_n](P_n)\Big|^2 \dx\dt\\
&\qquad+2  \int_0^T\int_\Omega P_{n} \Big(u_N[q_n]-u_D[q_n] \Big)\Big( \mathcal{U}_{N}[q_n](P_n)-\mathcal{U}_{D}[q_n](P_n)\Big) \dx\dt -\rho\int_{\Omega}|P_n|^2\dx,   \\
C_n&=  2\int_0^T\int_\Omega\nabla \Big(u_N[q_n]-u_D[q_n] \Big)\cdot\nabla \big( \mathcal{U}_{N}[q_n](P_n)-\mathcal{U}_{D}[q_n](P_n)\Big) \dx\dt\\
&\qquad\quad+2\int_0^T\int_\Omega P_{n} \Big(u_N[q_n]-u_D[q_n] \Big)\Big( \mathcal{U}_{N}[q_n](P_n)-\mathcal{U}_{D}[q_n](P_n)\Big) \dx\dt  \\&\qquad\qquad+\ds \int_0^T\int_\Omega P_{n} \Big|u_N[q_n]-u_D[q_n] \Big|^2\dx\dt -2\rho\int_{\Omega}P_n\;q_n\dx.
\end{align*}
The discriminant $\Delta_n $ of the quadratic equation \eqref{beta-equation} is given by
$
\Delta_n := 4(B_n^2 -3A_n\,C_n).
$
Hence,
\begin{itemize}
\item If $\Delta_n=0$, then 
$
\beta_{n}=\frac{-B_n}{2A_n}.
$

\item If $\Delta_n>0$, there are two candidate step sizes:
\begin{equation*}
 \beta_n^1=\frac{-B_n-\sqrt{\Delta_n}}{2A_n}\;\;\;\hbox{and}\;\;\beta_n^2=   \frac{-B_n+\sqrt{\Delta_n}}{2A_n}.
\end{equation*}
\begin{itemize}
\item[-] If $ \mathcal{K}_{\rho}(q_n - \beta_n^1 P_n)\leq  \mathcal{K}_{\rho}(q_n - \beta_n^2 P_n),$ then
$
\beta_{n}=\beta_n^1. 
$

\item[-] Otherwise,
$
\beta_{n}=\beta_n^2. 
$
\end{itemize}
\item If $\Delta_n < 0$, the quadratic equation \eqref{beta-equation} has no real solution, and a direct minimizer cannot be obtained. 
In this case, a standard approach is to employ a line-search technique, such as the Armijo condition \cite{armijo1966minimization}, 
which ensures a sufficient decrease of the functional. More precisely, in the Armijo condition, the step size $\beta_n$ is not determined from an explicit formula (as in the case $\Delta_n \geq 0$), 
but is chosen iteratively so that the sufficient decrease condition
$\mathcal{K}_\rho(q_n - \beta_n P_n) 
   \;\leq\; \mathcal{K}_\rho(q_n) 
   - \sigma \beta_n \,\langle \mathcal{K}_\rho^\prime(q_n), P_n \rangle_{L^2(\Omega)},$
is satisfied, where $0<\sigma<1$ is a small parameter (typically $\sigma \sim 10^{-4}$--$10^{-3}$). 
This procedure guarantees both a decrease in the functional and stability of the iterative scheme.
\end{itemize}

Based on the above discussion, we summarize in the following CGM algorithm the main steps of our reconstruction approach.

\begin{algorithm}[H]\label{algo1}
\caption{The CGM for solving the optimization problem \eqref{stable}}\label{alg:CG}
\KwIn{Given $q_{0},\;\rho,\; \varphi,\; \alpha ,\; tol>0 \text{ and } Max_{it}>0$.\\
Set $n=0$;
}
\While{$n\leq Max_{it}$}{
- Solve the Neumann boundary problem \eqref{PN} and the Dirichlet boundary problem \eqref{PD} to compute $u_N[q_{n}]$, $u_D[q_{n},\phi]$ and $\mathcal{K}_{\rho}(q_n)$\;

- Solve the adjoint problems \eqref{adjoint22} and \eqref{adjoint222} to compute the gradient $\mathcal{K}_{\rho}^\prime(q_n)$ from the equation \eqref{frechlet_K}\;
 
 - Compute the conjugate coefficient $\gamma_n$ to compute the
    direction of descent $P_n$\;

- Solve the sensitivity problems \eqref{P_d++} and \eqref{P_n++} by
    taking $\delta q_n(x) =P_n(x)$ to compute the sensitivity functions $\mathcal{U}_D[q_n](\delta q_n)$ and $\mathcal{U}_N[q_n](\delta q_n)$, respectively\;

- Compute the step length $\beta_n$\;
- Update the potential $q(x)$ \;
- Compute $q_{n+1}(x)$ by \eqref{qn}\; 
\If{
$\displaystyle
\frac{\Vert q_{n+1}-q_{n}\Vert_{L^2(\Omega)}}{
\Vert q_{n}\Vert_{L^2(\Omega)}}
\leq tol$
}{stop\;}
- Set $n=n+1$\;
}
\Return{$q_{n}$}.
\end{algorithm}

\section{Numerical Verification}\label{sec:numerical-result}

This section presents numerical results to illustrate the effectiveness of the proposed method in both one-dimensional and two-dimensional contexts. Synthetic data $\phi$ are generated by specifying all parameters of the direct problem \eqref{direct-problem}, including the potential $q^*$ and boundary data $\varphi$. The solution $u$ of the direct problem \eqref{direct-problem} is then computed using the standard Finite Element Method in space and the Finite Difference Method in time, following the procedure described in \cite{DiegoMurio2008}. Measurements are subsequently extracted on the boundary $\partial\Omega$. To evaluate the numerical stability of the algorithm, noise is added to the data using the formula:
\begin{equation*}
    \phi^{\epsilon}=\phi+\epsilon \cdot (2 \cdot \text{rand}(\text{size}(\phi)-1)),
\end{equation*}
where $rand$ represents a random perturbation and $\epsilon$ denotes the noise level.

To evaluate the precision of the numerical solution, we calculate the $L^2$ errors:
\begin{equation}\label{e-nn}
 e_{n}=\frac{\Vert q^* -q_{n}\Vert_{L^2(\Omega)}}{\Vert q^* \Vert_{L^2(\Omega)}},   
\end{equation}
where $q_{n}$ is the reconstructed potential at the $n$th iteration, and $q^*$ is the exact solution (actual potential). To ensure high-quality parameter estimates, we terminate the conjugate gradient algorithm when the relative change in the potential function falls below $10^{-7}$:
\[
\frac{\Vert q_{n+1}-q_{n}\Vert_{L^2(\Omega)}}{\Vert q_{n}\Vert_{L^2(\Omega)}} \leq 10^{-7},
\]
where $q_{n}$ denotes the potential parameter at iteration $n$.

\subsection{One-dimensional case}

In the numerical implementation, we set the final time $T=1.0$ and the spatial domain $\Omega=(0,2)$. In particular, the grid size for the time and space variables in the five examples is fixed as $\Delta t=\displaystyle\frac{1}{71}$ and $\Delta x=\displaystyle\frac{1}{90}$, respectively. Next, we define five examples to showcase the versatility of our approach.

\begin{exm}\label{example0}
For the first example, we consider the following smooth unknown function:
\[ q^*(x)=x. \]
\end{exm}
\begin{exm}\label{example1}
For the second example, we examine the subsequent smooth unknown function:
\[ q^*(x)=e^{-2x}\cos(2\pi x). \]
\end{exm}

\begin{exm}\label{example2}
For the third example, we choose the smooth unknown function:
\[ q^*(x)=\pi^2\sin(\pi x). \]
\end{exm}

\begin{exm}\label{example3}
For the four examples, we introduce the following nonsmooth function:
\[ 
q^*(x)=\left\lbrace 
\begin{array}{ l l l l}
x & \text{if }  x\in [0,1[,\\
2-x & \text{if }  x\in [1,2] .
\end{array}
\right. 
\]
\end{exm}

\begin{exm}\label{example4}
Lastly, for the fourth example, we define the nonsmooth function as:
\[ 
q^*(x)=\left\lbrace 
\begin{array}{ l l l l}
1 & \text{if } x\in [0.45,1.5[\\
2 &  \text{Otherwise}.
\end{array}
\right. 
\]
\end{exm}

\begin{remark}
In the numerical reconstructions, the regularization parameter was chosen primarily through empirical trial-and-error; see Section \ref{sec:reg-term-para} for further details.
\end{remark}
\subsubsection{ \bf Reconstruction results without noise} 
We first present the results of reconstructing the potential $q^*$ for the aforementioned examples without noise (i.e. $\epsilon=0$), with a Neumann data $\varphi(x,t)=t^2\sin(\pi x).$  In addition, we utilize an initial guess $q_0=1$ to initiate the reconstruction process. In the beginning, we present some tests for different orders of fractional derivatives $\alpha$.
\begin{figure}[H]
\centering
\subfigure[$\alpha=0.1$]
{\includegraphics[scale=0.4]{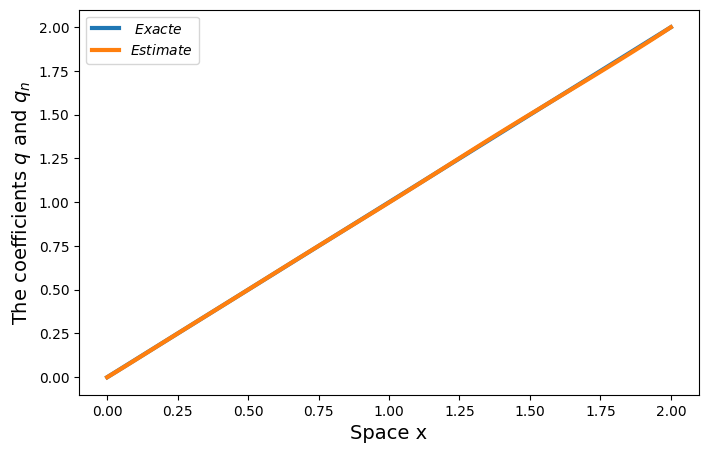}}
\subfigure[$\alpha=0.25$]
{\includegraphics[scale=0.4]{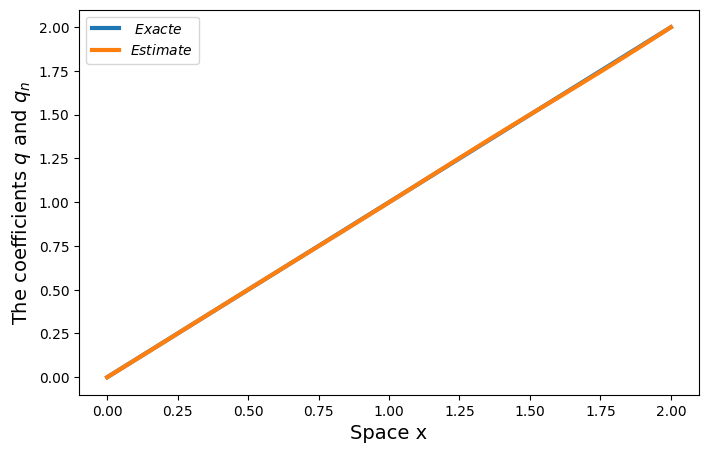}}
\subfigure[$\alpha=0.45$]
{\includegraphics[scale=0.4]{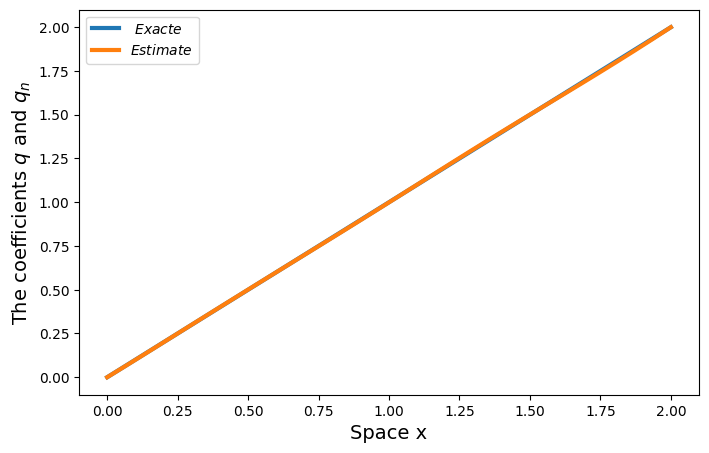}}
\subfigure[$\alpha=0.7$]
{\includegraphics[scale=0.4]{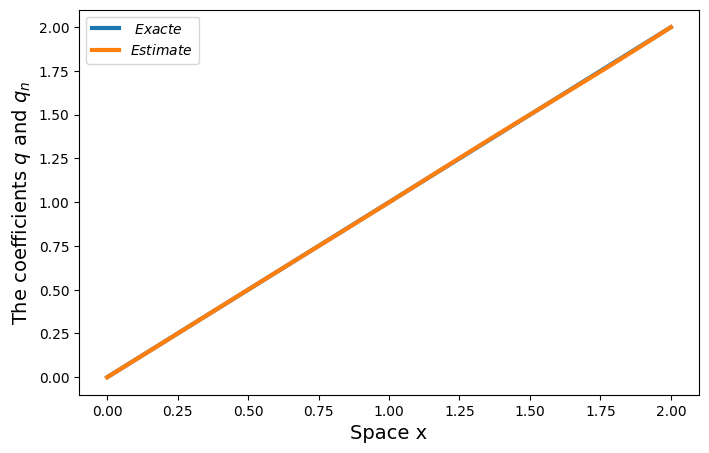}}

\caption{The results obtained with different  fractional orders for Example \ref{example0}.}
\label{fig1}
\end{figure}

Figure \ref{fig1} illustrates the impact of the fractional derivative order $\alpha$ on the reconstruction of the potential $q^*(x)=x$. The figure displays several curves corresponding to both the exact and reconstructed solutions for different values of $\alpha$. This comparative visualization highlights how variations in $\alpha$ influence the approximation quality. The results show that the proposed algorithm is able to reconstruct the target potential $q^*$ with high accuracy across all tested values of $\alpha \in \{0.1,\,0.25,\,0.45,\,0.7\}$, thereby confirming its robustness with respect to changes in the fractional order.

Continuing the numerical experiments, we fix $\alpha=0.45$. For the smooth Examples \ref{example1} and \ref{example2}, the reconstructed potentials obtained via the conjugate gradient algorithm are shown in Figure \ref{fig:aexample_12_sans}. In contrast, Figure \ref{fig:aexample_34_sans} presents a comparison between the exact and reconstructed potentials for the nonsmooth Examples \ref{example3} and \ref{example4}.

The results clearly indicate that the proposed method achieves accurate reconstructions in the smooth cases. However, in the nonsmooth settings, the algorithm encounters difficulties in capturing singularities with high fidelity. This limitation stems from the use of $L^2$ regularization in the optimization problem, which inherently favors smooth solutions and may oversmooth sharp transitions. 

\begin{figure}[H]
\centering
\subfigure[Example \ref{example1} without noise]
{
\begin{tikzpicture}[spy using outlines={ ,magnification=4,size=2cm, connect spies}]
        \node {\includegraphics[scale=0.4]{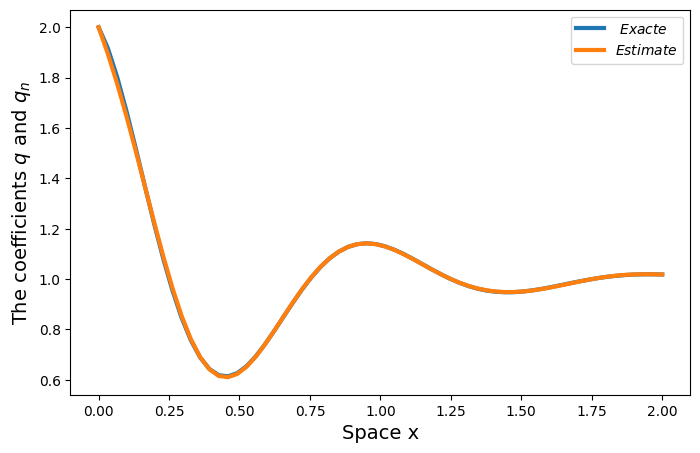}};
        \spy on (-1.4,-1.3) in node [left] at (1.4,1.2);
      \end{tikzpicture}}
\subfigure[Example \ref{example2} without noise]
{
\begin{tikzpicture}[spy using outlines={ ,magnification=4,size=2cm, connect spies}]
        \node {\includegraphics[scale=0.4]{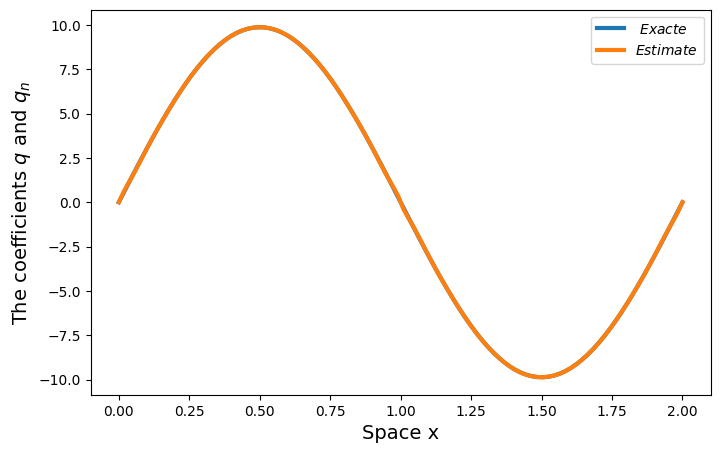}};
        \spy on (-0.7,1.8) in node [left] at (-0.3,-0.7);
      \end{tikzpicture}
}
\caption{Numerical results of the conjugate gradient algorithm for regular Examples.}
\label{fig:aexample_12_sans}
\end{figure}

\begin{figure}[H]
\centering
\subfigure[Example \ref{example3} without noise]
{\begin{tikzpicture}[spy using outlines={ ,magnification=3,size=2cm, connect spies}]
        \node {\includegraphics[scale=0.4]{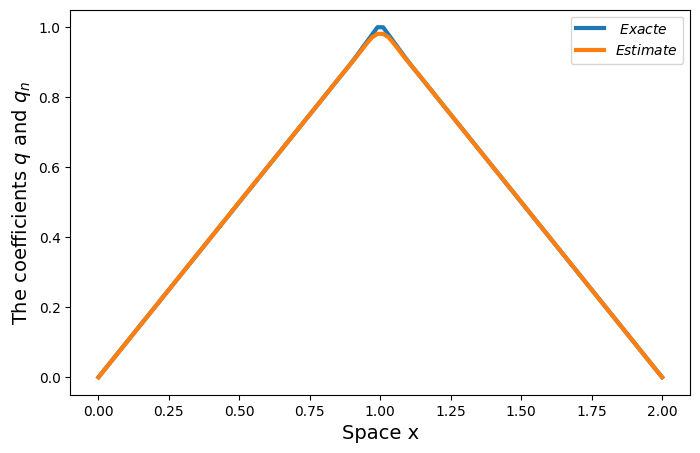}};
        \spy on (0.33,1.8) in node [left] at (1.4,-0.7);
      \end{tikzpicture}}
\subfigure[Example \ref{example4} without noise]
{
\begin{tikzpicture}[spy using outlines={ ,magnification=3,size=2cm, connect spies}]
        \node {\includegraphics[scale=0.4]{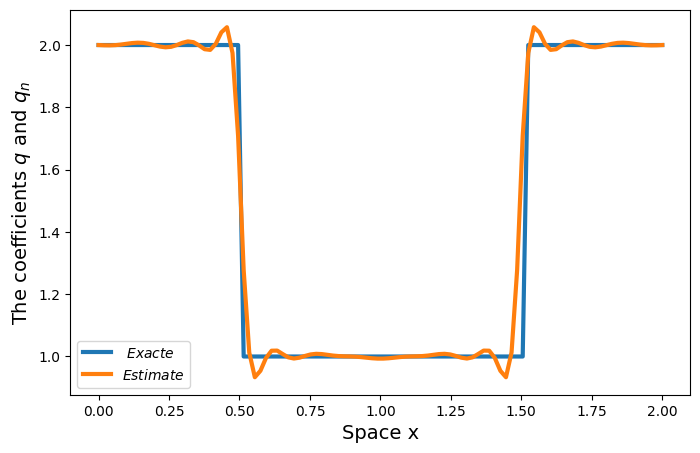}};
        \spy on (-1,-1.3) in node [left] at (1.2,1.2);
      \end{tikzpicture}
}
\caption{Numerical results of the conjugate gradient algorithm for complex Examples.}
\label{fig:aexample_34_sans}
\end{figure}

\subsubsection{\bf Effect of the regularization parameter}\label{sec:reg-term-para}
Finding an appropriate value for the regularization parameter $\rho$ is a delicate yet crucial task, since it directly impacts the accuracy and stability of the reconstruction. If $\rho$ is chosen too small, the solution may overfit the noisy data and become unstable, whereas excessively large values of $\rho$ may oversmooth the solution, leading to a significant loss of important features. A standard practical approach to address this difficulty is to examine the reconstruction error for a range of $\rho$ values and identify the one that yields the best balance between stability and accuracy. In the literature, various systematic strategies have been proposed for parameter selection. For instance, in \cite{oulmelk2022optimal}, a bi-level optimization framework was introduced, which provides an efficient and theoretically justified method for tuning the regularization parameter in regularized control problems. Meanwhile, \cite{engl2015regularization} proposed the so-called Morozov discrepancy principle, where the parameter is chosen by enforcing a discrepancy measure between the data and the model prediction to match the noise level. In contrast, in our study, the parameter $\rho$ was chosen empirically through a trial-and-error process guided by numerical experiments. Figure \ref{Figure-regularization} illustrates this selection procedure for Example \ref{example1}, where the target potential is $q^*(x)=e^{-2x}\cos(2\pi x)$. Reconstructions obtained with different values of $\rho$ are compared. The results clearly indicate that $\rho = 10^{-4}$ provides the most accurate approximation, offering a good compromise between suppressing noise and preserving the essential features of the exact solution.

\begin{remark}
Notably, our numerical experiments indicate that stable and accurate reconstructions are achieved even when the condition $\rho > M$ (where $M$ is the $\rho$-independent constant defined in \eqref{constant_M}) is not strictly satisfied. This observation is intriguing, as Theorem \ref{new-uniqueness-condition} imposes this constraint as a crucial hypothesis for ensuring the stability of the optimization problem \eqref{stable}. The empirical robustness of the method beyond its theoretical guarantees suggests the possibility of a sharper stability analysis in future work.
\end{remark}

\begin{figure}[H]
\centering
 \includegraphics[scale=0.75]{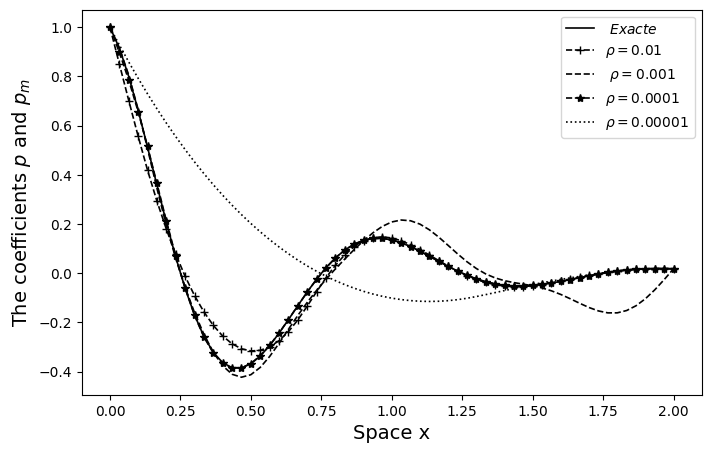}
\caption{Numerical solutions for Example \ref{example1} with varying values of the regularization parameter $\rho$.}
\label{Figure-regularization}
\end{figure}

\subsubsection{\bf Reconstruction results with noise}

In this section, we assess the robustness of the proposed reconstruction method when applied to noisy data. To this end, we replace the exact measured data $\phi$ with its noisy counterpart $\phi^\epsilon$, where $\epsilon$ denotes the noise level. This allows us to evaluate how the algorithm performs under realistic conditions where measurements are inevitably contaminated by errors. Figures \ref{fig:aexample_12_with} and \ref{fig:aexample_34_with} display the reconstruction results obtained for different noise levels in the smooth cases (Examples \ref{example1} and \ref{example2}) and in the nonsmooth cases (Examples \ref{example3} and \ref{example4}), respectively. For each example, reconstructions are compared against the ground truth, clearly illustrating the effect of noise on the solution.  

To further ensure these visual results, Table \ref{table:erreorv_various} reports the relative reconstruction errors across the tested noise levels. The quantitative results confirm that the proposed method exhibits a remarkable degree of stability: although the accuracy naturally degrades as the noise level increases, the reconstructions remain qualitatively reliable and preserve the main features of the exact solution. This demonstrates that the algorithm is robust with respect to measurement perturbations and can provide meaningful reconstructions even in the presence of significant noise. 

\begin{figure}[H]
\centering
\subfigure[Example \ref{example1} with noise]
{
\begin{tikzpicture}[spy using outlines={ ,magnification=4,size=2cm, connect spies}]
        \node {\includegraphics[scale=0.4]{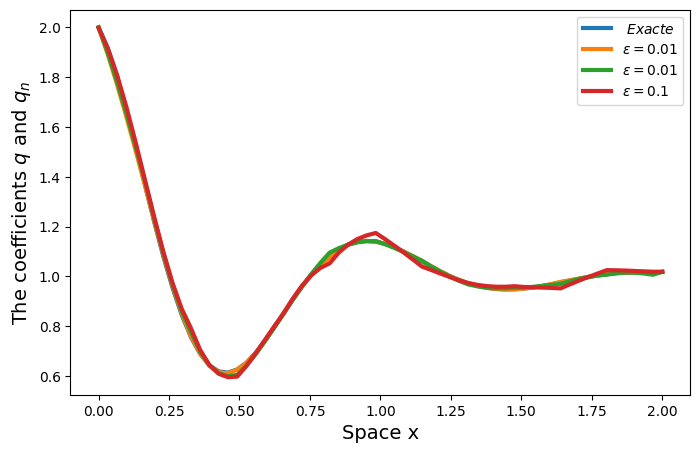}};
        \spy on (0,-0.3) in node [left] at (0.2,1.2);
      \end{tikzpicture}
}
\subfigure[Example \ref{example2} with noise]
{
\begin{tikzpicture}[spy using outlines={ ,magnification=4,size=2cm, connect spies}]
        \node {\includegraphics[scale=0.4]{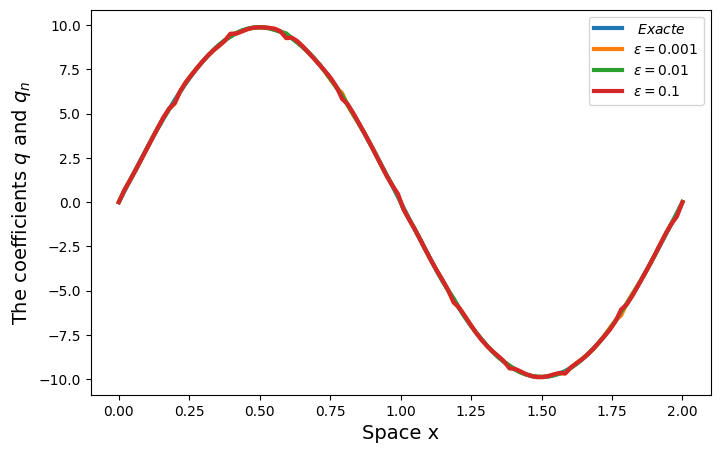}};
        \spy on (-0.7,1.8) in node [left] at (-0.35,-0.7);
      \end{tikzpicture}
}
\caption{Numerical results of the conjugate gradient algorithm for regular Examples  with noise.}
\label{fig:aexample_12_with}
\end{figure}

\begin{figure}[H]
\centering
\subfigure[Example \ref{example3} with noise]
{
\begin{tikzpicture}[spy using outlines={ ,magnification=4,size=2cm, connect spies}]
        \node {\includegraphics[scale=0.4]{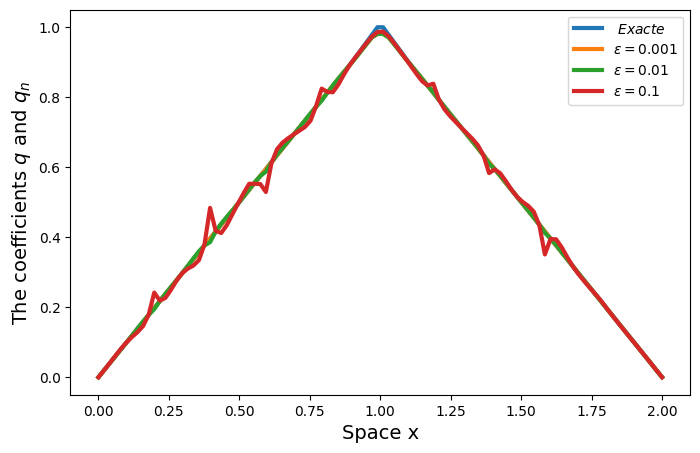}};
        \spy on (-0.18,1.4) in node [left] at (1.4,-0.7);
      \end{tikzpicture}
}
\subfigure[Example \ref{example4} with noise]
{
\begin{tikzpicture}[spy using outlines={ ,magnification=4,size=2cm, connect spies}]
        \node {\includegraphics[scale=0.4]{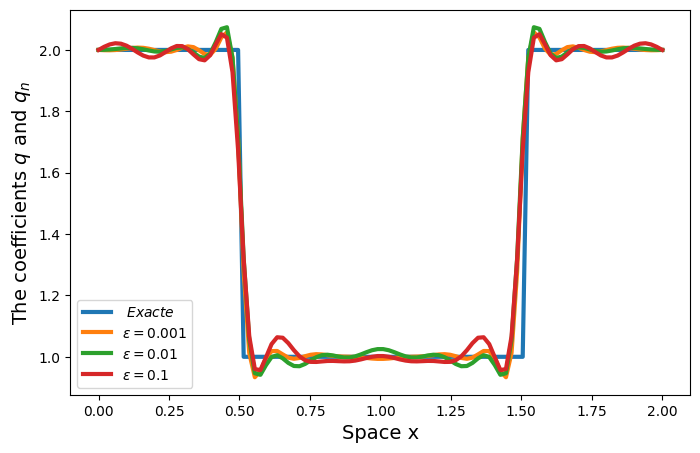}};
        \spy on (-0.45,-1.3) in node [left] at (1.2,1.2);
      \end{tikzpicture}
}
\caption{Numerical results of the conjugate gradient algorithm for complex Examples  with noise.}
\label{fig:aexample_34_with}
\end{figure}

\begin{table}[H]
\centering
\caption{Numerical results for the relative error $e_{n}$ as a function of $\epsilon$ for each example.}
\begin{tabular}{cccccccccc}
\toprule
Examples &&& $\epsilon=0\%$ &&& $\epsilon=1\%$ &&& $\epsilon=5\%$ \\
\midrule
Example \ref{example1} &&& $6.31\times 10^{-3}$ &&& $9.46\times 10^{-3}$ &&& $3.09\times 10^{-2}$ \\
\midrule
Example \ref{example2} &&& $7.28\times 10^{-3}$ &&& $9.52\times 10^{-3}$ &&& $4.01\times 10^{-2}$ \\
\midrule
Example \ref{example3} &&& $3.01\times 10^{-2}$ &&& $5.14\times 10^{-2}$ &&& $6.32\times 10^{-1}$ \\
\midrule
Example \ref{example4} &&& $5.43\times 10^{-2}$ &&& $6.12\times 10^{-2}$ &&& $8.37\times 10^{-2}$ \\
\bottomrule
\end{tabular}
\label{table:erreorv_various}
\end{table}

\subsubsection{\bf Relative error}

In this subsection, we analyze the convergence behavior of the proposed iterative method by tracking the relative error $e_n$ (see \eqref{e-nn}) as the number of iterations increases. The relative error provides a quantitative measure of the deviation between the reconstructed and exact solutions. Studying its evolution across iterations allows us to assess both the convergence rate and the accuracy of the algorithm.

Figures \ref{Figure5} depicts the evolution of the relative error $e_n$ for Examples \ref{example1} and \ref{example3}, corresponding respectively to a smooth and a nonsmooth potential. The results clearly show an exponential-type convergence, with the relative error rapidly approaching zero as the iterations proceed. This monotonic decrease highlights the effectiveness of the method in progressively refining the approximation of the true potential $q^*$. In particular, the observed convergence trend confirms the robustness and reliability of the algorithm, ensuring accurate recovery of $q^*$ even in challenging scenarios.

\begin{figure}[H]
\centering
\subfigure[Example \ref{example1}]{
  \includegraphics[scale=0.4]{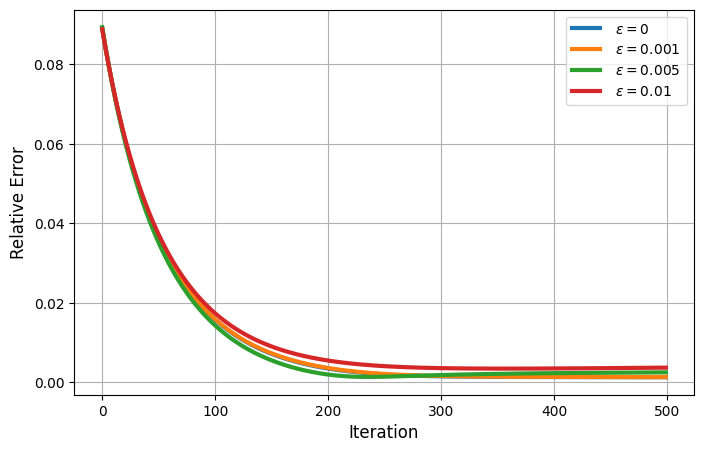}
  \label{fig:ex1}}
\subfigure[Example \ref{example3}]{
  \includegraphics[scale=0.4]{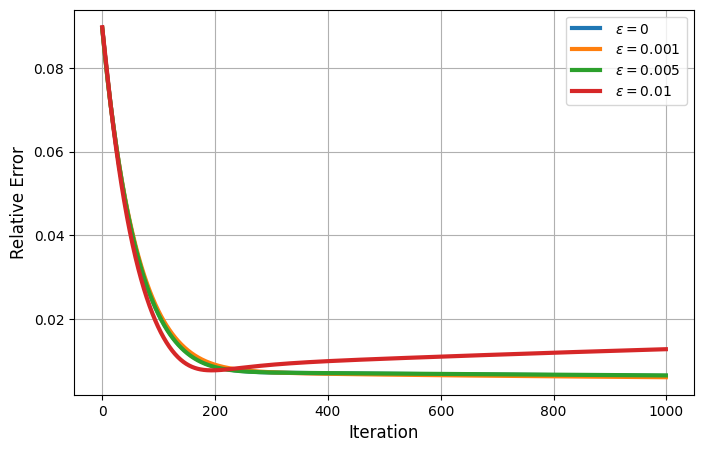}
  \label{fig:ex3}}
\caption{Numerical results of relative error $e_n$ for Examples \ref{example1} and \ref{example3}.}
\label{Figure5}
\end{figure}

\subsubsection{\bf Comparison with the classical least-squares method}
Here, we present a numerical comparison between the regularized Kohn--Vogelius method and the classical least-squares approach. 
To this end, the unknown potential $q^* \in \Phi_{ad}$ is characterized as the solution to the following minimization problem
\begin{equation*}
\underset{q\in \Phi_{ad}}{\operatorname{Minimize}}\; \mathcal{J}(q) \quad \text{subject to } \eqref{PNLS},
\end{equation*}
where $\mathcal{J}$ represents the least-squares cost functional, defined for each trial potential $q \in \Phi_{ad}$ by
\begin{equation*}
\mathcal{J}(q) := \int_0^T \int_{\partial\Omega} \big|u[q]-\phi\big|^2 \,\dss\dt+\mu\|q\|_{L^2(\Omega)}^2,
\end{equation*}
with $\mu$ is a regularization parameter and $u[q]$ being the solution of the subdiffusion
\begin{equation} \label{PNLS}
\left\{
  \begin{array}{rcll}
    \partial_{t}^{\alpha} u[q] - \Delta u[q] + q\, u[q] &=& 0 & \text{in } Q_T, \\[0.3em]
    \partial_\nu u[q] &=& \varphi & \text{on } \Sigma_T, \\[0.3em]
    u[q](\cdot,0) &=& 0 & \text{in } \Omega.
  \end{array}
\right.
\end{equation}
Following similar techniques as in Theorem \ref{frechlet_K_theorem}, $\mathcal{J}$ is Fréchet differentiable with gradient:
\begin{equation*}
    \mathcal{J}^\prime(q)=\int_0^T u[q]\;\zeta[q]\;\dt +2\mu \;q,
\end{equation*}
where $\zeta[q]$ satisfies the following adjoint problem:
\begin{eqnarray*} 
 \left\{
   \begin{array}{rclcl}
      {}_tD^{\alpha}_T \zeta[q]-\Delta  \zeta[q] +q\, \zeta[q] &=&0& \hbox{in} & Q_T, \\
 \partial_\nu\zeta[q]&=&-2\big(u[q] -\phi\big) &\hbox{on} & \Sigma_T,\\
 I^{1-\alpha}_{T}\zeta[q]&=&0 &\hbox{in} & \Omega\times\{T\}.
   \end{array}
 \right.
 \end{eqnarray*}

In order to implement Algorithm~\ref{algo1}, we need to determine the step size $\beta_{n}$. By applying an analysis similar to that in Section~\ref{sec:Iterative-algo}, one obtains
\begin{align*} 
    \beta_{n} = \frac{\displaystyle\int_0^T \int_{\partial\Omega} \Big(
 u [q_{n}]-\phi \Big) \mathcal{U}_N[q_n](P_n)\ \mathrm{d}s \dt + \mu \int_{\Omega} P_n\ q_n\  \dx}{\Big\| \mathcal{U}_N[q_n](P_n) \Big\|^2_{L^2(0,T;L^2(\partial\Omega))}  +\mu \Big\| P_n \Big\|^2_{L^2(\Omega)}},\;\;\;\; n=1,2,3,\cdots.
\end{align*}

\begin{figure}[H]
\centering
\subfigure[Example \ref{example1}]{
 \begin{tikzpicture}[spy using outlines={ ,magnification=4,size=2cm, connect spies}]
        \node {\includegraphics[scale=0.4]{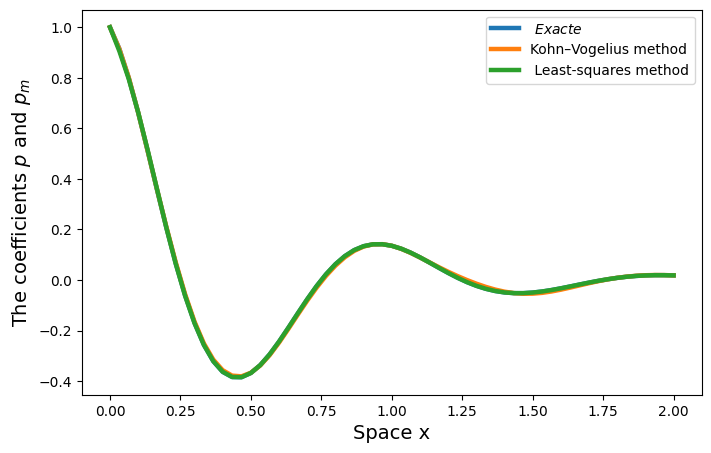}};
        \spy on (0.9,-0.5) in node [left] at (0.2,1.2);
      \end{tikzpicture}
  }
\subfigure[Example \ref{example3}]{
  \begin{tikzpicture}[spy using outlines={ ,magnification=4,size=2cm, connect spies}]
        \node {\includegraphics[scale=0.4]{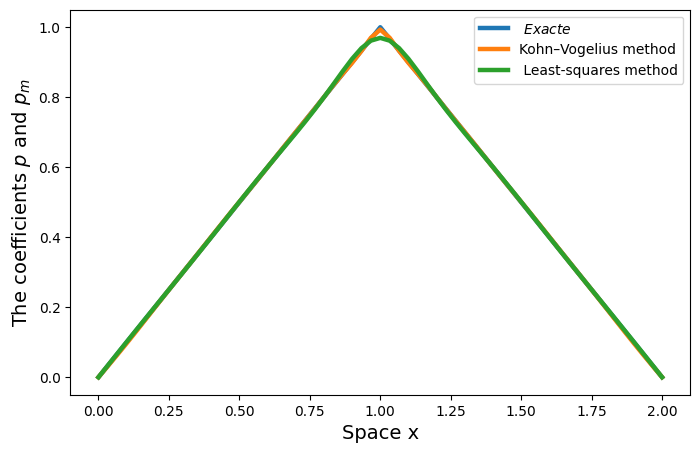}};
        \spy on (0.3,1.82) in node [left] at (1.4,-0.7);
      \end{tikzpicture}
  \label{fig:ex77}}
\caption{Numerical results  from the comparison with the classical least-squares approach  for Examples \ref{example1} and \ref{example3} without noise.}
\label{Figure77}
\end{figure}

 In Figure~\ref{Figure77}, we compare (without noise) the regularized Kohn–Vogelius method with the classical least-squares approach for Examples~\ref{example1} and \ref{example3}. In the smooth case, both methods exhibit nearly identical performances, with the Kohn–Vogelius method providing slightly more accurate approximations to the exact solution. In contrast, for the discontinuous cases illustrated in Figure \ref{fig:ex77}, the regularized Kohn–Vogelius method demonstrates a clear advantage, delivering significantly improved reconstructions compared to the least-squares approach. 
\begin{figure}[H]
\centering
\subfigure[Example \ref{example1}]{
 \begin{tikzpicture}[spy using outlines={ ,magnification=4,size=2cm, connect spies}]
        \node {\includegraphics[scale=0.4]{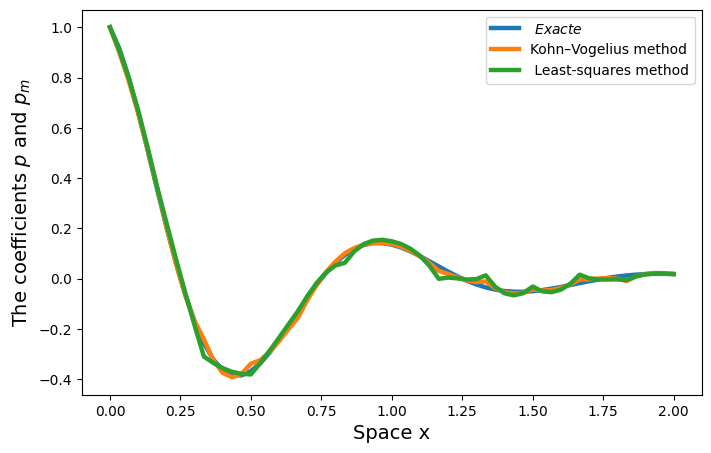}};
        \spy on (0.9,-0.5) in node [left] at (0.2,1.2);
      \end{tikzpicture}
  \label{fig:ex88}}
\subfigure[Example \ref{example3}]{
  \begin{tikzpicture}[spy using outlines={ ,magnification=4,size=2cm, connect spies}]
        \node {\includegraphics[scale=0.4]{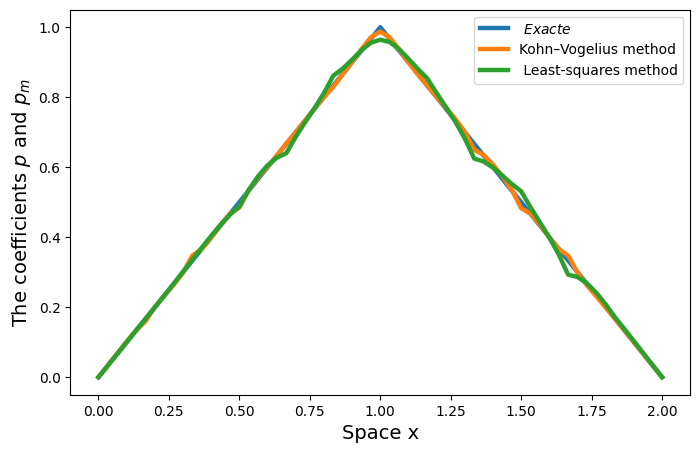}};
        \spy on (-0.18,1.4)in node [left] at (1.4,-0.7);
      \end{tikzpicture}
 }
\caption{Numerical results  from the comparison with the classical least-squares approach  for Examples \ref{example1} and \ref{example3} with noise.}
\label{Figure88}
\end{figure}

Figure \ref{Figure88} presents the numerical experiments with noise level $5\%$, indicating that both methods exhibit overall stability with respect to noise. However, the results clearly show that the Kohn–Vogelius method demonstrates greater robustness compared to the classical least-squares approach. While neither method is able to fully capture the singularity of the parameter, the Kohn–Vogelius method provides a closer approximation than least-squares.

Figure \ref{Figure55} compares the convergence of both methods by plotting relative error against iteration count. For the smooth potential in Example \ref{example1}, both approaches display a similar decreasing trend, confirming their stability and ability to progressively reduce reconstruction error. For the non-smooth potential in Example \ref{example3}, the Kohn–Vogelius method consistently achieves lower relative errors, demonstrating superior accuracy and efficiency throughout the reconstruction process.\\

Overall, the numerical results highlight the advantages of the regularized Kohn–Vogelius method, particularly in handling discontinuities and noisy data. In all tested scenarios, it outperforms the classical least-squares approach in terms of both accuracy and robustness.

\begin{figure}[H]
\centering
\subfigure[Example \ref{example1}]{
  \includegraphics[scale=0.4]{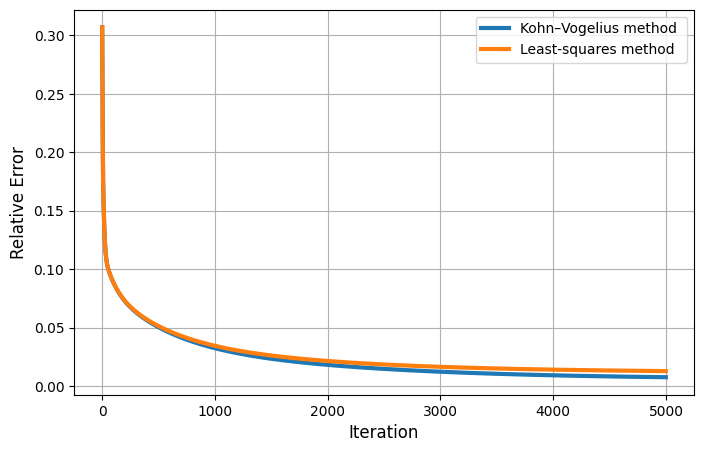}
  \label{figer:ex1}}
\subfigure[Example \ref{example3}]{
  \includegraphics[scale=0.4]{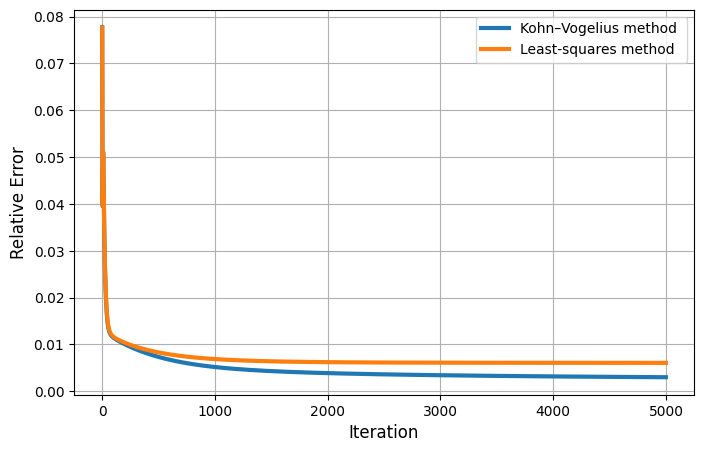}
  \label{figer:ex3}}
\caption{Convergence comparison: Kohn–Vogelius vs. least-squares methods.}
\label{Figure55}
\end{figure}

\subsection{\bf Two-dimensional case}

Let the coordinates be $(x,y)$. The computational domain is set as $\Omega = (0,1) \times (0,1)$, with fractional order $\alpha = 0.5$ and final time $T = 1$. The discretization parameters are chosen as $\Delta t = \displaystyle\frac{1}{51}$ for the temporal step and $(\Delta x, \Delta y) = \left(\displaystyle\frac{1}{70}, \frac{1}{70}\right)$ for the spatial steps. The Neumann boundary data are specified as
$$
\varphi(x,y;t) = t^\alpha \sin(\pi x) + \sin(\pi y).
$$
We evaluate the performance of the reconstruction algorithm when the potential $q^*$ is a subdifferential function.  
\begin{itemize}
   \item \textbf{Potential with smooth support:} The target potential $q^*$ is supported in a disk of radius $K = 0.03$ centered at $c = \left(\frac{1}{2}, \frac{1}{2}\right)$, i.e.,
   $$
   \Big(x - \tfrac{1}{2}\Big)^2 + \Big(y - \tfrac{1}{2}\Big)^2 \leq K.
   $$
   This smooth, localized structure allows us to test the algorithm's ability to recover continuous variations accurately.

The numerical reconstructions are shown in Figure \ref{fig12}, while the corresponding relative errors are plotted in Figure \ref{fig13}. From these results, it is evident that the algorithm successfully captures the shape and amplitude of the smooth potential, providing high-fidelity reconstructions. 

 \begin{figure}[H]
\center
        \subfigure[The exact potential $q^*$]{\includegraphics[height=5cm]{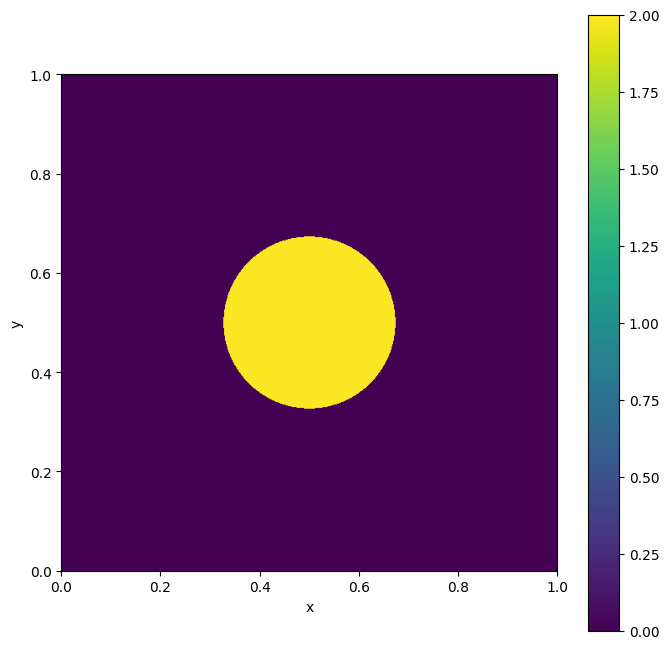}}
        \subfigure[ Approximate {$q^*$}  by the proposed method]{\includegraphics[height=5cm]{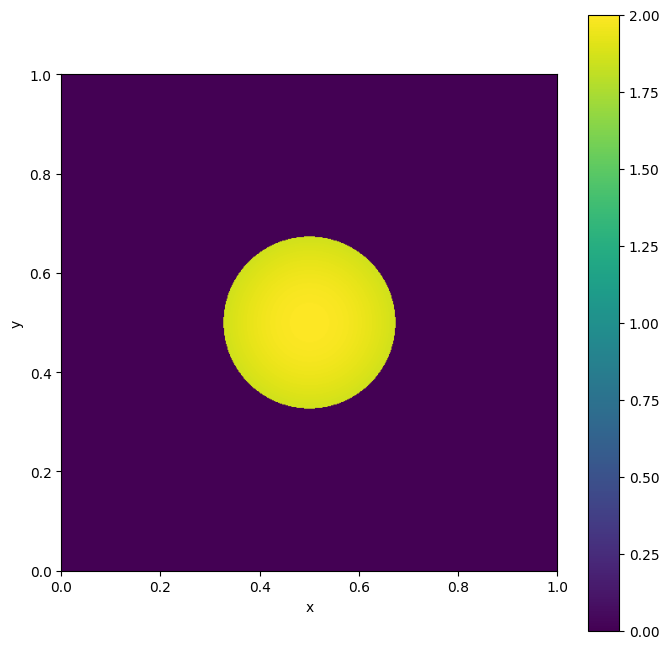}}
\caption{Numerical simulation of the computed solution $q^*$ using the proposed approach for 2D disc function.}\label{fig12}
\end{figure}

\begin{figure}[H]
\center
        \subfigure{\includegraphics[height=5cm]{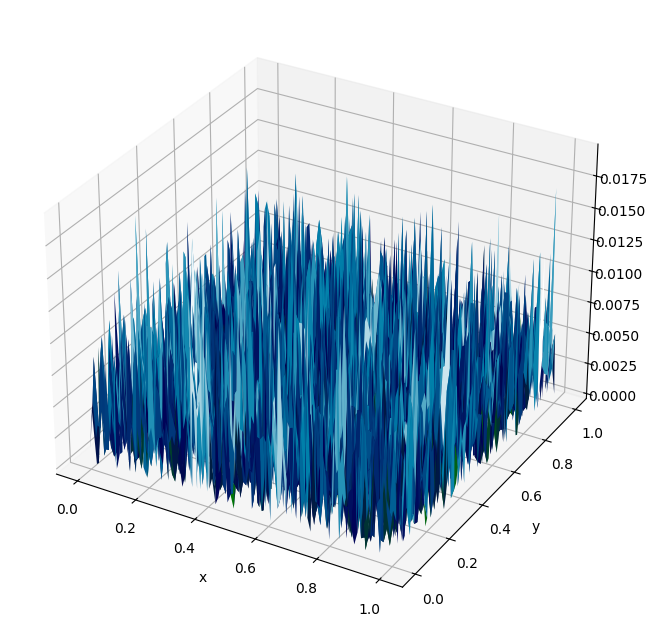}}
\caption{The relative error for the 2D potential $q^*$ using the proposed method}\label{fig13}
\end{figure}
   \item \textbf{Potential with non-smooth support:} The target potential $q^*$ is defined as a diamond-shaped function centered at $c = \left(\tfrac{1}{2}, \tfrac{1}{2}\right)$:
    \[
    \Big|x - \tfrac{1}{2}\Big| + \Big|y - \tfrac{1}{2}\Big| < K, \quad K = 0.3.
    \]
    This discontinuous shape tests the robustness of the method for reconstructing non-smooth parameters. The recovered potential and the corresponding relative error are shown in Figures \ref{fig10} and \ref{fig11}, respectively.

\begin{figure}[H]
\center
        \subfigure[The exact potential $q^*$]{\includegraphics[height=5cm]{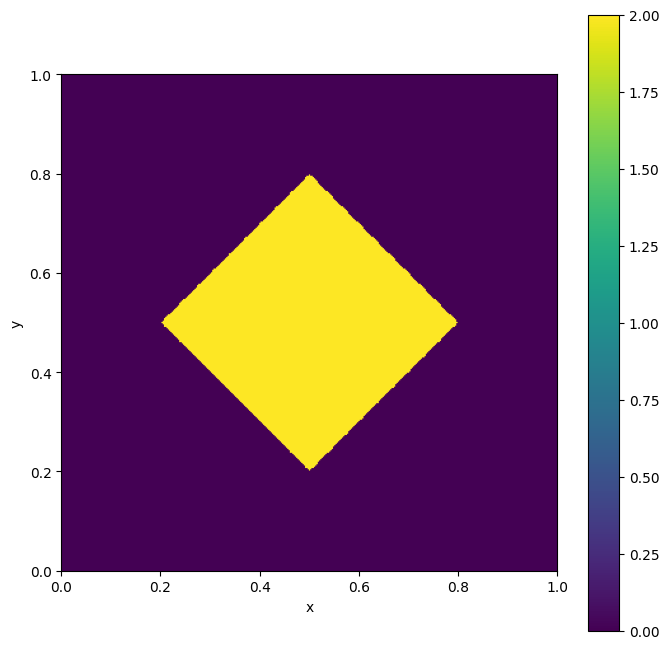}}
        \subfigure[ Approximate {$q^*$}  by the proposed method]{\includegraphics[height=5cm]{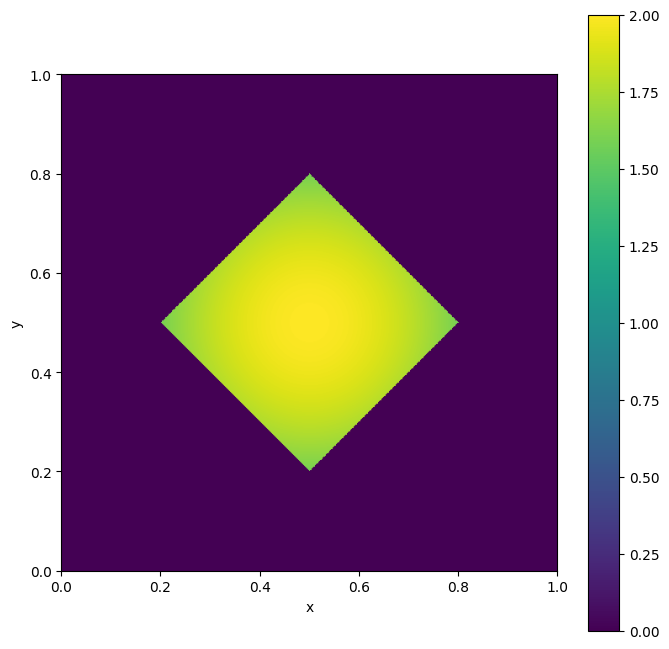}}
\caption{Numerical simulation of the computed solution $q^*$ using the proposed approach for 2D disc function.}\label{fig10}
\end{figure}
\begin{figure}[H]
\center
 \subfigure{\includegraphics[height=5cm]{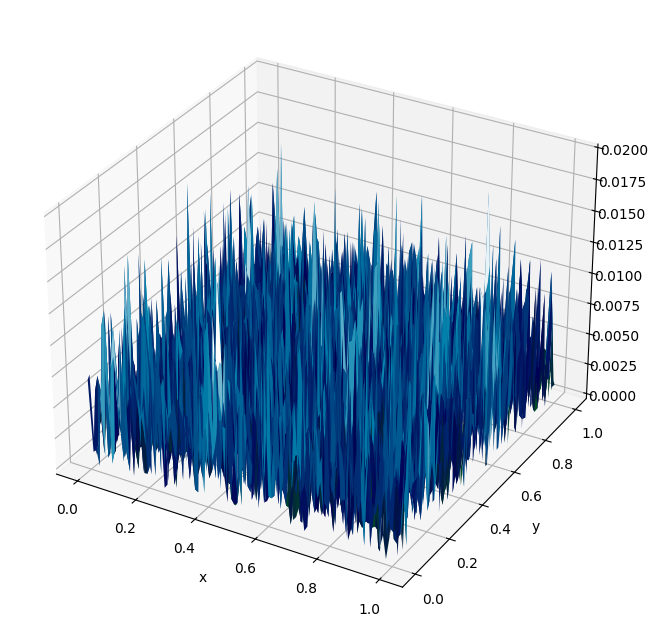}}
\caption{The relative error for the 2D potential $q^*$ using the proposed method.}\label{fig11}
\end{figure}
\end{itemize}

In both scenarios, the algorithm demonstrates high accuracy and stability, confirming its effectiveness for two-dimensional inverse potential problems, regardless of whether the target parameter $q^*$ is supported in smooth or non-smooth domain.

\section{Further comments and conclusions}\label{sec:9}

This paper has presented a comprehensive theoretical and numerical framework for solving the ill-posed inverse problem of reconstructing a space-dependent potential in subdiffusion from boundary measurements. The method is grounded in a Kohn-Vogelius  optimization approach, regularized with a Tikhonov term, which reformulates the inverse problem by minimizing a cost functional that measures the energetic discrepancy between the solutions of Dirichlet and Neumann boundary value problems.

Our main contributions are threefold:
\begin{itemize}
    \item \textbf{Theoretical analysis:} We established the well-posedness of the resulting optimization problem, proving the existence of a unique minimizer within a set of admissible potentials and demonstrating its stability with respect to perturbations in the measurement data. This provides a solid mathematical foundation for the approach.

    \item \textbf{Algorithm development:} We derived the Fr\'echet derivative of the Kohn-Vogelius cost functional and proved its Lipschitz continuity, enabling the design of an efficient conjugate gradient algorithm with a rigorous convergence guarantee. A distinctive numerical feature of our approach is the effective handling of the non-trivial quadratic equation for step size selection. It should be noted that the convergence analysis for the fully discrete problem, employing finite element approximations, remains a subject for future investigation.

\item \textbf{Numerical validation:} The efficacy and robustness of the proposed algorithm were confirmed through extensive numerical experiments in both one- and two-dimensional cases. The method performs well in reconstructing both smooth and discontinuous potentials, even in the presence of substantial noise. These results demonstrate the stability of the approach and its ability to produce reliable reconstructions across diverse scenarios.  
\end{itemize}


For future work, advanced optimization strategies could be employed, such as the topological derivative method \cite{NovotnyBook2013}, which offers a powerful framework for precise reconstruction of subdifferential functions. Furthermore, integrating total variation regularization \cite{chambolle1997image} would be highly beneficial for handling discontinuous or piecewise constant solutions, a common scenario in practical applications. Implementing these improvements would not only enhance reconstruction accuracy for non-smooth problems but would also greatly expand the scope and utility of the proposed framework.
Besides, future works could also focus on extensive numerical experiments in higher dimensions (e.g. three dimensions) and complex geometries (for example, when the support of the potential has complex geometries), a deeper analysis of stability under various noise models, and a comparison with the recently proposed Coupled Complex Boundary Method \cite{cheng2014novel}, which transforms the inverse problem into a challenging boundary problem, where a complex Robin condition links Dirichlet and Neumann data.

\vspace{1cm}
\paragraph{\bf{Conflict of interest}} The authors report no conflict of interest.

\bibliographystyle{plain}
\bibliography{simple}

\end{document}